\documentclass[10pt]{article}

\usepackage[colorlinks,linkcolor=blue,citecolor=blue]{hyperref}
\usepackage{amsthm}
\usepackage{multirow}
\usepackage{marginnote}
\usepackage{a4wide}
\usepackage{amssymb}
\usepackage{amsfonts}
\usepackage{amsmath}
\usepackage{mathrsfs}
\usepackage{tikz}
\usetikzlibrary{arrows,matrix}
\usetikzlibrary{positioning}
\usepackage{mdframed} % draw fram of texts
\usepackage{lipsum} % this is used for drawing frams of texts
\usepackage{extarrows} % making long equal sign (or arrows) with text under or above
\usepackage{color}
\usepackage{enumitem} 
\usepackage{tocloft} % change the style of table of contents
\usepackage{titlesec} % control the space before and after section titles

\usepackage[all]{xy} % draw diagrams

\input xy
\xyoption{arrow} \xyoption{matrix}

\date{}

\newtheorem{proposition}{Proposition}[section]
\newtheorem{theorem}[proposition]{Theorem}
\newtheorem{lemma}[proposition]{Lemma}

\newtheorem{definition}[proposition]{Definition}
\newtheorem{corollary}[proposition]{Corollary}

\def\der{\partial }

\def\nFM0{{\nu }_{F,M_0}}
\def\nFN0{{\nu }_{F,N_0}}
\def\nGN0{{\nu }_{G,N_0}}

\def\N0{ {\bf N}_0 }

\def\t{\otimes}
\def\g{\gamma}

\def\ra{\rightarrow}

\def\Xpm{X^{\pm }}

\def\s{\sigma}
\def\Z{\mathbb{Z}}

\def\l1{{\lambda}_1}

\def\a{\alpha}
\def\a0{ {\alpha }_0}
\def\a1{ {\alpha }_1}

\def\l{\lambda}

\def\nFGM0{{\nu }_{F,G,M_0}}

\def\nFN0{{\nu}_{F,N_0}}

\def\sm{{\sigma}^m}

\def\sm1{{\sigma}^{-1}}

\def\smtp1{{\sigma}^{-t+1}}

\def\S1{S^{-1}}

\def\Xpm1{X^{\pm 1}_1}

\def\sPM1{{\sigma }^{\pm 1}}
\def\sMP1{{\sigma }^{\mp 1 }}

\def\d{\delta}

\def\di{{\rm d.ind}}

\def\L{\Lambda}

\def\CA{{\cal A}}

\def\Ytm1{Y^{t-1}}
\def\Yim1{Y^{i-1}}

\def\CK{{\cal K}}

\def\CN{{\cal N}}

\def\CG{{\cal G}}

\def\ass{{\rm ass}}

\def\Aut{{\rm Aut}}

\def\Der{{\rm Der }}

\def\gcd{ {\rm gcd } }

\def\SL2Z{ {\rm SL}_2({\bf Z}) }

\def\Gp1{ G^{1 , 1 } }
\def\P11{ P^{-1 , 1 } }
\def\Pp1{ P^{1 , 1 } }

\def\lcm{{\rm lcm }}

\def\nCLsr{{}^\nu\kern-2pt {\cal L}^{\sigma , \rho  }}
\def\nP{{}^\nu \kern-2pt P}
\def\nL{{}^\nu\kern-2pt L}
\def\nLL{{}^\nu\kern-2pt \Lambda}
\def\nPsr{{}^\nu\kern-2pt P^{\sigma , \rho  }}
\def\nLsr{{}^\nu\kern-2pt L^{\sigma , \rho  }}
\def\nuCL{{}^\nu\kern-2pt  {\cal L}}
\def\nCLsr{{}^\nu\kern-2pt {\cal L}^{\sigma , \rho  }}
\def\nCL1m{{}^\nu\kern-2pt {\cal L}^{-1 , 1  }}
\def\x1nu{x^\frac{1}{\nu}}
\def\xm1nu{x^{-\frac{1}{\nu}}}

\def\ob{\overline{b}}

\def\CN{{\cal N}}
\def\ra{\rightarrow }

\def\CB{{\cal B}}
\def\lcm{{\rm lcm}}

\def\CC{ {\cal C}}

\def\nAM0{{\nu }_{{\cal A},M_0}}
\def\nAN0{{\nu }_{{\cal A},N_0}}

\def\End{ {\rm End }}
\def\Der{ {\rm Der }}

\def\bI{\overline{I}}

\def\bR{\overline{R}}

\def\bq{\overline{q}}

\def\ga{\mathfrak{a}}
\def\gb{\mathfrak{b}}

\def\SL{{\rm SL}}

\def\di!{\frac{\der^i}{i!}}
\def\dik!{\frac{\der^k_i}{k!}}
\def\id{{\rm id}}

\def\N{\mathbb{N}}

\def\0{\overline{0}}
\def\1{\overline{1}}

\def\Ln1{\L_{n,\overline{1}}}

\def\oa{\overline{a}}

\def\a1{a_{\overline{1}}}

\def\S{\Sigma}

\def\CU{{\cal U}}

\def\vn1{\overrightarrow{n-1}}

\def\im{{\rm im}}

\def\sl{{\rm sl}}

\def\soc{{\rm soc}}
\def\Sub{{\rm Sub}}

\def\bu{\overline{u}}

\def\mJ{\mathbb{J}}
\def\mI{\mathbb{I}}

\def\ann{{\rm ann}}
\def\lann{{\rm l.ann}}

\def\bM{\overline{M}}

\def\K1{{\rm K}_1}

\def\mP{\mathbb{P}}

\def\hmI1{\widehat{\mI_1}}
\def\tmI1{\widetilde{\mI_1}}
\def\tmJ1{\widetilde{\mJ_1}}
\def\hB1{\widehat{B_1}}
\def\hCB1{\widehat{\CB_1}}

\def\bS{\overline{S}}

\def\Den{{\rm Den}}

\def\Ore{{\rm Ore}}

\def\Den{{\rm Den}}

\def\br{\overline{r}}
\def\bc{\overline{c}}

\def\ga{\mathfrak{a}}

\def\tor{{\rm tor}}
\def\udim{{\rm udim}}

\def \S{\mathcal{S}}

\def\sl2{\mathfrak{sl}_2}

\def\bv{\overline{v}}

\def\EM{{\rm EM}}

\makeatletter
\newenvironment{proof*}[1][\proofname]{\par
  \pushQED{\qed}%
  \normalfont \partopsep=\z@skip \topsep=\z@skip
  \trivlist
  \item[\hskip\labelsep
        \itshape
    #1\@addpunct{.}]\ignorespaces
}{%
  \popQED\endtrivlist\@endpefalse
}
\makeatother

\begin{document}

\author{V. V. \  Bavula}
%%%% [Euclidean-Mods.tex]]  
\title{Euclidean modules, submodules  of their direct sums and strong reduced echelon form}

\maketitle

\begin{abstract}

The aim of the paper is to study Euclidean  and l-finite Euclidean  modules over a ring, submodules of their finite  direct sums, reduced and  strong reduced echelon forms of matrices with coefficients in  Euclidean modules and bimodules. The concept of Euclidean module is a module-theoretic analogue of the concept of Euclidean ring. The class of (not necessarily commutative)  Euclidean rings is a nice but `very  small' class of rings that have arithmetic-like properties. On the contrary, surprisingly, 
 the class of Euclidean modules is a `large class' of modules.  The  Euclidean modules exist for many {\em constructions} of rings that are defined over  {\em arbitrary} rings. For example, for every  skew polynomial ring $A=D[x;\s, \d]$ with coefficients in an {\em arbitrary} ring $D$, the generalized Weyl algebra $D[x,y;\s, a]$, the Weyl algebra $A_1=K[x][\der; \frac{d}{dx}]$, $U(\sl2)$ and many classical algebras,  there exist Euclidean modules.

$\noindent$

 \\


{ \small \tableofcontents}
\end{abstract}

%%%%%%%%%%%%%%%%%% SECTION 1 %%%%%%%%%%%%%%%%%%%%%

\section{Introduction} \label{INTR} %\marginpar{INTR}

 In this paper, rings are unital,  module means a left module and  $[n]:=\{ 1, \ldots , n\}$ for each natural number $n$.

Euclidean rings, not necessarily commutative, provide one of the few settings in noncommutative algebra where division algorithms and explicit reduction procedures remain available. Their Euclidean structure ensures effective control over left and right ideals and supports constructive methods that are typically inaccessible in general rings. Torsion-free modules over such rings inherit this computational clarity: they admit explicit bases, canonical embeddings and reduction techniques that parallel linear algebra over fields while retaining genuinely noncommutative features. Central to this framework is the reduced echelon form for matrices over left Euclidean rings, which extends Gaussian elimination to the noncommutative context and yields canonical representatives for homomorphisms and submodules of torsion-free modules. The interaction between Euclidean rings, torsion‑free modules and reduced echelon forms produces a coherent and algorithmically rich theory that combines structural transparency with effective computability, offering a natural setting for explicit module-theoretic and homological constructions.

Euclidean rings provide a rare combination of noncommutative generality and algorithmic strength. Their modules - especially torsion‑free modules - inherit this structure, allowing explicit computations, canonical forms and constructive proofs that are unavailable in general rings. The reduced echelon form serves as the central computational tool linking ring-theoretic and module-theoretic properties.

Their Euclidean structure guarantees that left and right ideals are principal, that unimodular rows can be completed and that matrix reduction techniques extend to noncommutative arithmetic. Modules over Euclidean rings inherit this constructive behaviour: torsion-free modules admit explicit bases, canonical embeddings into free modules and algorithmic descriptions of submodules and homomorphisms. A central computational tool is the reduced echelon form for matrices over left Euclidean rings, which generalises Gaussian elimination and provides canonical representatives for linear maps and submodules. Together, these results yield a coherent and computable module theory in which greatest common divisor-type reductions, ideal-theoretic structure  and noncommutative linear algebra interact seamlessly. The framework offers both conceptual clarity and algorithmic strength, enabling explicit classification, decomposition  and reduction procedures that are unavailable in general noncommutative rings.  More information about Euclidean rings and their modules can be found in the following books 
\cite{Anderson-Fuller1992, Cohn1971, Dauns1994, Golan-1974, Jacobson2009, Lam2001,  LambekLect1996, MR}.

 The main object of this paper are Euclidean modules and l-finite Euclidean modules.

 \begin{definition}
Let $R$ be a ring and $(\N^\diamond, \geq)$  be a partially ordered set (poset, for short). An $R$-module $M$ is called a {\bf Euclidean module} if there exists a function $d_M: M\backslash \{ 0\} \ra \N^\diamond$, which is called a {\bf Euclidean function} of $M$, such that the poset $(\im (d_M)), \geq)$ is an Artinian poset and the function $d_M$   satisfies the  {\bf left division property}: For any nonzero elements $u,v\in M$, there exist elements $q\in R$  and $r\in M$ (that are called the {\bf quotient} and the {\bf remainder}) such that 
$$
u=qv+r\;\; \text{where either}\;\;  r=0 \;\; {\rm or}\;\;d_M(r)<d_M(v).
$$
Let $\EM (R)$ be the  class of Euclidean $R$-modules. For a submodule $N$ of $M$, let 
$$
g(N):=\min \{ d_M(n)\, | \, 0\neq n\in N\}\;\; {\rm and}\;\; 
\CG (N):=\{n\in N \, | \, d_M(n)=g(N)\}.
$$
 The function $d_M$ is called an {\bf l-finite Euclidean function} if every interval in  the poset $(\im (d_M)), \geq)$ has finite length (equivalently, $l([g(M), d_M(m)])<\infty$ for all nonzero elements $m\in M$). Then the $R$-module $M$ is called an {\bf l-finite Euclidean module}.  Let $\EM (R, {\rm l-fin.})$ be the  class of l-finite Euclidean $R$-modules.
\end{definition} 
 In a similar way, the concept of  {\em Euclidean bimodule} is defined.  The aim of      the paper is to develop a {\em module-theoretic} theory of Euclidean modules analogous  to  the {\em ring-theoretic} theory of Euclidean rings: 
 to study Euclidean  and l-finite Euclidean  modules and bimodules  over a ring, submodules of their finite  direct sums, reduced and  strong reduced echelon forms of matrices with coefficients in  Euclidean modules and bimodules. The theory is similar but not identical to the classical one as we do not put any restrictions to the ring $R$.  
 The difficulties stem from the fact that the ring $R$ contains zero divisors or  elements of a Euclidean $R$-module can have non-zero annihilators.

The paper is organized as follow. We outline below the structure  and key results of the paper. 
The aim of Section \ref{EUCLID-MOD}  is to introduce and study module-theoretic analogues of the following ring-theoretic concepts:  Euclidean ring, division property and principal left ideal domain. The corresponding concepts are Euclidean and l-finite Euclidean modules, division property  and module with all submodules principal. The class of 
Euclidean rings (not necessarily commutative) form a large yet distinct class of rings that exhibit integer--like properties. Similarly, 
Euclidean modules also exhibit integer--like properties: every submodule of a Euclidean module is generated by  a single element and  for elements of a Euclidean module there are concepts of the {\em greatest common divisor} and the {\em least common multiple}.

All submodules and factor modules of Euclidean module are principal submodules (Proposition \ref{A28May26}.(1)). 
All submodules and factor modules of Euclidean module are also Euclidean modules  (Lemma \ref{a7Jun26}.(1,2)). All submodules and factor modules of l-finite Euclidean module are l-finite Euclidean modules  (Lemma \ref{a7Jun26}.(4)). Lemma \ref{a14Jun26} is a criterion for two submodules of a Euclidean module to be equal provided one contains the other. Every Euclidean module is a Noetherian module (Theorem \ref{14Jun26}.(1)). Every nonzero submodule of a Euclidean module is a  co-finite submodule (Theorem \ref{14Jun26}.(2)). Each l-finite Euclidean module is a uniform (hence indecomposable) module of Krull dimension 1 with zero socle and every nonzero submodule is of infinite length  (Lemma \ref{b14Jun26}). Proposition \ref{A14Jul26} and Proposition \ref{B14Jul26} are characterizations of the greatest common divisor and the least common multiple of a set of finitely many elements of a Euclidean module, respectively. Theorem \ref{15Jun26} is  a module-theoretic analogue of the equality $ab=\gcd (a,b)\lcm (a,b)$ for integers $a$ and $b$. 
 Lemma \ref{a8Jun26} shows that localizations of Euclidean modules have properties very similar to Euclidean modules. Theorem \ref{8Jun26} provides several sufficient conditions for extending   a Euclidean function of a Euclidean ring to its localization.

 In Section  \ref{GRAD-FILT}, the following resulrs are proven.  For an arbitrary  skew polynomial ring $A=D[x;\s, \d]$, 
Theorem \ref{Sig-Del-28May26} describes a class of Euclidean modules. Similarly, for an arbitrary  skew Laurent polynomial ring $L=D[x^{\pm 1};\s]$, 
Theorem \ref{SkewLaur-28May26} describes a class of Euclidean modules. 
 For an $\N$-graded ring, Theorem \ref{11Jun26} is a criterion for a $\N$-graded module to be a Euclidean module. Similarly, for a skew Laurent polynopmial ring,  Theorem \ref{12Jun26}  is a criterion for a $\Z$-graded module to be a Euclidean module. Proposition \ref{FiltEucl-28May26} shows that if the associated graded module is a Euclidean module over the associated graded ring then the original module is a Euclidean module.

In Section \ref{SUB-FACTOR-MOD}, we study  direct sums $M=\bigoplus_{i=1}^n M_i$ of Euclidean $R$-modules $M_i$ and their submodules. These modules are Noetherian $R$-modules. So, every  $R$-submodule $N$ of $M$ admits a finite set of generators $a_i=(a_{i1},\ldots  , a_{in})$, where $i=1,\ldots , m$ and $a_{ij}\in M_j$ and they yield an $m\times n$ matrix $A=(a_{ij})\in M_{m,n}(M_1,\ldots, M_n)$ with coefficients in the direct summands. Each matrix $A'$ in $M_{m,n}(M_1,\ldots, M_n)$ yields an $R$-submodule of $M$ which is generated by the rows of the matrix $A'$. On the rows of the matrix $A$ we can perform elementary row transformations (ER1) and (ER2) where for a fixed row the transformation  (ER1) adds a linear combination of the other rows to it and (ER2) swaps two rows. These operations are reversible and they change the set of generators of the submodule $N$. Two matrices $A, A'\in  M_{m,n}(M_1, \ldots , M_n)$ are called {\bf left-row equivalent}, $A\sim A'$, if one of them can be obtained from the other by finitely many elementary row operations (ER1) and (ER2). The relation $\sim$ is an equivalence relation.  If $A\sim A'$ then $N(A)=N(A')$. Theorem \ref{REForm} is about existence of left reduced echelon form for the matrix $A$.\\

{\bf The strong reduced echelon form.}  
Let $M=\bigoplus_{i=1}^nM_i$ be a direct sum of Euclidean $R$-modules $M_i$. The $R$-module $M$ admits a strictly descending of chain of $R$-submodules 
$$
M=M_{\geq 1}\supset \cdots  \supset M_{\geq i}\supset \cdots  \supset M_{\geq n}\supset M_{\geq n+1}:=\{ 0\}\;\; {\rm where}\;\;   M_{\geq i}:=M_i\oplus\cdots \oplus M_n.
$$
The descending chain defines the $R$-module filtration $\{ M_{\geq i}\}$ on $M$. The associated graded $R$-module ${\rm gr}(M):=\bigoplus_{i=1}^n {\rm gr}(M)_i$, where ${\rm gr}(M)_{\geq i}:=M_i/M_{\geq i+1}\simeq M_i$, is isomorphic to the $R$-module $M$. Let $N$ be an $R$-submodule of $M$. Then the $R$-modules $N$ and $\bM :=M/N$ admit the induced filtrations $ \{ N_{\geq i}:=N\cap M_{\geq i} \}$ and   $ \{ \bM_i:= (N +M_{\geq i})/N \}$, respectively. Then the associated graded $R$-module ${\rm gr}(N):=\bigoplus_{i=1}^n{\rm gr}(N)_i$ is a submodule of the $R$-module $M\simeq {\rm gr}(M)$ where ${\rm gr}(N)_i:=N_{\geq i}/N_{\geq i+1}$. Recall that  $[n]:=\{ 1, \ldots , n\}$ for each natural number $n\geq 1$.

\begin{definition}
The set $\mP_s (N):=\mP_s (N,M):=\{ i\in [n]\, | \, {\rm gr}(N)_i\neq 0 \}$ is called the {\bf strong pivot set} and the elements of the set $\mP_s (N)$ are called the set of {\bf strong pivots} of $N$.
\end{definition}

 In addition to the elementary row operations (ER1) and (ER2), on the rows $a_1, \ldots , a_m$ of the matrix $A$, we  perform a new  elementary row operation -- adding to the  zero row a linear combination of the  other rows:
\begin{eqnarray*}
{\rm (ER0)}\;\;&\;\;\;\;\; & 0\mapsto 0+\sum_{i=1}^m r_ia_i\;\; {\rm where}\;\; r_i\in R.
\end{eqnarray*}
On the level of generators, it means that we replace the  generating set $a_1, \ldots , a_m$ of the $R$-module $N$ by the generating set  $a_1, \ldots , a_m,\sum_{i=1}^mr_ia_i$.  
In the classical situation when the ring $R$ is a left Euclidean domain and all  $M_i=R$, this operation is redundant because nonzero elements of the ring $R$ act injectively on the modules $M_i=R$. In the general situation, this is not the case and this is the main reason that we have to add the operation (ER0) in order to obtain the strong reduced echelon form which is a more refined version of the reduced echelon form and it contains more information about the module as we will see later. The element $\sum_{i=1}^mr_ia_i$ can be deleted from the list of generators $a_1, \ldots , a_m,\sum_{i=1}^mr_ia_i$ by using the operation (ER1). 
 
Two matrices $A, A'\in  M_{m,n}(M_1, \ldots , M_n)$ and $A'\in  M_{m',n}(M_1, \ldots , M_n)$ are called {\bf strongly left-row equivalent}, $A\sim_s A'$, if one of them can be obtained from the other by finitely many elementary row operations (ER0), (ER1) and (ER2) (`$s$' stands for `strongly'). The relation $\sim_s$ is an equivalence relation.   If $A\sim_s A'$ then $N(A)=N(A')$ where $N(A)=N(a_1, \ldots , a_m)=\sum_{i=1}^mRa_i$ and $N(A')=N(a_1', \ldots , a_{m'}')=\sum_{i=1}^{m'}Ra_i'$.

\begin{definition}[{\bf Strong reduced echelon form}]
Let $M=\bigoplus_{i=1}^nM_i$ be a direct sum of Euclidean $R$-modules $M_i$, $0\neq A =(a_1, \ldots , a_s)^t\in M_{s,n}(M_1, \ldots , M_n)$, where $a_1, \ldots , a_s$ are the rows of the matrix $A$, $N(A)=\sum_{i=1}^sRa_i$  and $\mP_s(N(A))=\{ j_1, \ldots , j_s\}$ where $j_1< \cdots < j_s$. The  matrix $A$ is said to be in {\bf strong (left) reduced echelon form} if 
\begin{enumerate}

\item  for each $\nu = 1, \ldots, s$, $N_{\geq  j_\nu}=\sum_{\mu =\nu}^sRa_\mu$ and the element $a_{\nu j_\nu}\in M_{j_\nu}$ has the least degree in the set of nonzero elements of the $R$-module  ${\rm gr}(N)_{j_\nu}=Ra_{\nu j_\nu}$, and 

\item 
  the degrees of the  elements $a_{1j_\nu}, \ldots , a_{\nu-1, j_\nu}\in M_{j_\nu}$ are  strictly less that the degree of the element $a_{\nu j_\nu}\in M_{j_\nu}$.

\end{enumerate}
\end{definition}
It follows from the definition that  if the matrix $A$ is in strong reduced echelon form then it is also in reduced echelon form. 
 Theorem \ref{STR-REForm} is about existence of strong reduced echelon form. 

\begin{theorem}[{\bf Existence of strong reduced echelon form}] \label{STR-REForm}%\marginpar{STR-REForm}

Let $M=\bigoplus_{i=1}^nM_i$ be a direct sum of Euclidean $R$-modules $M_i$, $0\neq A \in M_{m,n}(M_1, \ldots ,M_n)$ and $\mP_s(N(A))=\{ j_1, \ldots , j_s\}$  where $j_1< \cdots < j_s$. Then the  matrix
$A$ is strongly left-row equivalent to a matrix in strong 
   reduced echelon form.
\end{theorem}
 Theorem \ref{26Jul26} is a criterion for two matrices in  strong reduced echelon form to be strongly equivalent. Theorem \ref{b26Jul26} describes the equivalence class of the matrix $A$. Proposition \ref{a30Jul26} provides criteria ensuring  that  the $R$-modules $N(A)$ and $M/N(A)$ have finite length and it gives explicit expressions for these lengths. Proposition \ref{REForm-N1} presents a sufficient condition  for a reduced echelon form of the matrix $A$ to be a strong reduced echelon form and it also computes the  uniform dimensions of the $R$-module $N(A)$ and $N(A)+M(C\mP_s(N(A))$.\\

{\bf Left-row-column equivalence $\approx $.} 
Let $E$ be a Euclidean $R$-module  and   $0\neq A=(a_{ij}) \in M_{m,n}(E, \ldots ,E)$.
 On the columns $c_1, \ldots , c_n$ of the matrix $A$, we can do two types of operations -- {\em elementary column  operations}:
\begin{eqnarray*}
 &{\rm (EC1)}& c_i\mapsto c_i+\sum_{j\neq i}r_jc_j\;\; {\rm where}\;\; r_j\in R,  \\
 &{\rm (EC2)}& c_i\mapsto c_j, \;\; c_j\mapsto c_i.
\end{eqnarray*}
For a fixed column, (EC1) adds a linear combination of the other columns to it, and (EC2) swaps two columns. 
 These operations are reversible.

Two matrices $A, A'\in  M_{m,n}(E,\ldots , E)$ are called {\bf left-column equivalent}, $A\sim_c A'$, if one of them can be obtained from the other by finitely many elementary column operations (EC1) and (EC2). The relation $\sim_c$ is an equivalence relation.  Two matrices $A, A'\in   M_{m,n}(E,\ldots , E)$ are called {\bf left-row-column equivalent}, $A\approx A'$, if one of them can be obtained from the other by finitely many elementary row and column operations: (ER1), (ER2), (EC1) and (EC2). The relation $\approx$ is an equivalence relation. For a matrix $A=(a_{ij})\in   M_{m,n}(E,\ldots , E)$, let $\gcd (A):=\gcd \{ a_{ij}\, | \, i=1,\ldots, m; j=1,\ldots , n\}$. Then $\CN_R(A):=\sum_{i=1}^m\sum_{j=1}^nRa_{ij}=R\gcd (A)$. 
Clearly, if $A\approx A'$ then $R\gcd (A)=R\gcd (A')$.

\begin{theorem}\label{REForm-N3}%\marginpar{REForm-N3}

Let $(E,d_E)$ be a  Euclidean $R$-module  and   $0\neq A \in M_{m,n}(E, \ldots ,E)$ where $(m,n)\neq (1,1)$. Then $A\approx B$ where
\[
B
=
\begin{pmatrix}
b_{11} & 0      & \cdots & 0      & 0      & \cdots & 0 \\
0      & b_{22} & \cdots & 0      & 0      & \cdots & 0 \\
\vdots & \vdots & \ddots & \vdots & \vdots &        & \vdots \\
0      & 0      & \cdots & b_{ss} & 0      & \cdots & 0 \\[4pt]
%\hline
0      & 0      & \cdots & 0      & 0      & \cdots & 0 \\
\vdots & \vdots &        & \vdots & \vdots &        & \vdots \\
0      & 0      & \cdots & 0      & 0      & \cdots & 0
\end{pmatrix}_{m\times n}
\]
where  $Rb_{11}\supset Rb_{22}\supset \cdots \supset Rb_{ss}\neq \{0\}$, $d_E(b_{11})<d_E(b_{22})< \cdots < d_E(b_{ss})$ and $b_{11}=\gcd (A)$. 
\end{theorem}

{\bf Euclidean bimodules.} An $R$-bimodule $E$ is called a {\bf cyclic $R$-bimodule} if $E=Re=eR$ for some element $e\in E$ which is called a {\bf cyclic generator} for $E$. An element $c\in E$ of an $R$-module $E$ is called a {\bf cyclic element} of $E$ if $Rc=cR$. Let $\CC (E)$ be the set of cyclic elements of $E$. 
\begin{definition}
An $R$-bimodule $(E,d_E)$ is called a {\bf Euclidean $R$-bimodule} if the left and right $R$-modules, $(_RE,d_E)$ and $(E_R,d_E)$, are Euclidean. 
\end{definition}
 Every sub-bimodule of the Euclidean $R$-bimodule $E$ is a cyclic $R$-bimodule  (Lemma \ref{a9Aug26}).
\begin{definition}
Let $E$ be a Euclidean $R$-bimodule. For  elements $a_1, \ldots , a_n\in E$, a cyclic  generator 
$\gcd_b (a_1, \ldots , a_n)$ of the cyclic  sub-bimodule $N_b(a_1, \ldots , a_n):=\sum_{i=1}^n Ra_iR$ of $E$   is called the {\bf bimodule greatest common divisor} of the elements  $a_1, \ldots , a_n$ where $Ra_iR$ is a sub-bimodule of $E$ generated by the element $a_i$. Similarly, a cyclic generator 
$\lcm_b (a_1, \ldots , a_n)$ of the cyclic  sub-bimodule $I_b(a_1, \ldots , a_n):=\bigcap_{i=1}^n Ra_iR$ of $E$   is called the {\bf bimodule least common multiple} of the elements  $a_1, \ldots , a_n$.
\end{definition}
Proposition \ref{bi-A14Jul26} and Proposition \ref{bi-B14Jul26} are characterizations of the bimodule greatest common divisor and the bimodule least common multiple, respectively.

Let $(E,d_E)$ be a Euclidean $R$-bimodule  and   $0\neq A=(a_{ij}) \in M_{m,n}(E, \ldots ,E)$.
 On the columns $c_1, \ldots , c_n$ of the matrix $A$, we can do two types of operations -- {\em elementary column  operations}:
\begin{eqnarray*}
 &{\rm (EC1)}& c_i\mapsto c_i+\sum_{j\neq i}c_jr_j\;\; {\rm where}\;\; r_j\in R,  \\
 &{\rm (EC2)}& c_i\mapsto c_j, \;\; c_j\mapsto c_i.
\end{eqnarray*}

Two matrices $A, A'\in  M_{m,n}(E,\ldots , E)$ are called {\bf right-column equivalent}, $A\sim_{rc} A'$, if one of them can be obtained from the other by finitely many elementary column operations (EC1) and (EC2). The relation $\sim_{rc}$ is an equivalence relation.  Two matrices $A, A'\in   M_{m,n}(E,\ldots , E)$ are called {\bf row-column equivalent}, $A\approx_b A'$ (where $b$ stands for `bimodule'), if one of them can be obtained from the other by finitely many elementary row and column operations: (ER1), (ER2), (EC1) and (EC2). The relation $\approx_b$ is an equivalence relation. For a matrix $A=(a_{ij})\in   M_{m,n}(E,\ldots , E)$, let $\gcd_b (A):=\gcd_b \{ a_{ij}\, | \, i=1,\ldots, m; j=1,\ldots , n\}$. Then $N_b(A):=\sum_{i=1}^m\sum_{j=1}^nRa_{ij}R=R\gcd_b (A)=\gcd_b (A)R$. 
Clearly, if $A\approx_b A'$ then $N_b(A)=N_b(A')$ and $R\gcd_b (A)=R\gcd_b (A')$. 

\begin{theorem}\label{bi-REForm-N3}%\marginpar{bi-REForm-N3}

Let $(E,d_E)$ be a  Euclidean $R$-bimodule  and   $A \in M_{m,n}(E, \ldots ,E)$ where $m,n\geq 2$. Then $A\approx_b B$ where
\[
B
=
\begin{pmatrix}
b_{11} & 0      & \cdots & 0      & 0      & \cdots & 0 \\
0      & b_{22} & \cdots & 0      & 0      & \cdots & 0 \\
\vdots & \vdots & \ddots & \vdots & \vdots &        & \vdots \\
0      & 0      & \cdots & b_{ss} & 0      & \cdots & 0 \\[4pt]
%\hline
0      & 0      & \cdots & 0      & 0      & \cdots & 0 \\
\vdots & \vdots &        & \vdots & \vdots &        & \vdots \\
0      & 0      & \cdots & 0      & 0      & \cdots & 0
\end{pmatrix}_{m\times n}
\]
where  $N_b(A)=Rb_{11}R\supseteq Rb_{22}R\supseteq \cdots \supseteq Rb_{ss}R\neq \{0\}$.

%*** ***
% $Rb_{11}\supset Rb_{22}\supset \cdots \supset Rb_{ss}\neq \{0\}$, $b_{11}R\supset b_{22}R\supset \cdots \supset b_{ss}R\neq \{0\}$, $d_E(b_{11})<d_E(b_{22})< \cdots < d_E(b_{ss})$ and $b_{11}=\gcd (A)=\gcd_r (A)$. In particular, $Rb_{11}R\supseteq Rb_{22}R\supseteq \cdots \supseteq Rb_{ss}R\neq \{0\}$, $Rb_{11}=\sum_{i=1}^m\sum_{j=1}^nRa_{ij}$ and $b_{11}R=\sum_{i=1}^m\sum_{j=1}^na_{ij}R$.
 
% **
\end{theorem}

%%%%%%%%%%%%%   Section 2    %%%%%%%%

\section{Euclidean modules, their propersties and localizations}\label{EUCLID-MOD} %\marginpar{EUCLID-MOD}

%*** UBRAT  ***

%All submodules and factor modules of Euclidean module are principal submodules (Proposition \ref{A28May26}.(1)). 
%All submodules and factor modules of Euclidean module are also Euclidean modules  (Lemma \ref{a7Jun26}.(1,2)). All submodules and factor modules of l-finite Euclidean module are l-finite Euclidean modules  (Lemma \ref{a7Jun26}.(4)). Lemma \ref{a14Jun26} is a criterion for two submodules of a Euclidean module to be equal provided one contains the other. Every Euclidean module is a Noetherian module (Theorem \ref{14Jun26}.(1)). Every nonzero submodule of a Euclidean module is a  co-finite submodule (Theorem \ref{14Jun26}.(2)). Each l-finite Euclidean module is a uniform (hence indecomposable) module of Krull dimension 1 with zero socle and every nonzero submodule is of infinite length  (Lemma \ref{b14Jun26}). Proposition \ref{A14Jul26} and Proposition \ref{B14Jul26} are characterizations of the greatest common divisor and the least common multiple of a set of finitely many elements of a Euclidean module, respectively. Theorem \ref{15Jun26} is  a module-theoretic analogue of the equality $ab=\gcd (a,b)\lcm (a,b)$ for integers $a$ and $b$. 
% Lemma \ref{a8Jun26} shows that localizations of Euclidean modules have properties very similar to Euclidean modules. Theorem \ref{8Jun26} provides several sufficient conditions for extending   a Euclidean function of a Euclidean ring to its localization. \\

%***

{\bf Euclidean rings.} 
Briefly,  {\em Euclidean domain}  is a domain endowed with a specific type of function that allows for a generalized division with remainder.

 A (not necessarily commutative)  domain $E$ is called  a {\bf left Euclidean domain} if there exists a function (often called a {\em Euclidean function, valuation, degree}, or {\em norm}) $d: E \setminus \{0\} \to \N$ such that it satisfies the  {\bf left division property}: For any elements $a,b\in E$ with $b\neq 0$, there exist elements $q,r\in E$ ({that are called the {\bf quotient} and the {\bf remainder}) such that 
 $$
 a=qb+r\;\; \text{where either}\;\;  r=0\;\; {\rm  or}\;\;d(r)<d(b).
 $$ 
 Similarly, a {\bf right Euclidean domain} is defined. A ring $F$ is called a {\bf principal left ideal domain} if every left ideal is a principal left ideal. Every left/right Euclidean domain is a principal left/right ideal domain but not vice versa, in general.
  
  The rings $E=\Z$ and $K[x]$ are left/right  Euclidean domains with Euclidean functions $d(z)=|z|$ ($z\in \Z$) and $d (f)=\deg (f)$, the degree of the polynomial $f\in K[x]$ where $K$ is a field, respectively. The Gaussian integers $\Z[i]=\Z \oplus \Z i$, where $i^2=-1$, are Euclidean domain with $d(a+bi)=a^2+b^2$. 
 
Let $D$ be a ring and  $\s\in \Aut (D)$. Then  an endomorphism $\d: D\ra D$  of the abelian group $(D,+)$ is called  a $\s$-{\em derivation} if $\d(ab)=\d (a)b+\s (a)\d(b)$ for all elements $a,b\in D$. A {\bf skew polynomial ring} $A=D[x;\s, \d]$ is a ring that is generated by the ring $D$ and an element $x$ subject to the defining relations:  $xa=\s (a)x+\d (a)$ for all elements $a\in D$. Then $A=\bigoplus_{i\in \N}Dx^i=\bigoplus_{i\in \N}x^iD$. 
 Important classes are $D[x;\s]:=D[x;\s, \d=0]$ and $D[x;\d]:=D[x;\s=\id_D, \d]$.
Each element $a\in A$, is a unique sum
$$a=\sum_{i=0}^na_ix^i=\sum_{i=0}^nx^i\s^{-i}(a_i)\;\; {\rm where}\;\;  a_i\in D.
$$
 Suppose that $a\neq 0$, i.e. there is a nonzero coefficient $a_i\neq 0$. Then the natural number $\deg (a):=\max\{i\in \N\, | \, a_i\neq 0\}$ is called the {\em degree} of the element $a$, the element $a_nx^n=x^n\s^{-n}(a_n)$, where $n=\deg (a)$, is called the {\em leading term} of $a$ and the elements $a_n$  and $\s^{-n}(a_n)$ are called the {\em left} and {\em right  leading coefficients} of $a$, respectively. By  definition, $\deg (0):=-\infty$.
  Then for all elements $a,b\in A$ and $d\in D$:
 \begin{eqnarray*}
 \deg (a+b)&\leq &\max \{ \deg (a), \deg (b)\},\\
 \deg (ab)&\leq &\deg (a)+\deg (b),\\
\deg (ab)&=&\deg (a)+\deg (b),\;\; {\rm if}\;\; \text{$D$ is a domain.}
 \end{eqnarray*}
Clearly, if $D$ is a division ring then the ring $A=D[x;\s, \d]$ is a left and right Euclidean ring since the function $\deg$ satisfies the left and right division property. Therefore, the ring $A=D[x;\s, \d]$ is a left and right principal ideal domain. In particular, the skew polynomial rings $D[x;\s]$ and  $D[x;\d]$ are left/right Euclidean domain and left and right principal ideal domains.

 The degree function $\deg$ determines the ascending {\em degree filtration} $\{ A_{\leq n} \}_{i\in \N}$ on $A$ where $A_{\leq n}:=\{ a\in A\, | \, \deg (a)\leq n\}$:
$$
A=\bigcup_{i\in \N} A_{\leq n}\;\; {\rm where}\;\;  A_{\leq n} A_{\leq m}\subseteq  A_{\leq n+m}\;\; \text{for all }\;\; n,m\in \N.
$$ 
Then the {\em associated graded ring} of $A$, 
$${\rm gr}(A):=\bigoplus_{n\in \N}A_{\leq n}/A_{\leq n-1}\simeq D[x; \s],
$$
is a skew polynomial ring.\\
 
{\bf Euclidean modules and their properties.} A partially ordered set $(\N^\diamond, \geq)$ is called {\bf Artinian} if every descending chain $i_1\geq i_2\geq \cdots$ stabilizes, that is there is a natural number $j\geq 1$ such that $i_1\geq i_2\geq \cdots \geq i_j=i_{j+1}=\cdots$.
 For each pair of elements of $\N^\diamond$, $i\geq j$, the set $[i,j]:=\{k\in \N^\diamond \, | \, i\leq k\leq j \}$ is a called an {\bf interval}. We say that the interval $[i,j]$ has finite length,  denoted by   $l([i,j])$, if 
 $l([i,j]):=\sup\{ k\in \N\, | \, i=i_1<i_2<\cdots <i_k=j\}<\infty$. 
A poset in which every interval has finitely many elements is called a {\bf locally finite poset}. The poset $(\N, \geq)$ is a locally finite poset and $l([i,j])=j-i+1$. A locally finite poset is a poset in which every interval has finite length but not vice versa. 
  Let $P:=\{1-\frac{1}{n}\, | \, n=1,2,\ldots\}\cup \{ 1\}$.   The poset $(P, \geq )$ is Artinian but  not  every interval has finite length since  $l([0,1])=\infty$. The set of nonzero ideals of the ring of integers $\Z$ is an Artinian  poset where $I\geq J$ means $I\subseteq J$ and $l([I,J])=l_\Z (J/J)<\infty$, the $\Z$-length of the $\Z$-module $J/I$. 
 
 \begin{definition}
Let $R$ be a ring and $(\N^\diamond, \geq)$  be a partially ordered set (poset, for short). An $R$-module $M$ is called a {\bf Euclidean module} if there exists a function $d_M: M\backslash \{ 0\} \ra \N^\diamond$, which is called a {\bf Euclidean function} of $M$, such that the poset $(\im (d_M)), \geq)$ is an Artinian poset and the function $d_M$   satisfies the  {\bf left division property}: For any nonzero elements $u,v\in M$, there exist elements $q\in R$  and $r\in M$ (that are called the {\bf quotient} and the {\bf remainder}) such that 
$$
u=qv+r\;\; \text{where either}\;\;  r=0 \;\; {\rm or}\;\;d_M(r)<d_M(v).
$$
Let $\EM (R)$ be the  class of Euclidean $R$-modules. For a submodule $N$ of $M$, let 
$$
g(N):=\min \{ d_M(n)\, | \, 0\neq n\in N\}\;\; {\rm and}\;\; 
\CG (N):=\{n\in N \, | \, d_M(n)=g(N)\}.
$$
 The function $d_M$ is called an {\bf l-finite Euclidean function} if every interval in  the poset $(\im (d_M)), \geq)$ has finite length (equivalently, $l([g(M), d_M(m)])<\infty$ for all nonzero elements $m\in M$). Then the $R$-module $M$ is called an {\bf l-finite Euclidean module}.  Let $\EM (R, {\rm l-fin.})$ be the  class of l-finite Euclidean $R$-modules.
\end{definition}

If the poset $(\N^\diamond, \geq)$ is an Artinian poset such that  every interval  has finite length then every sub-poset is an Artinian poset such that  every interval  has finite length. In particular, every function $\d : M\backslash \{ 0\}\ra \N^\diamond$ that satisfies the left division property is automatically an l-finite Euclidean function and the $R$-module $M$ is an l-finite Euclidean module. For example, the poset $(\N ,\geq )$ is  an Artinian poset such that  every interval  has finite length. 
 
\begin{definition}
Let $R$ be a ring and $(\N^\diamond, \geq)$  be a partially ordered set (poset, for short). The ring $R$ is called a {\bf left/right Euclidean ring} if the left/right $R$-module $R$ is a left/right Euclidean $R$-module. If, in addition,  the left/right $R$-module $R$ is l-finite then the ring $R$ is called an {\bf l-finite  left/right Euclidean ring}.
\end{definition}

%***************** DELETE ******

%\begin{definition}
%Let $R$ be a ring. An $R$-module $M$ is called a {\bf Euclidean module} if there exists a function $d_M: M\ra \N$, which is called a {\bf Euclidean function} of $M$, that satisfies the  {\bf left division property}: For any elements $u,v\in M$ with $v\neq 0$, there exist elements $q\in R$  and $r\in M$ (that are called the {\bf quotient} and the {\bf remainder}) such that 
%$$
%u=qv+r\;\; \text{where either}\;\;  r=0 \;\; {\rm or}\;\;d_M(r)<d_M(v).
%$$
%Let $\EM (R)$ be the  class of Euclidean $R$-modules. For a submodule $N$ of $M$, let $g(N):=\min \{ d_M(n)\, | \, 0\neq n\in N\}$ and $\CG (N):=\{n\in N \, | \, d_M(n)=g(N)\}$.
%$\end{definition}

%**************  END DELETE   ***********

An $R$-module $M$ is called a {\em cyclic} or  {\em principal}  or 1-{\em generated} module if $M=Rm$ for some element $m\in M$. 

\begin{definition}
Let $R$ be a ring. An $R$-module $M$ is called a {\bf module with all submodules principal} if every submodule of $M$ is a principal module. 
\end{definition}

A ring $R$ is called a {\bf left/right  Bezout} if every finitely generated left/right ideal of 
$R$  is principal.  A module is called a {\bf Bezout module} if every finitely generated submodule   is principal. 

\begin{proposition}\label{A28May26}%\marginpar{A28May26}
Let $R$ be a ring and $M$ be a  Euclidean module with Euclidean function $d_M$. Then:
\begin{enumerate}

\item The  module $M$ is a module with all submodules principal. In particular, the $R$-module $M$ is a cyclic Bezout  module.

\item If $N$ is a nonzero submodule of $M$ then the set of 1-generators of the $R$-module $N$ contains  the non-empty  set $\CG (N)$.
%$:=\{ n\in N\, | \, d_M(n)=l \}$ where $l:=\min \{ d_M(n')\, | \, 0\neq n'\in N \}$. 

\end{enumerate}
\end{proposition}

\begin{proof} 2. By the definition, $\CG (N)\neq \emptyset$. Fix an element $v\in \CG (N)$. Then for  any nonzero element $u\in N$, there exist elements $q\in R$  and $r\in M$  such that $u=qv+r$ where either $r=0$ or $d_M(r)<d_M(v)=g(N)$. Since $=u-qv\in N$, we must have $r=0$, by the minimality of $g(N)$. Therefore, $N=Rv$.

1. Statement 1 follows from statement 2.
\end{proof}

Lemma \ref{a7Jun26} shows that all submodules and factor modules of  Euclidean modules  are also Euclidean modules. 

\begin{lemma}\label{a7Jun26}%\marginpar{a7Jun26}
Let $(M, d_M)$ be a Euclidean $R$-module. Then:
\begin{enumerate}

\item Every submodule $N$ of $M$ is a Euclidean $R$-module $(N, d_N)$ where $d_N:=d_M |_N$.

\item Every factor module $M/N$ of $M$ is a  Euclidean $R$-module $(M/N,d_{M/N})$ where $ d_{M/N}(m+N):=\min\{d_M(m+n)\, | \, n\in N \}$.

\item All submodules and factor modules of a Euclidean module are module with all submodules principals.

\item All submodules and factor modules of an l-finite Euclidean modules are l-finite Euclidean modules.

\end{enumerate}
\end{lemma}

\begin{proof} 1. The proof is evident.

2. For each element $m\in M$, let $\overline{m}:=m+N\in M/N$. Clearly, $d_{M/N}(\overline{m})\leq d_M(m)$ for all elements $m\in M\backslash N$. Suppose that  $\bu$ and $\bv$ are nonzero elements of the module $M/N$. Since the poset $(\im (d_M),\geq )$ is Artinian,  we can fix an element $n\in N$ such that $$
d_M(v+n)=d_{M/N}(\overline{v}).
$$ 
Now, $u=q(v+n)+r$ for some elements $q\in R$ and $r\in M$ such that $d_M(r)<d_M(v+n)$. Therefore, $\overline{u} =q\overline{v}+\overline{r}$ where $\overline{r} \in M/N$ and  $d_{M/N}(\overline{r})\leq  d_M(r)<d_M(v+n)=d_{M/N}(\overline{v})$, and statement 2 follows. 

3. Statement 3 follows from statements 1 and 2 and Proposition \ref{A28May26}.(1).

4. Statement 4 follows from statements 1 and 2.
\end{proof}

Lemma \ref{a14Jun26} is a very useful and simple equality  criterion for  two submodules of a Euclidean module such that one contains the other. This an analogue of a `similar' result for finite dimensional vector subspaces such that one contains the other: They are equal iff their dimensions are equal. 

\begin{lemma}\label{a14Jun26}%\marginpar{a14Jun26}
Suppose that $M_1\subseteq M_2$ are submodules of a Euclidean module $M$. Then  $M_1=M_2$ iff $g(M_1)=g(M_2)$. 

%Let $(M, d_M)$ be a Euclidean $R$-module and $M_1\subseteq M_2$ be submodules of $M$. Then  $M_1=M_2$ iff $g(M_1)=g(M_2)$. 
\end{lemma}

\begin{proof} $(\Rightarrow)$ Clear.

 $(\Leftarrow)$ Suppose that $g(M_1)=g(M_2)$. Then $\CG (M_1)\subseteq \CG (M_2)$, and so $M_1=M_2$, by Proposition \ref{A28May26}.(2).
\end{proof}

For an $R$-module $M$, we denote by $l_R(M)$ its {\bf length}. A submodule $N$ of $M$ is called a {\bf co-finite} submodule if $l_R(M/N)<\infty$. 

Theorem \ref{14Jun26} shows that every Euclidean module is a Noetherian module such that all its proper factor modules have finite length. 

\begin{theorem}\label{14Jun26}%\marginpar{14Jun26} 
\begin{enumerate}

\item Every Euclidean module is a Noetherian module. 

\item Each  nonzero submodule  $N$ of an l-finite Euclidean $R$-module $(M,d)$   is a co-finite submodule with $l_R(M/N)\leq l([g(M), g(N)])$. Furthermore, if $(\N^\diamond, \geq )=(\N, \geq )$ then $l_R(M/N)\leq g(N)+1$.

\end{enumerate}
\end{theorem}

\begin{proof} 1. Suppose that $M_1\subset M_2\subset \cdots$ is a strictly ascending chain of submodules of a Euclidean module $M$. By Lemma \ref{a14Jun26}, we have the strictly descending chain  $g(M_1) >g(M_2)>\cdots$ in the Artinian poset, a contradiction.  Therefore, the module $M$ is a Noetherian module. 

2. By statement 1, the $R$-module $M$ is a Noetherian module. Therefore, we can produce a chain of submodules, 
$N\subseteq M_n\subset M_{n-1}\subset \cdots \subset M_1=M$, such that all the factor modules $M_1/M_2, \ldots , M_{n-1}/M_n$ are simple $R$-modules. By Lemma \ref{a14Jun26},
$$ 
g(M_1)<g(M_2)<\cdots <g(M_n)\leq g(N).
$$ 
Therefore, $l_R(M/N)\leq l([g(M), g(N)])$, as required. In particular case when  $(\N^\diamond, \geq )=(\N, \geq )$ we obtain that  $l_R(M/N)\leq g(N)+1$.
\end{proof}

 A nonzero submodule of a module is called an {\bf essential  submodule} if it intersects non-trivially every nonzero submodule. A nonzero module is called a {\bf uniform module} if all nonzero submodules are essential. A module is called an {\bf indecomposable module} if it is not a direct sum of  two nonzero submodules. The {\bf Krull dimension} of an $R$-module $M$ is denoted by  $\CK (M)=\CK_R(M)$. The  $R$-module $M$ is called a {\bf critical} $R$-module if for all nonzero submodules $N$ of $M$, $\CK (M/N)<\CK (M)$.

Lemma \ref{b14Jun26} contains some properties of Euclidean modules of infinite length.

\begin{lemma}\label{b14Jun26}%\marginpar{b14Jun26}
Let $M$ be an l-finite  Euclidean $R$-module with $l_R(M)=\infty$. Then:  

\begin{enumerate}

\item $M$ is a uniform $R$-module.

\item $M$ is an indecomposable module.

\item $\soc_R(M)=0$. 

\item Every nonzero submodule of $M$ is of infinite length. 

\item $\CK (M)=1$ and $\CK (M/N)=0$  for all nonzero submodules $N$, i.e. the $R$-module $M$ is a critical $R$-module.

\end{enumerate}
\end{lemma}

\begin{proof} 1. Suppose that $N\cap L=0$ for some nonzero submodules $N$ and $L$ of $M$. Then there is an $R$-monomorphism
$$
M=M/(N\cap L)\ra M/N\times M/L, \;\;m\mapsto (m+N,m+L),
$$
 and we have the contradiction, $\infty =l_R(M)\leq l_R(M/N)+l_R(M/L)<\infty$, by Theorem \ref{14Jun26}.(2). 
Therefore, $M$ is a uniform $R$-module.

2. Clearly, each uniform module is indecomposable. 

3. Suppose that $\soc_R(M)\neq 0$. Then we can fix a simple submodule of $M$, say $N$. Then, by Theorem \ref{14Jun26}.(2), $\infty=l_R(M)=l_R(N)+l_R(M/N)<\infty$, a contradiction. Therefore, $\soc_R(M)=0$. 

4. Notice that every nonzero module  of infinite length has nonzero socle. Therefore, 
 every nonzero submodule of $M$ is of infinite length, by statement 3. 

5. Since the $R$-module $M$ is a Noetherian module (Theorem \ref{14Jun26}.(1)), it has Krull dimension.  Since the $R$-module $M$ is a Noetherian module with $\soc_R(M)=0$, 
 the $R$-module $M$ admits a strictly descending chain of submodules with simple factors. Therefore, $\CK (M)\geq 1$. Since for every nonzero submodule $N$ of $M$, the factor module $M/N$ has finite length  (Theorem \ref{14Jun26}.(2)), $\CK (M/N)=0$. Therefore, $\CK (M)\leq 1$ (as follows from the definition of the Krull dimension of modules). Thus the $R$-module $M$ is a critical module with $\CK (M)=1$. 
\end{proof}

{\bf The greatest common divisor and the list common multiple of elements of a Euclidean module.} Our goal is to define the greatest common divisor and the list common multiple of elements of a Euclidean module. Since every  element of a Euclidean module generated a (cyclic) submodule we have to prove some results about cyclic modules.

For any left ideal $I$ of a ring $R$, the 
{\bf idealizer} of $I$, $\mI_R(I) := \{r \in R \mid Ir \subseteq I\}$, is the largest subring of $R$ containing $I$ as a two-sided ideal. The map
$$
\mI_R(I)/I\ra \text{End}_R(R/I), \;\; \oa :=a+I\mapsto \cdot \oa : R/I\ra R/I, \;\; r+I\mapsto (r+I)\oa = ra+I
$$
is a ring isomorphism. We identify these two rings via the isomorphism.
 We write endomorphisms on the {\em opposite side} of scalars so as to avoid working with opposite rings. When only a single endomorphism is involved, we place it on the left of scalars to simplify the notation.
In particular, the endomorphisms $\End_R(R/I)=\mI_R(I)/I$ are  identified with right  multiplication by the elements of the  ring $\mI_R(I)/I$. Clearly, the $\Z$-module 
$\mI_R(I)/I$ is a subring of the left $R$-module $R/I$.

For any two left ideals $I$ and $J$ of the  ring $R$, let $\mI_R(I, J) := \{r \in R \mid Ir \subseteq J\}$. The map
$$
\mI_R(I, J)/J\ra \text{Hom}_R(R/I, R/J), \;\; \ob :=b+J\mapsto \cdot \ob : R/I\ra R/J, \;\; r+I\mapsto (r+I)\ob +J= rb+J
$$
is a $\Z$-isomorphism. We identify these two  $\Z$-modules via the isomorphism. Clearly, the $\Z$-module  
$\mI_R(I,J)/J$ is a $(\mI_R(I)/I, \mI_R(J)/J)$-bimodule where the bimodule structure is given by the rule: For all elements $(\oa, \bc)\in (\mI_R(I)/I, \mI_R(J)/J)$ and $\ob \in \mI_R(I,J)/J$, $(\cdot \oa)(\cdot \ob)(\cdot \bc)=\cdot\overline{abc}$, i.e.
$$
R/I\stackrel{\cdot \oa}{\ra} R/I \stackrel{\cdot \ob}{\ra}R/J \stackrel{\cdot \bc}{\ra}R/J, \;\; r+I\mapsto rabc+J.
$$
 Lemma \ref{a16Jun26} describes automorphisms
 between cyclic modules.  
 
\begin{lemma}\label{a16Jun26}%\marginpar{a16Jun26}

Let $R$ be a  ring and $I$ and $J$ be its left ideals. Then the $R$-modules $R/I$ and $R/J$ are isomorphic iff there exist elements $x,y\in R$ such that $Ix\subseteq J$, $Jy\subseteq I$, $xy\equiv 1 \mod I$ and $yx\equiv 1 \mod J$. Given the elements $x$ and $y$ as above, the map $f=f_{x,y}: R/I\ra R/J$, $a+I\mapsto ax+J$ is an $R$-module isomorphism with inverse  $b+J\mapsto by+I$ and all isomorphisms are obtained in this way. 
\end{lemma}

\begin{proof} The $R$-modules $R/I$ and $R/J$ are isomorphic iff there exist an isomorphism $f: R/I\ra R/J$, $a+I\mapsto ax+J$ with inverse $f^{-1}: R/J\ra R/I$, $b+J\mapsto by+I$ for     some elements $x,y\in R$ such  $Ix\subseteq J$, $Jy\subseteq I$ and  for all elements $r\in R$, 
$$
r+I = f^{-1}f(r+I)=rxy+I\;\; {\rm and}\;\; 
r+J = ff^{-1}(r+J)=ryx+J,
$$
i.e. $xy\equiv 1 \mod I$ and $yx\equiv 1 \mod J$, and the lemma follows.
\end{proof}

Corollary  \ref{b16Jun26} describes automorphism  between cyclic submodules of a cyclic module in term of elements of ring. 

\begin{corollary}\label{b16Jun26}%\marginpar{b16Jun26}

Let $R$ be a  ring, $I$  be its left ideal and $M=R/I=R\overline{1}$ where $\overline{1}:=1+I$ and $\oa:= a\overline{1}=a+I$ for all elements $a\in R$.
\begin{enumerate}

\item Given elements $a,b\in R$, then $R\oa \simeq R/\ga$ where $\ga :=\ann_R(\oa)$ and $R\ob \simeq R/\gb$ where $\gb:=\ann_R(\ob)$ and 
there is an  $R$-module isomorphism $f:R\oa\ra R\ob$ iff  there exist elements $x,y\in R$ such that $\ga x\subseteq \gb$, $\gb y\subseteq \ga$,
$xy\equiv 1 \mod \ga$ and $yx\equiv 1 \mod\gb$ iff $\ga x\subseteq \gb$, $\gb y\subseteq \ga$,
$$
(1-xy)a=0\;\; {\rm and}\;\; (1-yx)b=0.
$$
Given the elements $x$ and $y$ as above, the map 
$$
f=f_{x,y}: R\oa\simeq R/\ga\ra R/\ob\simeq R/\gb , \;\; a+\ga\mapsto ax+\gb
$$ 
is an $R$-module isomorphism with inverse  $b+\gb\mapsto by+\ga$ and all isomorphisms are obtained in this way.

\item Given elements $a,b\in R$, then $R\oa =R\ob$ iff there exist elements $u,v\in R$ such that 
$$a\equiv ub\mod I\;\; {\rm  and}\;\; b\equiv va\mod I.
$$ 
The elements $u$ and $v$ necessarily satisfy the following conditions: $(1-uv)a\equiv 0\mod I$ and $(1-vu)b\equiv 0\mod I$. 
\end{enumerate}
\end{corollary}

\begin{proof}1.  Clearly,   for the given elements $a,b\in R$, we have  isomorphisms of $R$-modules  $R\oa \simeq R/\ann_R(\oa)$ and $R\ob \simeq R/\ann_R(\ob)$. Now, the corollary follows from Lemma \ref{a16Jun26}.

2.  For the given elements $a,b\in R$, the equality $R\oa=R\ob$ holds iff $\oa= u\ob$ and $\ob=v\oa$ for some elements $u,v\in R$ iff $a\equiv ub\mod I$ and $b\equiv va\mod I
$. Clearly, these two equivalences  yield 
the equivalence $(1-uv)a\equiv 0\mod I$ and $(1-vu)b\equiv 0\mod I$.
\end{proof}

\begin{definition}
Let $M=R/I=\{ \oa =a+I\, | \, a\in R\}$ be a Euclidean $R$-module 
%of infinite length 
 where $I$ is a left ideal of $R$. For  elements $\oa_1, \ldots , \oa_n\in M$, a generator 
$\gcd (\oa_1, \ldots , \oa_n)$ of the cyclic  submodule $\sum_{i=1}^n R\oa_i$ of $M$   is called a {\bf (left) greatest common divisor} of the elements  $\oa_1, \ldots , \oa_n$. Similarly, a generator 
$\lcm (\oa_1, \ldots , \oa_n)$ of the cyclic  submodule $\bigcap_{i=1}^n R\oa_i$ of $M$   is called a {\bf (left) least common multiple} of the elements  $\oa_1, \ldots , \oa_n$. In the case of right Euclidean modules, a {\bf right greatest common divisor} $\gcd_r$   and a {\bf right least common multiple} $\lcm_r$ are defined in the dual way. 
\end{definition}
Notice that, in general,  the elements $\gcd (\oa_1, \ldots , \oa_n)$ and $\lcm (\oa_1, \ldots , \oa_n)$
are not unique, see Corollary \ref{b16Jun26}.(2). More precisely, by Corollary \ref{b16Jun26}.(2), if elements $g, g'\in M$  are greatest common divisors of the  set of elements $\oa_1, \ldots , \oa_n\in M$ then there exist elements $u,v\in R$ such that 
$$g\equiv ug'\mod I\;\; {\rm  and}\;\; g'\equiv vg\mod I.
$$ 
Similarly, if elements $l, l'\in M$  are little  common multiples  of the  set of elements $\oa_1, \ldots , \oa_n\in M$ then there exist elements $u',v'\in R$ such that 
$$l\equiv u'l'\mod I\;\; {\rm  and}\;\; l'\equiv v'l\mod I.
$$

\begin{proposition}\label{A14Jul26}%\marginpar{A14Jul26}
Let $M$ be a Euclidean $R$-module 
%of infinite length 
 and  $\oa_1, \ldots , \oa_n\in M$. Then for an element $g\in M$ the following statements are equivalent:
 
\begin{enumerate}

\item $g=\gcd (\oa_1, \ldots , \oa_n)$.

\item For each $i=1, \ldots , n$, $\oa_i=r_ig$ for some  $r_i\in R$ and if an element $g'\in M$ satisfies the condition above (i.e. $\oa_i=r_i'g'$ for some elements $r_i'\in R$)  then $g=\alpha g'$ for some element $\alpha \in R$.  

\item For each $i=1, \ldots , n$, $\oa_i=r_ig$ for some  $r_i\in R$ and $g=\sum_{i=1}^ns_i\oa_i$ for  some elements $s_i\in R$. 

\end{enumerate}

In particular,  $\gcd (\oa_1, \ldots , \oa_n)= \sum_{i=1}^ns_i\oa_i$ for  some elements $s_i\in R$.
\end{proposition}

\begin{proof} $(1\Leftrightarrow 2)$ $g=\gcd (\oa_1, \ldots , \oa_n)$ iff $Rg=\sum_{i=1}^nR\oa_i$ iff $Rg\supseteq R\oa_i$ for $i=1, \ldots , n$ and $Rg= \sum_{i=1}^nR\oa_i$ 
iff for each $i=1, \ldots , n$, $\oa_i=r_ig$ for some  $r_i\in R$ and if an element $g'\in M$ satisfies the condition above (i.e. for each $i=1, \ldots , n$, $\oa_i=r_i'g'$ for some  $r_i'\in R$ )    then $g=\alpha g'$ for some element $\alpha \in R$ (since $g\in \sum_{i=1}^nR\oa_i=\sum_{i=1}^nRr_i'g'\subseteq Rg'$).

$(1\Leftrightarrow 3)$ $g=\gcd (\oa_1, \ldots , \oa_n)$ iff $Rg=\sum_{i=1}^nR\oa_i$ iff $Rg\supseteq R\oa_i$ for $i=1, \ldots , n$ and $Rg\subseteq \sum_{i=1}^nR\oa_i$ iff  for each $i=1, \ldots , n$, $\oa_i=r_ig$ for some  $r_i\in R$ and $g=\sum_{i=1}^ns_i\oa_i$ for  some elements $s_i\in R$.
\end{proof}
In a left Euclidean ring $R$, the left greatest common divisor of elements $a_1, \ldots , a_n$ 
 is the element $d\in R$  
 that left-divides each element $a_i$
 and is maximal among all such left divisors, in the sense that any other common left divisor also left-divides $d$. 
 In left Euclidean domains, left greatest common divisors  always exist and are unique up to left multiplication by a unit. Proposition \ref{A14Jul26} is module-theoretic analogue of these results.

Similarly, in the left Euclidean ring $R$, the left least common multiple of elements 
 $a_1, \ldots , a_n$  is the element $l\in R$ 
 that is left-multiple of each $a_i$
 and is minimal among such left multiples, in the sense that any other common left multiple is itself a left multiple of  $l$. Equivalently, 
$l$  is a generator of the intersection of the principal left ideals $Ra_i$. Proposition \ref{B14Jul26} is module-theoretic analogue of these results.

\begin{proposition}\label{B14Jul26}%\marginpar{B14Jul26}
Let $M$ be a Euclidean $R$-module   
 and  $\oa_1, \ldots , \oa_n\in M$. Then for an element $l\in M$ the following statements are equivalent:
 
\begin{enumerate}

\item $l=\lcm (\oa_1, \ldots , \oa_n)$.

\item For each $i=1, \ldots , n$, $l=\alpha_i\oa_i$ for some  $\alpha_i\in R$ and if an element $l'\in M$ satisfies the condition above (i.e. $l'=\alpha_i'\oa_i$ for some  $\alpha_i'\in R$)  then $l'=\alpha l$ for some element $\alpha \in R$.

\end{enumerate}
\end{proposition}

\begin{proof} $(1\Leftrightarrow 2)$ $l=\lcm (\oa_1, \ldots , \oa_n)$ iff $Rl=\bigcap_{i=1}^nR\oa_i$ iff $Rl\subseteq R\oa_i$ for $i=1, \ldots , n$ and $Rl= \bigcap_{i=1}^nR\oa_i$ 
iff for each $i=1, \ldots , n$, $l=\alpha_i\oa_i$ for some  $\alpha_i\in R$ and if an element $l'\in M$ satisfies the condition above  (i.e. $l'=\alpha_i'\oa_i$ for some  $\alpha_i'\in R$)  then $l'=\alpha l$ for some element $\alpha \in R$ (since $Rl=\bigcap_{i=1}^nR\oa_i\ni l'$).
\end{proof}
Clearly, 
$$\gcd (\oa_1, \ldots , \oa_n)=\gcd\Big( \gcd (\oa_1, \ldots ,\oa_{n-1}),  \oa_n\Big)\;\; {\rm and}\;\; \lcm (\oa_1, \ldots , \oa_n)=\lcm\Big( \lcm (\oa_1, \ldots ,\oa_{n-1}),  \oa_n\Big).
$$

\begin{definition}
Let $M$ be an l-finite Euclidean $R$-module. For each nonzero element elements $m\in M$, let $|m|:=l_R(M/Rm)$ (by Theorem \ref{14Jun26}.(2), $|m|\in \N$). 
\end{definition}

Let $M$ be an l-finite Euclidean $R$-module and $N$ and $L$ be nonzero submodules of $M$.

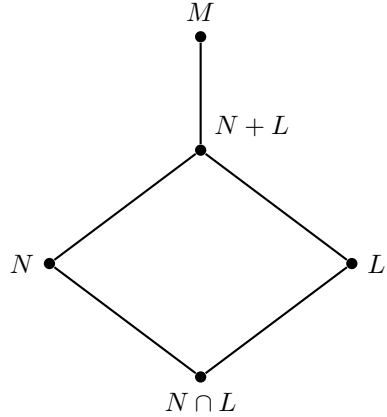
\begin{figure}[htbp]
  \centering
  \begin{tikzpicture}[scale=1.0, >=stealth]
    % Define coordinates for the rhombus vertices
    % N+L is at the top, counted clockwise: N+L(top), L(right), N\cap L(bottom), N(left)
    \node (N+L) at (0, 1.5) [circle, fill, inner sep=1.5pt, label=above right:$N+L$] {};
    \node (L) at (2, 0) [circle, fill, inner sep=1.5pt, label=right:$L$] {};
    \node (NL) at (0, -1.5) [circle, fill, inner sep=1.5pt, label=below:$N\cap L$] {};
    \node (N) at (-2, 0) [circle, fill, inner sep=1.5pt, label=left:$N$] {};
    
    % Vertex M above N+L
    \node (M) at (0, 3) [circle, fill, inner sep=1.5pt, label=above:$M$] {};
    
    % Draw the rhombus sides
    \draw[thick] (N+L) -- (L) -- (NL) -- (N) -- (N+L) -- cycle;
    
    % Draw connection from M to N+L with a solid line
    \draw[thick] (M) -- (N+L);
    
  \end{tikzpicture}
 % \caption{%Rhombus $ABCD$ with vertex $M$ above $N+L$.}
  \label{fig:rhombus}
\end{figure}

Suppose that $R=\Z$ and $a,b\in \Z\backslash \{ 0\}$. Theorem \ref{15Jun26} gives two  `additive/logarithmic' analogues of the equality 
$$ab=\gcd (a,b)\lcm(a,b).$$ The first statement in Theorem \ref{15Jun26} is a module-theoretic one and the second statement is given  element-wise. 

\begin{theorem}\label{15Jun26}%\marginpar{15Jun26}
Let $M=R/I=\{ \oa =a+I\, | \, a\in R\}$ be an l-finite Euclidean $R$-module 
%of infinite length 
where $I$ is a left ideal of $R$ Then: 
\begin{enumerate}

\item  For all nonzero submodules $N$ and $L$  of $M$,  
$$l_R(M/N)+l_R(M/L)=l_R(M/(N+L))+l_R(M/N\cap L).$$ 

\item For all nonzero elements $\oa,\ob\in M\backslash \{ 0\}$, $|\oa| +|\ob|=|\gcd (\oa, \ob)|+|\lcm (\oa, \ob)|$.

\end{enumerate}
\end{theorem}

\begin{proof} 1. (i) $L_R(M/N\cap L)<\infty$: 
 The statement is obvious if the $R$-module $M$ is of finite length. Suppose that the $R$-module $M$ has infinite length. 
 By the assumptions, the $R$-modules $N$ and $L$ are nonzero, then so is their intersection since the $R$-module $M$ is a uniform module, by  Lemma  \ref{b14Jun26}.(1).

(ii) $l_R(M/N)+l_R(M/L)=l_R(M/(N+L))+l_R(M/N\cap L)$: By statement (i), $L_R(M/N\cap L)<\infty$. Therefore, the following lengths are  finite:
\begin{eqnarray*}
i&:=&l_R(M/N\cap L),\\
j &:=&l_R((N+L)/N)=l_R(L/N\cap L)\;\;  \text{(since $(N+L)/N)\simeq L/N\cap L$)},\\
k&:=&l_R((N+L)/L)=l_R(N/N\cap L)\;\; \text{ (since $(N+L)/L)\simeq N/N\cap L$)}.
\end{eqnarray*}
It follows from the diagram above that 
$$ 
l_R(M/N)=i+j, \;\; l_R(M/L)=i+k, \;\; l_R(M/(N+L))=i\;\; {\rm and}\;\; l_R(M/N\cap L)=i
+j+k.
$$
Hence, $l_R(M/N)+l_R(M/L)=2i+j+k=l_R(M/(N+L))+l_R(M/N\cap L)$.

2. Statement 2 is a particular case of statement 1 where $N=R\oa$,  $L=R\ob$, $N+L=R
\gcd (\oa, \ob)$ and $N\cap L=R
\lcm (\oa, \ob)$.
\end{proof}

{\bf Euclidean rings and modules and their localizations.} For a ring $R$ and an ideal $\ga$, let $\Den_l(R, \ga)$ be the set of left denominator sets $S$ of $R$ with $ \ass_l(S):=\{ r\in R\, | \, sr=0$ for some $s\in S\}=\ga$. Similarly,  the set  $\Den_r(R, \ga)$ of right denominator sets of $R$ is defined. Their intersection $\Den (R, \ga):=\Den_l(R, \ga) \cap \Den_r(R, \ga)$ is the set of denominator sets of $R$.

Let $S$ be a left denominator set of $R$ and  $M$ be an $R$-module. Then $S^{-1}M$ is an $S^{-1}R$-module. The $R$-module $M$ contains the $S$-{\em torsion} submodule $\tor_S(M):=\{ m\in M\, | \, sm=0$ for some element $s\in S\}$. Then the factor module $\bM=M/\tor_S(m)$ is  an essential submodule of the $R$-module $S^{-1}M$ such that $S^{-1}\bM\simeq S^{-1}M$. A submodule 
 $N$ of the $R$-module $\bM$ is called an $S$-{\em saturated} submodule if $sm\in N$ for some elements $s\in S$ and $m\in N$ implies  $m\in N$. Let $\Sub_{R,S}(\bM)$ be the set of $S$-saturated submodule of the $R$-module $M$ and $\Sub_{S^{-1}R}(S^{-1}M)$ be the set of submodules of the $S^{-1}R$-module $S^{-1}M$. Then the map
%\marginpar{Sub-RSM}
\begin{equation}\label{Sub-RSM}
\Sub_{R,S}(\bM)\ra \Sub_{S^{-1}R}(S^{-1}M), N\mapsto S^{-1}N
\end{equation}
is a  bijection with inverse $\CN\mapsto \bM\cap \CN$.

 Lemma \ref{a8Jun26} shows that localizations of Euclidean modules have properties very similar to Euclidean modules. 
 
\begin{lemma}\label{a8Jun26}%\marginpar{a8Jun26}

Let $(M, d)$ be a Euclidean $R$-module and   $S$ be a left denominator set of $R$. Then the $S^{-1}R$-module $S^{-1}M$ is a principal submodule  module that contains an essential Euclidean submodule $(\bM=  M/\tor_S(M), d_{\bM})$ such that $S^{-1}\bM\simeq S^{-1}M$ and  the set $\CG (\bM):=\{ m\in \bM\, | \, d_{\bM}(m)=l \}$, where $l:=\min \{ d_{\bM}(m)\, | \, 0\neq m\in \bM \}$, of 1-generators of the $R$-module $\bM$ 
is also a set of 1-generators of the  $S^{-1}R$-module $S^{-1}M$. The map
$
\Sub_{R,S}(\bM)\ra \Sub_{S^{-1}R}(S^{-1}M)$,  $N\mapsto S^{-1}N $
is a  bijection with inverse $\CN\mapsto \bM\cap \CN$, see (\ref{Sub-RSM}). 

\end{lemma}

\begin{proof}  The $R$-module $(\bM, d_{\bM})$ is a Euclidean $R$-module as a factor module of the Euclideam $R$-module $M$ (Lemma \ref{a7Jun26}.(2)). The $S^{-1}R$-modules $S^{-1}\bM$ and $S^{-1}M$ are isomorphic. Therefore, the set $\CG (\bM)$ of 1-generators of the $R$-module $\bM$ 
is also a set of 1-generators of the  $S^{-1}R$-module $S^{-1}M$.
\end{proof}

In general, for a Euclidean $R$-module $(M, d)$ and a left denominator set $S$ of $R$, the degree function $d$ cannot be extended to the localization $S^{-1}M$. For example, $R=K[x]$ is a polynomial algebra over a field $K$,  $(M=K[x], \deg)$ and $S=\{x^i\, | \, i\in \N\}$. Then $S^{-1}K[x]=K[x^{\pm 1}]$ and the concept of the degree $\deg$ makes no sense for the Laurent polynomial algebra  $K[x^{\pm 1}]$. 

Theorem \ref{8Jun26} identifies several sufficient conditions for extending  Euclidean function to  localization ring.

\begin{theorem}\label{8Jun26}%\marginpar{8Jun26}
\begin{enumerate}

\item Let $(R,d)$ be a left (resp., right)  Euclidean  ring,  $S\in \Den(R, \ga)$, $\pi :R\ra \bR :=R/\ga$, $r\mapsto \br :=r+\ga$ and $\bS :=\pi (S)$. Then $(\bR:=R/\ga, d_{\bR})$ is a left  (resp., right) Euclidean ring (Lemma \ref{a7Jun26}.(2)). Suppose that $d_{\bR}(sr)=d_{\bR}(rs)=d_{\bR}(r)$  for all elements $s\in \bS$ and $r\in \bR$. Then the ring ($S^{-1}R\simeq \bS^{-1}\bR \simeq \bR\bS^{-1}, \d )$ is a left  (resp., right) Euclidean  ring where $\d (s^{-1}r):=
d_{\bR}(r)$  (resp., $\d (rs^{-1}):=
d_{\bR}(r)$) for all elements $s\in \bS$ and $r\in \bR$. Furthermore, $\d (rs^{-1}):=
d_{\bR}(r)$ (resp., $\d (s^{-1}r):=
d_{\bR}(r)$ ) for all elements $s\in \bS$ and $r\in \bR$.

\item Let $(R,d)$ be a left Euclidean  ring,  $S\in \Den_l(R, \ga)$, $\pi :R\ra \bR :=R/\ga$, $r\mapsto \br :=r+\ga$ and $\bS :=\pi (S)$. Then $(\bR:=R/\ga, d_{\bR})$ is a left Euclidean ring (Lemma \ref{a7Jun26}.(2)). Suppose that $S$ is a right saturated left denominator set (that is the inclusion $st^{-1}\in \bR$, where $s,t\in \bS$, implies  $st^{-1}\in \bS$)  and $d_{\bR}(sr)=d_{\bR}(r)$  for all elements $s\in \bS$ and $r\in \bR$. Then the ring ($S^{-1}R, \d )$ is a left Euclidean  ring where $\d (s^{-1}r):=
d_{\bR}(r)$ for all elements $s\in \bS$ and $r\in \bR$.

\item  Let $(R,d)$ be a right Euclidean  ring,  $S\in \Den_r(R, \ga)$, $\pi :R\ra \bR :=R/\ga$, $r\mapsto \br :=r+\ga$ and $\bS :=\pi (S)$. Then $(\bR:=R/\ga, d_{\bR})$ is a right  Euclidean ring (Lemma \ref{a7Jun26}.(2)). Suppose that  that $S$ is a left  saturated right denominator set (that is the inclusion $t^{-1}s\in \bR$, where $s,t\in \bS$, implies  $t^{-1}s\in \bS$)  and  
 $d_{\bR}(rs)=d_{\bR}(r)$ for all elements $s\in \bS$ and $r\in \bR$. Then the ring ($ \bR\bS^{-1}, \d )$ is a right Euclidean  ring where $\d (rs^{-1}):=
d_{\bR}(r)$ for all elements $s\in \bS$ and $r\in \bR$.  

\end{enumerate}
\end{theorem}

\begin{proof} 1.  Suppose  that $R$ is a left Euclidean ring.

(i) {\em The map $\d$ is a well-defined map}: By the definition, $\d (s^{-1}a)=d_{\bR}(a)$ for all elements $s\in \bS$ and $a\in \bR$. To prove the statement (i),  we have to show that if $s^{-1}a=t^{-1}b$ for some elements $s,t\in \bS$ and $a,b\in \bR$ then $d_{\bR}(a)=d_{\bR}(b)$. Since $\bS\in \Den_r(\bR, 0)$,
$$
t^{-1}b=b_1s_1^{-1}
$$
 for some elements $b_1\in \bR$ and $s_1\in \bS$. Then the equalities $s^{-1}a=b_1s_1^{-1}$ and $t^{-1}b=b_1s_1^{-1}$ are equivalent to the equalities $as_1=sb_1$
 and $bs_1=tb_1$. Then 
$$
d_{\bR}(a)=d_{\bR}(as_1)=d_{\bR}(sb_1)=d_{\bR}(b_1)=
d_{\bR}(tb_1)=d_{\bR}(bs_1)=d_{\bR}(b),
$$
as required.

(ii)  $\d (as^{-1})=
d_{\bR}(a)$ {\em for all elements $s\in \bS$ and} $a\in \bR$: Since $\bS\in \Den_l(\bR, 0)$, $as^{-1}=s_1^{-1}a_1$ for some elements $s_1\in \bS$ and $a_1\in \bR$. Therefore, $s_1a=a_1s$ and
$$
\d (as^{-1})=\d (s_1^{-1}a_1)=d_{\bR}(a_1)=d_{\bR}(a_1s)=d_{\bR}(s_1a)=d_{\bR}(a).
 $$

(iii) {\em The map $\d$ is a Euclidean map for the ring $\bS^{-1}\bR$}: Given elements $s^{-1}a, t^{-1}b\in \bS^{-1}\bR$, where $s,t\in \bS$ and $a,b\in \bR$, we have to find elements $q, r\in \bS^{-1}\bR$ such that 
$$ s^{-1}a=q t^{-1}b+r \;\; {\rm and}\;\;  \d(r)<\d (t^{-1}b).
$$
The ring $(\bR, d_{\bR})$ is an Euclidean ring. Therefore, there exist  elements $q', r'\in \bR$ such that 
$$ a=q' b+r' \;\; {\rm and}\;\;  \d(r')<\d (b).
$$
Now, $s^{-1}a=s^{-1}q' t \cdot t^{-1}b+s^{-1}r'= q t^{-1}b+r$ where
$q:=s^{-1}q' t$ and $r:=s^{-1}r'$, and 
$$
\d (r)=\d (s^{-1}r')=d_{\bR} (r')<d_{\bR}(b)=\d (t^{-1}b),
$$
as required. 

So, the result holds for left Euclidean rings. Notice that the assumptions of statement 1 are left-right symmetric.
 The opposite ring of a {\em left} Euclidean ring is a {\em right} Euclidean ring. Let $R$ be a right Euclidean ting. Then applying the result to its opposite ring, which is a left Euclidean, we obtain the result for all right Euclidean rings. 

2. (i) {\em The map $\d$ is a well-defined map}: By the definition, $\d (s^{-1}a)=\d_{\bR}(a)$ for all elements $s\in \bS$ and $a\in \bR$. To prove the statement (i),  we have to show that if $s^{-1}a=t^{-1}b$ for some elements $s,t\in \bS$ and $a,b\in \bR$ then $d_{\bR}(a)=d_{\bR}(b)$. Since $\bS\in \Den_l(\bR, 0)$,
$$
st^{-1}=s_1^{-1}t_1
$$
 for some elements $t_1\in \bR$ and $s_1\in \bS$. Since the left denominator set $\bS$ is right saturated and $s,s_1,t\in \bS$, the inclusion  $t_1=s_1st^{-1}\in  \bR $ implies the inclusion $t_1=s_1st^{-1}\in  \bS $. 
 Now, $a=st^{-1}b=s_1^{-1}t_1b$, and so $s_1a=t_1b$. Therefore,
 $$d_{\bR}(a)=   d_{\bR}(s_1a)=d_{\bR}(t_1b)=d_{\bR}(b),
 $$
 as required.
 
 (ii) {\em The map $\d$ is a Euclidean map for the ring $\bS^{-1}\bR$}: Given elements $s^{-1}a, t^{-1}b\in \bS^{-1}\bR$, where $s,t\in \bS$ and $a,b\in \bR$, we have to find elements $q, r\in \bS^{-1}\bR$ such that 
$$ s^{-1}a=q t^{-1}b+r \;\; {\rm and}\;\;  \d(r)<\d (t^{-1}b).
$$
The ring $(\bR, d_{\bR})$ is an Euclidean ring. Therefore, there exist  elements $q', r'\in \bR$ such that 
$$ a=q' b+r' \;\; {\rm and}\;\;  \d(r')<\d (b).
$$
Now, $s^{-1}a=s^{-1}q' t \cdot t^{-1}b+s^{-1}r'= q t^{-1}b+r$ where
$q:=s^{-1}q' t$ and $r:=s^{-1}r'$, and so
$$
\d (r)=\d (s^{-1}r')=d_{\bR} (r')<d_{\bR}(b)=\d (t^{-1}b),
$$
as required. 

3. Statement 3 is a dual statement to statement 2, i.e. a ring satisfies statement 2 iff its opposite ring satisfies statement 3. So, statement 3 follows from statement 2.  
\end{proof}

%%%%%%%%%%%%%   Section 3    %%%%%%%%

\section{Criteria for graded and filtered modules to be Euclidean  modules and examples}\label{GRAD-FILT} %\marginpar{GRAD-FILT}

%***    delete    ***

% In this section, the following resulrs are proven.  For an arbitrary  skew polynomial ring $A=D[x;\s, \d]$, 
%Theorem \ref{Sig-Del-28May26} describes a class of Euclidean modules. Similarly, for an arbitrary  skew Laurent polynomial ring $L=D[x^{\pm 1};\s]$, 
%Theorem \ref{SkewLaur-28May26} describes a class of Euclidean modules. 
% For an $\N$-graded ring, Theorem \ref{11Jun26} is a criterion for a $\N$-graded module to be a Euclidean module. Similarly, for a skew Laurent polynopmial ring,  Theorem \ref{12Jun26}  is a criterion for a $\Z$-graded module to be a Euclidean module. Proposition \ref{FiltEucl-28May26} shows that if the associated graded module is a Euclidean module over the associated graded ring then the original module is a Euclidean module. \\

% ***

{\bf Skew  polynomial rings $A=D[x;\s , \d]$ and their Euclidean modules.}
Let $A=D[x;\s, \d]$ and $V$ be a $D$-module. Then the induced $A$-module 
%\marginpar{AtDV}
\begin{equation}\label{AtDV}
A\t_DV=\bigoplus_{i\in \N}x^i\t_DV\simeq \bigoplus_{i\in \N}{}^{\s^{-i}}V
\end{equation}
 is a direct sum of twisted  $D$-modules          
$x^i\t_DV\simeq {}^{\s^{-i}}V$ since for all elements $d\in D$ and $v\in V$, $dx^i\t_Dv=
x^i\t_D \s^{-i}v$. Each element $v\in A\t_DV$, is a unique sum
$$v=\sum_{i=0}^nx^i\t_D v_i\;\; {\rm where}\;\;  v_i\in V.
$$
 Suppose that $v\neq 0$, i.e. there is a nonzero coefficient $v_i\neq 0$. Then the natural number $\deg (v):=\max\{i\in \N\, | \, v_i\neq 0\}$ is called the {\em degree} of the element $v$, the element $x^n\t_D v_n$, where $n=\deg (v)$, is called the {\em leading term} of $v$ and the element $v_n$  is called the  {\em  leading coefficient} of $v$. By  definition, $\deg (0):=-\infty$.
  Then for all elements  $v, u\in A\t_DV$, $a\in A$ and $d\in D$:
 \begin{eqnarray*}
 \deg (v+u)&\leq &\max \{ \deg (v), \deg (u)\},\\
 \deg (av)&\leq &\deg (a)+\deg (v),\\
\deg (dv)&\leq &\deg (v).
 \end{eqnarray*}
 The degree function $\deg$ determines the ascending {\em degree filtration} $\{ (A\t_DV)_{\leq n} \}_{i\in \N}$ on the $A$-module $A\t_DV$ where $(A\t_DV)_{\leq n}:=\{ v\in A\t_DV\, | \, \deg (v)\leq n\}$:
$$
A\t_DV=\bigcup_{i\in \N} (A\t_DV)_{\leq n}\;\; {\rm where}\;\;  A_{\leq n}(A\t_DV)_{\leq m}\subseteq (A\t_DV)_{\leq n+m}\;\; \text{for all }\;\; n,m\in \N.
$$ 
Then the {\em associated graded module} of $A\t_DV$, 
$${\rm gr}(A\t_DV):=\bigoplus_{n\in \N}(A\t_DV)_{\leq n}/(A\t_DV)_{\leq n-1}\simeq \bigoplus_{i\in \N}x^i\t_DV\simeq \bigoplus_{i\in \N}{}^{\s^{-i}}V
$$
is an $\N$-graded ${\rm gr} (A)$-module, i.e. an $\N$-graded  module over the  skew polynomial ring ${\rm gr} (A)\simeq R:=D[x;\s]=\bigoplus_{i\in \N}Dx^i$. This means that for all elements $i,j\in \N$, 
$$ Dx^i\cdot x^j\t_DV\subseteq x^{i+j}\t_DV.
$$
 Clearly, ${\rm gr}(A\t_DV)\simeq R\t_DV$, an isomorphism of graded $R$-modules. 
Let us consider the {\em leading term map}
       %\marginpar{lV-map}
\begin{equation}\label{lV-map}
 l_V:A\t_DV \ra {\rm gr}(A\t_DV), \;\; v=\sum_{i=0}^nx^i\t_D v_i\mapsto x^n\t_D v_n \;\; {\rm where}\;\; n=\deg(v).
\end{equation}
It is not an additive map.

%By Lemma \ref{a7Jun26}.(1),  every  submodule $N$ of the Euclidean $A$-module $A\t_DV$ is also a Euclidean module $(N, \deg |_N )$.
 The $A$-module $N$ admits the degree filtration that is determined by the degree function $\deg |_N$. By the definition,  the degree function $\deg |_N$ is the restriction of the degree function $\deg$ of the $A$-module $A\t_DV$. Therefore, 
 the associated graded ${\rm gr}(A)$-module ${\rm gr}(N)$ is a submodule of the ${\rm gr}(A)$-module ${\rm gr}(A\t_DV)$.

\begin{theorem}\label{Sig-Del-28May26}%\marginpar{Sig-Del-28May26}
Let $D$ be a ring, $A=D[x;\s, \d]$ be a skew polynomial, $\s\in \Aut (D)$, $\d$ be a $\s$-derivation of $D$ and  $V$ be a simple $D$-module. Then:

\begin{enumerate}

\item  The $A$-module $A\t_DV$ is a Euclidean module with Euclidean function $\deg$.
%  where for a nonzero element $v=\sum_{i\geq 0}x^i\t_Dv_i\in A\t_DV$, $d_V (v):=\max \{ i\in \N\, | \, v_i\neq 0\}$.

\item The  $A$-module $A\t_DV$ is a principal submodule  module.

\item If $N$ is a nonzero submodule of $A\t_DV$ then the set of 1-generators of the $A$-module $N$ contains  the non-empty  set $\CG (N):= \{ n\in N\, | \, d_V(n)=m \}$, where $m:=\min \{ d_V(n)\, | \, 0\neq n\in N \}$. For all elements $n=x^{i_1}\t_Dn_{i_1}+\cdots +x^{i_s}\t_Dn_{i_s}\in \CG (N)$,  where  $i_1<\cdots <i_s=m$ and $n_{i_1},\ldots , n_{i_s}\in V\backslash \{ 0\}$, 
$\lann_D(x^{i_\nu}\t_Dn_{i_\nu})\subseteq \lann_D(x^m\t_Dn_m)$ for all $\nu=1,\ldots , s$.

\item For every $A$-submodule $N$ of $A\t_DV$,  $l_V(\CG (N))=\CG ({\rm gr}(N))$ where ${\rm gr} (N)$ is a ${\rm gr} (A)$-submodule of ${\rm gr}(A\t_DV)$.

\end{enumerate}
\end{theorem}

\begin{proof} 1. Given elements $u=x^m\t_Du_m+\cdots$ and $v=x^n\t_Dv_n+\cdots$ of the $A$-module $A\t_DV$ of degrees $m$ and $n$, respectively,  where the three dots denote elements of smaller degrees, i.e. the elements $x^m\t_Du_m$ and $x^n\t_Dv_n$ are the leading terms of the elements $u$ and $v$, respectively, where $u_m,v_n\in V\backslash \{ 0\}$, we have to show that $u=qv+r$ for some elements $q\in A$ and $r\in A\t_DV$ with $\deg (r)<\deg (v)$.  If $m<n$ then $u=0v+u$, as required (since $\deg (u)=m<n=\deg (v)$).

Suppose that $m\geq n$. The $D$-module $V$ is a simple module. Therefore, $V=Du_m=Dv_n$, and so 
$$
u_m=\s^{-m}(d)v_n\;\; \text{for some element}\;\; d\in D.
$$
It follows from the equalities  $dx^{m-n}x^n\t_Dv_n=\Big( x^m\s^{-m}(d)+\cdots\Big)\t_Dv_n=x^m\t_D \s^{-m}(d)v_n+\cdots=x^m\t_Du_m+\cdots$ that 
$$\deg (u-dx^{m-n}v)<m=\deg (u),$$
and we use induction on the degree $m$ to finish the proof (alternatively, repeat the same argument several times). 

2. By Proposition \ref{A28May26}.(1), statement 2 follows from statement 1.

3. By Proposition \ref{A28May26}.(2), the set of 1-generators of the $A$-module $N$ contains  the non-empty  set $\CG (N)$.
It remains to  show that for all elements $n=x^{i_1}\t_Dn_{i_1}+\cdots +x^{i_s}\t_Dn_{i_s}\in \CG (N)$,  where  $i_1<\cdots <i_s=m$ and $n_{i_1},\ldots , n_{i_s}\in V\backslash \{ 0\}$, 
$$
\lann_D(x^{i_\nu}\t_Dn_{i_\nu})\subseteq \lann_D(x^m\t_Dn_m)\;\; \text{ for all} \;\;\nu=1,\ldots , s.
$$
 Suppose that this is not the case. Then 
 there is an element $n$ as above such that 
 $s\geq 2$ and there is an index, say  $t$,  such that 
$\lann_D(x^{i_\nu}\t_Dn_{i_\nu})\subseteq \lann_D(x^m\t_Dn_m)$ for all $\nu=t,\ldots , s$ and $\lann_D(x^{i_{t-1}}\t_Dn_{i_{t-1}})\not\subseteq \ga:=\lann_D(x^m\t_Dn_m)$. Then $\{0\}\neq \ga n\subseteq N$ and the degrees of  all nonzero elements of the set $\ga n$ are strictly smaller than $m$, a contradiction. 

4. Recall that ${\rm gr}(A)=R:=D[x\; \s]$,   the ${\rm gr}(A)$-modules ${\rm gr}(N)$ and ${\rm gr}(A\t_DV)$ are graded $R$-modules and ${\rm gr}(A\t_DV)\simeq R\t_DV$, as graded $R$-modules. For all elements $v\in A\t_DV$, $\deg(l_V(v)) = \deg (v)$. Therefore, $l_V(\CG (N))=\CG ({\rm gr}(N))$.
\end{proof}

In the case of the skew polynomial ring $R=D[x;\s]$, we can strengthen statement 3 of Theorem \ref{Sig-Del-28May26}.

\begin{corollary}\label{SkewPol-28May26}%\marginpar{SkewPol-28May26}
Let $D$ be a ring and  $R=D[x;\s]$ be a skew polynomial and $V$ be a simple $D$-module. Then:

\begin{enumerate}

\item  The $R$-module $R\t_DV$ is a Euclidean module with Euclidean function $\deg$. 
% (the $V$-degree)  where for a nonzero element $v=\sum_{i\geq 0}x^i\t_Dv_i\in R\t_DV$, $d_V (v):=\max \{ i\in \N\, | \, v_i\neq 0\}$.

\item The  $R$-module $R\t_DV$ is a principal submodule  module.

\item If $N$ is a nonzero submodule of $R\t_DV$ then the set of 1-generators of the $R$-module $N$ contains  the non-empty  set $\CG (N):= \{ n\in N\, | \, d_V(n)=l \}$, where $l:=\min \{ d_V(n')\, | \, 0\neq n'\in N \}$, and 
$\CG (N)\subseteq \CU\CA (R\t_DV)$.

\end{enumerate}
\end{corollary}

\begin{proof} 1 and 2. Statements 1 and 2 are particular cases of Theorem \ref{Sig-Del-28May26}.(1,2).

3. By Theorem \ref{Sig-Del-28May26}.(3), the set of 1-generators of the $R$-module $N$ contains  the non-empty  set $\CG (N)$. Let us show that $\CG (N)\subseteq \CU\CA (R\t_DV)$. Suppose that $n=x^{i_1}\t_Dn_{i_1}+\cdots +x^{i_s}\t_Dn_{i_s}\in \CG (N)$,  where  $i_1<\cdots <i_s=m$ and $n_{i_1},\ldots , n_{i_s}\in V\backslash \{ 0\}$. By  
Theorem \ref{Sig-Del-28May26}.(3), 
$$
\ga_{i_\nu}:=\lann_D(x^{i_\nu}\t_Dn_{i_\nu})\subseteq \ga_m:=\lann_D(x^m\t_Dn_m)\;\; \text{for all}\;\;\nu=1,\ldots , s.
$$
Suppose that the element $n$ is not a uni-annihilator element. This means that  
 $s\geq 2$ and there is an index, say  $t$,  such that  $\ga_{i_{t-1}}\subset \ga_{i_t}=\ga_{i_{t+1}}=\cdots =\ga_m$. Take an element $a\in \ga_{i_t}\backslash \ga_{i_{t-1}}$. Then $\deg (an)=i_{t-1}<m$, a contradiction.
 \end{proof}

{\bf Skew Laurent polynomial rings $L=D[x^{\pm 1};\s]$ and their Euclidean modules.} Let $L=D[x^{\pm 1};\s]=\bigoplus_{i\in \Z}Dx^i=\bigoplus_{i\in \Z}x^iD$ be a skew Laurent polynomial ring. The ring $L=S_x^{-1}R\simeq RS_x^{-1}$ is the localization of the skew polynomial ring $R=D[x; \s]$ at the denominator set $S_x:=\{ x^i\, \ \, i\in \N\}$ of $R$. Clearly, $R\subseteq L$ and the ring $L$ is a $\Z$-graded ring.  Each element $a\in L$, is a unique sum
$$a=\sum_{i=m}^na_ix^i=\sum_{i=m}^nx^i\s^{-i}(a_i)\;\; {\rm where}\;\;  a_i\in D.
$$
 Suppose that $a\neq 0$, i.e. there is a nonzero coefficient $a_i\neq 0$. Then the natural number $l(a):=n-m$ is called the {\em length} of the element $a$ where
  $m:=\min \{i\in \N\, | \, a_i\neq 0\}$ and $n:=\max\{i\in \N\, | \, a_i\neq 0\}$. 
  The elements $a_mx^m$ and $a_nx^n$ are  called the {\em lowest} and  {\em highest/leading terms} of $a$, respectively.   The elements $a_m$  and $a_n$ are called the  {\em lowest} and {\em highest/leading coefficients} of $a$, respectively. By  definition, $l (0):=-\infty$.
  Then for all elements  $a, b\in L$ and $d\in D$:
 $$ 
 l(ab)\leq l (a)+l(b),\;\; l(da)\leq l(a) \;\; {\rm and}\;\; l(x^ia)=l(ax^i)=l(a)\;\; {\rm for \; all}\;\; i\in \Z.
 $$
Let  $V$ be a $D$-module. Then the induced $L$-module 
%\marginpar{LtDV}
\begin{equation}\label{LtDV}
L\t_DV=\bigoplus_{i\in \Z}x^i\t_DV\simeq \bigoplus_{i\in \Z}{}^{\s^{-i}}V
\end{equation}
 is a direct sum of twisted  $D$-modules          
$x^i\t_DV\simeq {}^{\s^{-i}}V$ since for all elements $d\in D$ and $v\in V$, $dx^i\t_Dv=
x^i\t_D \s^{-i}v$. Each element $v\in L\t_DV$, is a unique sum
$$v=\sum_{i=m}^nx^i\t_D v_i\;\; {\rm where}\;\;  v_i\in V.
$$
 Suppose that $v\neq 0$, i.e. there is a nonzero coefficient $v_i\neq 0$. Then the natural number $l(v):=n-m$ is called the {\em length} of the element $v$ where
  $m:=\min \{i\in \N\, | \, v_i\neq 0\}$ and $n:=\max\{i\in \N\, | \, v_i\neq 0\}$. 
  The elements $x^m\t_D v_m$ and $x^n\t_D v_n$ are  called the {\em lowest} and  {\em highest/leading terms} of $v$, respectively.   The elements $v_m$  and $v_n$ are called the  {\em lowest} and {\em highest/leading coefficients} of $v$, respectively. By  definition, $l (0):=-\infty$.
  Then for all elements  $v\in L\t_DV$, $a\in L$ and $d\in D$:
 $$
 l(av)\leq l (a)+l(v) \;\; {\rm and}\;\; 
l(dv)\leq l(v).
 $$

\begin{proposition}\label{SkewLaur-28May26}%\marginpar{SkewLaur-28May26}
Let $D$ be a ring,   $L=D[x^{\pm 1};\s]$ be a skew Laurent polynomial and $V$ be  a simple $D$-module. Then:

\begin{enumerate}
\item  The $L$-module $L\t_DV$ is a Euclidean module with Euclidean function $l$ (the length function).

\item The  $L$-module $L\t_DV$ is a principal submodule  module.

\item If $N$ is a nonzero submodule of $L\t_DV$ then the set of 1-generators of the $L$-module $N$ contains  the non-empty  set $\CG (N) := \{ n\in N\, | \, l(n)=l \}$, where $l:=\min \{ l(n')\, | \, 0\neq n'\in N \}$, and 
$\CG (N)\subseteq \CU\CA (L\t_DV)$.

\end{enumerate}
\end{proposition}

\begin{proof} 1. Given elements $u,v\in L\t_DV$, we have to show that $u=qv+r$ for some elements $q\in L$ and $r\in L\t_DV$ such that $l(r)<l(v)$. 
 If $l(u)<l(v)$ then $u=0v+u$ where $q=0$ and $r=u$. 
 
 Suppose that $l(u)\geq l(v)$ and  
$u=x^m\t_Du_m+\cdots$ and $v=x^n\t_Dv_n+\cdots$  where the three dots denote elements of smaller degrees, i.e. the elements $x^m\t_Du_m$ and $x^n\t_Dv_n$ are the leading terms of the elements $u$ and $v$, respectively, where $u_m,v_n\in V\backslash \{ 0\}$.  The $D$-module $V$ is a simple module. Therefore, $V=Du_m=Dv_n$, and so 
$$
u_m=\s^{-m}(d)v_n\;\; \text{for some element}\;\; d\in D.
$$
It follows from the equalities  $dx^{m-n}x^n\t_Dv_n=\Big( x^m\s^{-m}(d)+\cdots\Big)\t_Dv_n=x^m\t_D \s^{-m}(d)v_n+\cdots=x^m\t_Du_m+\cdots$ and the inequality  $l(u)\geq l(v)$ that 
$$l (u-dx^{m-n}v)<l (u),$$
and we use induction on the length $l(u)$ to finish the proof (alternatively, repeat the same argument several times).

2. Statement 2 follows from statement 1 and Proposition \ref{A28May26}.(1).

3.  The first part of statement 3 follows from statement 1 and Proposition \ref{A28May26}.(2). It remains to show that 
$\CG (N)\subseteq \CU\CA (L\t_DV)$. Suppose that $n=x^{i_1}\t_Dn_{i_1}+\cdots +x^{i_s}\t_Dn_{i_s}\in \CG (N)$,  where  $i_1<\cdots <i_s=m$ and $n_{i_1},\ldots , n_{i_s}\in V\backslash \{ 0\}$. Clearly, 
$$
\ga_{i_\nu}:=\lann_D(x^{i_\nu}\t_Dn_{i_\nu})\subseteq \ga_m:=\lann_D(x^m\t_Dn_m)\;\; \text{for all}\;\;\nu=1,\ldots , s
$$
since $l(an)<l(n)$ for all elements $a\in \ga_m$. 
Suppose that the element $n$ is not a uni-annihilator element. This means that  
 $s\geq 2$ and there is an index, say  $t$,  such that  $\ga_{i_{t-1}}\subset \ga_{i_t}=\ga_{i_{t+1}}=\cdots =\ga_m$. Take an element $a\in \ga_{i_t}\backslash \ga_{i_{t-1}}$. Then $an\neq 0$ and  $l (an)<l(n)$, a contradiction.
\end{proof}

{\bf Criterion for  $\N$-graded module to be a Euclidean module.}  Theorem \ref{11Jun26} is a criterion for $\N$-graded module to be  a Euclidean module. 

\begin{theorem}\label{11Jun26}%\marginpar{11Jun26}

Let $A=\bigoplus_{i\in \N}A_i$ be an $\N$-graded ring and  $M=\bigoplus_{i\in \N} M_i$ be an $\N$-graded $A$-module. Then the $A$-module $(M,\deg_M)$ is a Euclidean module iff for all natural numbers $i\leq j$ and elements $m_i\in M_i\backslash \{ 0\}$, $M_j=A_{j-i}m_i$. If the $A$-module $(M,\deg_M)$ is a Euclidean module then all $A_0$-modules $M_i$ are  simple modules. 

\end{theorem}

\begin{proof} $(\Rightarrow)$ Suppose that the $A$-module $M$ is a Euclidean module. 
 Given natural numbers $i\leq j$ and an element $m_i\in M_i\backslash \{ 0\}$, we have to show that  $M_j=A_{j-i}m_i$. Let $m_j\in M_j$. 
Since the $A$-module $M$ is a Euclidean module, $m_j=qm_i+r$ for some elements $q=\sum_{k\in \N}q_k\in A$, where $q_k\in A_k$, and $r\in M$ such that $\deg_M(r)<\deg_M(m_i)=i$.
 Comparing the homogeneous components of degree $j$ in the equality  $m_j=qm_i+r$, we have the equality $m_j=q_{j-i}m_i$. Therefore, 
$M_j=A_{j-i}m_i$.

$(\Leftarrow)$ Suppose that for all natural numbers $i\leq j$ and elements $m_i\in M_i\backslash \{ 0\}$, $M_j=A_{j-i}m_i$. Given elements $u,v\in M$, we have to show that $u=qv+r$ for some elements $q\in A$ and $r\in M$ such that $\deg_M(r)<\deg_M(v)$. Suppose that $u=m_j+\cdots$ and $v=m_i+\cdots$ where the three dots denote smaller terms, $m_j\in M_j$ and $m_i\in M_i$. If $j<i$ then $u=0v+u$.

Suppose that $i\leq j$. Then, by the assumption, $m_j=am_i$ for   some element $ 
a\in A_{j-i}$. Therefore, $\deg_M(u-av)<j=\deg_M(u)$, and so $u=av+u_1$ for some element $u_1\in M$ such that $\deg_M(u_1)<j$. By the induction on $j$, 
$u_1=bv+r$ for some elements $b\in A$ and $r\in M$ such that $\deg_M(r)<\deg_M(v)=i$. 
Now, $u=qv+r$ where $q:=a+b$. This finishes the proof of the  first statement of the theorem. Now, the second statement of the theorem (about simplicity) follows from the first  one. 
\end{proof}

\begin{definition}
Let $S=\N, \Z$, $A=\bigoplus_{i\in S}A_i$ be an $S$-graded  ring and  $M=\bigoplus_{i\in S} M_i$ be an $S$-graded  $A$-module. 
An element $m=\sum_{i\in S} m_i\in M$, where $m_i\in M_i$,  is called a {\bf uni-annihilator element} if $\lann_{A_0}(m)=\lann_{A_0}(m_i)$ for all nonzero elements  $m_i$. The set of all uni-annihilator elements of the $A$-module $M$  is denoted by $\CU\CA (M)$.
\end{definition}

\begin{lemma}\label{a13Jun26}%\marginpar{a13Jun26}
Let $A=\bigoplus_{i\in \N}A_i$ be an $\N$-graded  ring and  $M=\bigoplus_{i\in \N} M_i$ be an $\N$-graded Euclidean $A$-module with the Euclidean map $\deg_M$. If $N$ is a nonzero submodule of $M$ then the set of 1-generators of the $A$-module $N$ contains  the non-empty  set $\CG (N):= \{ n\in N\, | \, \deg_M(n)=l \}$, where $l:=\min \{ \deg_M (n')\, | \, 0\neq n'\in N \}$, and 
$\CG (N)\subseteq \CU\CA (M)$. 

\end{lemma}

\begin{proof} By Proposition \ref{A28May26}, the set of 1-generators of the $A$-module $N$ contains  the non-empty  set $\CG (N)$. Let us show that $\CG (N)\subseteq \CU\CA (M)$. Suppose that $m=m_{i_1}+\cdots +m_{i_s}\in \CG (N)$   where $m_{i_k}\in M_k\backslash \{ 0\}$ and  $i_1<\cdots <i_s$. Clearly, 
$$
\ga_{i_\nu}:=\lann_{A_0}(m_{i_\nu})\subseteq \ga_{i_s}:=\lann_{A_0}(m_{i_s})\;\; \text{for all}\;\;\nu=1,\ldots , s
$$
since $\deg_M(am)<\deg_M(m)$ for all elements $a\in \lann_{A_0}(m_{i_s})$.
Suppose that the element $m$ is not a uni-annihilator element. This means that  
 $s\geq 2$ and there is an index, say  $t$,  such that  $\ga_{i_{t-1}}\subset \ga_{i_t}=\ga_{i_{t+1}}=\cdots =\ga_m$. Take an element $a\in \ga_{i_t}\backslash \ga_{i_{t-1}}$. Then $\deg_M (am)=i_{t-1}<i_s$, a contradiction.
 \end{proof}

{\bf Criterion for  $\Z$-graded module to be a Euclidean module.}  Theorem \ref{12Jun26} is a criterion for  $\Z$-graded graded module to be  a Euclidean module. 

\begin{theorem}\label{12Jun26}%\marginpar{12Jun26}

Let $A=\bigoplus_{i\in \Z}A_i$ be an $\Z$-graded ring and  $M=\bigoplus_{i\in \Z} M_i$ be a $\Z$-graded $A$-module where $l_M:M\ra \N$ is the length map on $M$. Then the $A$-module $(M,l_M)$ is a Euclidean module iff for all natural numbers $i\leq j$ and elements $m_i\in M_i\backslash \{ 0\}$, $M_j=A_{j-i}m_i$. If the $A$-module $(M,l_M)$ is a Euclidean module then all $A_0$-modules $M_i$ are  simple modules. 

\end{theorem}

\begin{proof} $(\Rightarrow)$ Suppose that the $A$-module $M$ is a Euclidean module. 
 Given integers $i\leq j$ and an element $m_i\in M_i\backslash \{ 0\}$, we have to show that  $M_j=A_{j-i}m_i$. Let $m_j\in M_j$. 
Since the $A$-module $M$ is a Euclidean module, $m_j=qm_i+r$ for some elements $q=\sum_{k\in \Z}q_k\in A$, where $q_k\in A_k$, and $r\in M$ such that $l_M(r)<l_M(m_i)=0$, i.e. $r=0$, and so $m_j=qm_i$. 
$M_j=A_{j-i}m_i$.

$(\Leftarrow)$ Suppose that for all integers $i\leq j$ and elements $m_i\in M_i\backslash \{ 0\}$, $M_j=A_{j-i}m_i$. Given elements $u,v\in M$, we have to show that $u=qv+r$ for some elements $q\in A$ and $r\in M$ such that $l_M(r)<l_M(v)$. Suppose that $u=m_j+\cdots$ and $v=m_i+\cdots$ where  $m_j\in M_j$,  $m_i\in M_i$ and the three dots denote smaller terms with respect to the graded degree, i.e. the elements $m_j$ and $m_i$ are the leading terms of the elements $u$ and $v$, respectively. If $l_M(u)<l_M(v)$ then $u=0v+u$.

Suppose that $l_M(u)\geq l_M(v)$. Then, by the assumption, $m_j=am_i$ for   some element $ 
a\in A_{j-i}$. Therefore, $l_M(u-av)<l_M(u)$, and so $u=av+u_1$ for some element $u_1\in M$ such that $l_M(u_1)<l_M(u)$. Now, we can repeat the same argument but for the element 
 $u_1$ rather than $u$ and then keep  repeating it finitely many times we obtain homogeneous elements, say  $a, \ldots,  b\in A$, such that 
 $l_M(u-(a+\cdots +b)v)<l_M(v)$. Now, $u=qv+r$ where $q:=a+\cdots +b$ and $r:=u-(a+\cdots +b)v$.
 
  This finishes the proof of the  first statement of the theorem. Now, the second statement of the theorem (about simplicity) follows from the first  one. 
 \end{proof}

\begin{lemma}\label{b13Jun26}%\marginpar{b13Jun26}
Let $A=\bigoplus_{i\in \Z}A_i$ be a $\Z$-graded  ring and  $M=\bigoplus_{i\in \Z} M_i$ be a $\Z$-graded Euclidean $A$-module with the Euclidean map $l_M$. If $N$ is a nonzero submodule of $M$ then the set of 1-generators of the $A$-module $N$ contains  the non-empty  set $\CG (N):= \{ n\in N\, | \, l_M(n)=l \}$, where $l:=\min \{ l_M (n')\, | \, 0\neq n'\in N \}$, and 
$\CG (N)\subseteq \CU\CA (M)$. 

\end{lemma}

\begin{proof} By Proposition \ref{A28May26}, the set of 1-generators of the $A$-module $N$ contains  the non-empty  set $\CG (N)$. Let us show that $\CG (N)\subseteq \CU\CA (M)$. Suppose that $m=m_{i_1}+\cdots +m_{i_s}\in \CG (N)$   where $m_{i_k}\in M_k\backslash \{ 0\}$ and  $i_1<\cdots <i_s$. Clearly, 
$$
\ga_{i_\nu}:=\lann_{A_0}(m_{i_\nu})\subseteq \ga_{i_s}:=\lann_{A_0}(m_{i_s})\;\; \text{for all}\;\;\nu=1,\ldots , s
$$
since $l_M(am)<l_M(m)$ for all elements $a\in \lann_{A_0}(m_{i_s})$.
Suppose that the element $m$ is not a uni-annihilator element. This means that  
 $s\geq 2$ and there is an index, say  $t$,  such that  $\ga_{i_{t-1}}\subset \ga_{i_t}=\ga_{i_{t+1}}=\cdots =\ga_m$. Take an element $a\in \ga_{i_t}\backslash \ga_{i_{t-1}}$. Then $l_M (am)=i_{t-1}<i_s$, a contradiction.
 \end{proof}

{\bf Filtered rings and Euclidean  filtered modules.}  Let $B=\bigcup_{i\in N}B_i$ be a filtered ring, ${\rm gr}(B)=\bigoplus_{i\in \N}B_i/B_{i-1}$ be the associated graded ring of $B$, $M=\bigcup_{i\in \N}M_i$ be a filtered $R$-module and  ${\rm gr}(M)=\bigoplus_{i\in \N} M_i/M_{i-1}$ be the associated graded ${\rm gr}(B)$-module. Let 
$\deg_{{\rm gr}(M)}\Big(\sum_{i\geq 0}m_i\Big):=\max \{ i\in \N \, | \, m_i\neq 0\}$ and $m_i\in M_i/M_{i-1}$. 
Let us consider the maps
%\marginpar{gr-map-B,gr-map-M}
\begin{align}
{\rm gr}:& B \ra {\rm gr}(B), \;\; b\mapsto b+B_{i-1}\in B_i/B_{i-1}\;\; {\rm where}\;\; i=i(b):=\min \{ i'\in \N\, | \, b\in B_{i'}\}, \label{gr-map-B}\\
l_{{\rm gr}(M)}&: {\rm gr}(M) \ra {\rm gr}(M), \;\; m=\sum_{i=0}^nm_i\mapsto m_n \;\; {\rm where}\;\; n=\deg_{{\rm gr}(M)}(m). \label{gr-map-M}
\end{align}
These maps are not  additive map. For an element $m\in M$, the {\em filtration degree} of $m$,  $\deg_M(m)$,  is a unique number $i=i(m)$ such that $m\in M_i\backslash M_{i-1}$, i.e. $\deg_M (m):=\deg_{{\rm gr}(M)}({\rm gr}(m))$. In the  particular case $M=B$, $\deg_B (b):=\deg_{{\rm gr}(B)}({\rm gr}(b))$ for all elements $b\in B$.

\begin{proposition}\label{FiltEucl-28May26}%\marginpar{FiltEucl-28May26}
Let $B=\bigcup_{i\in N}B_i$ be a filtered ring, ${\rm gr}(B)=\bigoplus_{i\in \N}B_i/B_{i-1}$ be the associated graded ring of $B$, $M=\bigcup_{i\in \N}M_i$ be a filtered $R$-module and  ${\rm gr}(M)=\bigoplus_{i\in \N} M_i/M_{i-1}$ be the associated graded ${\rm gr}(B)$-module. Suppose that the ${\rm gr}(B)$-module  ${\rm gr}(M)$ is a Euclidean module with Euclidean function $\deg_{{\rm gr}(M)}$.
% where  $\deg_{{\rm gr}(M)}\Big(\sum_{i\geq 0}m_i\Big):=\max \{ i\in \N \, | \, m_i\neq 0\}$ and $m_i\in M_i/M_{i-1}$. 
 Then:

\begin{enumerate}

\item  The $B$-module $M$ is a Euclidean module with Euclidean function $\deg_M$.  Furthermore, for all elements $u,v\in M$ such that $\deg_M(u)\geq \deg_M(v)$ there exist elements $q\in B$ and $r\in M$ such that $\deg_B (q)=\deg_M(u)-\deg_M(v)$ and $ \deg_M(r)<\deg_M(v)$. 

\item The  $B$-module $M$ is a principal submodule  module.

\item If $N$ is a nonzero submodule of $M$ then the set of 1-generators of the $B$-module $N$ contains  the non-empty  set $\CG (N):= \{ n\in N \, | \, \deg_M(n)=l(N) \}$, where $l(N):=\min \{ \deg_M(n')\, | \, 0\neq n'\in N \}$. Similarly, the set of 1-generators of the ${\rm gr}(B)$-module ${\rm gr}(N)$ contains  the non-empty  set $\CG ({\rm gr}(N)):= \{ \nu\in {\rm gr}(N) \, | \, \deg_{{\rm gr}(M)}(\nu)=l({\rm gr}(N)) \}$, where $l({\rm gr}(N)):=\min \{ \deg_M(\nu')\, | \, 0\neq \nu'\in {\rm gr}(N) \}$. Then $l(N)=l({\rm gr}(N))$ and ${\rm gr}(\CG (N))=\CG ({\rm gr}(N))\subseteq \CU\CA ({\rm gr}(N))$.

\end{enumerate}
\end{proposition}

\begin{proof} In this proof,  $d:=\deg_M$,  $\d:=\deg_{{\rm gr}(M)}$ and  $\bu :={\rm gr }(u)$ and $\bq :={\rm gr}(q)$ for  element elements $u\in M$ and $q\in B$.

1. Let $u,v\in M$. If $d(u)< d(v)$ then $u=0\cdot v+u$. 

Suppose that $d(u)\geq  d(v)$. We  have to show that there exist elements $q\in B$ and $r\in M$ such that $\deg_B (q)=d(u)-d(v)$ and $ d(r)<d(v)$. Since the ${\rm gr}(B)$-module  ${\rm gr}(M)$ is a Euclidean module with Euclidean function $\d$, for the homogeneous elements $\bu,\bv\in {\rm gr}(M)$ there exists a homogeneous element $\oa\in {\rm gr}(B)$ such that $\bu = \oa\bv$, by Theorem \ref{11Jun26}. The ${\rm gr}(B)$-module  ${\rm gr}(M)$ is $\N$-graded. Therefore,  by comparing the elements of the  graded degree $\d (\bu)$ in the equality 
 $\bu = \oa\bv$ we may assume that $\deg_{{\rm gr}(B)}(\oa)=\d (\bu)-\d(\bv)$. Then $d(u_1)<d(u)$ where $u_1:=u-av\in M$.

Now, we can repeat the same argument but for the element 
 $u_1$ rather than $u$ and then keep  repeating it finitely many times we obtain  elements, say  $a, \ldots,  b\in B$, such that $d(a)>\cdots >d(b)$ and 
 $d(u-(a+\cdots +b)v)<d(v)$. Now, $u=qv+r$ where $q:=a+\cdots +b$ and $r:=u-(a+\cdots +b)v$.  Clearly, $d(q)=\d (a)=\d(\bu)-\d(\bv)=d(u)-d(v)$.

2. Statement 2 follows from statement 1 and Proposition \ref{A28May26}.(1).

3. Clearly, $l(N)=l({\rm gr}(N))$. Therefore,   ${\rm gr}(\CG (N))\subseteq \CG ({\rm gr}(N))$.  Now, the inclusion $\CG ({\rm gr}(N))\subseteq \CU\CA ({\rm gr}(N))$ follows from Lemma \ref{a13Jun26}.

\end{proof}

%%%%%%%%%%%%%   Section 4    %%%%%%%%

\section{Submodules of direct sums of Euclidean modules and reduced echelon forms}\label{SUB-FACTOR-MOD} %\marginpar{SUB-FACTOR-MOD}

{\bf Reduced echelon form.}  Let $M=\bigoplus_{i=1}^nM_i$ be a direct sum of Euclidean $R$-modules $M_i$ and 
 $M_{m,n}(M_1,\ldots , M_n)$ be the set of all $m\times n$ matrices
$$
A = \begin{pmatrix}
a_{11} & a_{12} & \cdots & a_{1n} \\
a_{21} & a_{22} & \cdots & a_{2n} \\
\vdots & \vdots & \ddots & \vdots \\
a_{m1} & a_{m2} & \cdots & a_{mn}
\end{pmatrix}\;\; \text{such that $a_{ij}\in M_j$ for all $i$ and $j$.}
$$
On the rows $a_1, \ldots , a_m$ of the matrix $A$, we can do two types of operations -- {\em elementary row operations}:
\begin{eqnarray*}
 &{\rm (ER1)}& a_i\mapsto a_i+\sum_{j\neq i}r_ja_j\;\; {\rm where}\;\; r_j\in R,  \\
 &{\rm (ER2)}& a_i\mapsto a_j, \;\; a_j\mapsto a_i.
\end{eqnarray*}
For a fixed row, (ER1) adds a linear combination of the other rows to it, and (ER2) swaps two rows. 
 These operations are reversible and  they change the generators $a_1, \ldots a_m$ of the $R$-submodule $N=N(a_1, \ldots , a_m)=\sum_{i=1}^mRa_i$ of $M$ to another set of generators that also contains $m$ elements.    
Two matrices $A, A'\in  M_{m,n}(M_1, \ldots , M_n)$ are called {\bf left-row equivalent}, $A\sim A'$, if one of them can be obtained from the other by finitely many elementary row operations (ER1) and (ER2). The relation $\sim$ is an equivalence relation.  If $A\sim A'$ then $N(A)=N(A')$.

The following lemma, which is a generalization of the Euclidean division algorithm in the case of two elements,  is used in many proofs.

\begin{lemma}\label{EuclDivAlg}%\marginpar{EuclDivAlg}

Let $M$ be a Euclidean $R$-module and $A=(a_{11} \, a_{12}\,\ldots \,  a_{1m})^t\in M_{m,1}(M)$. Then $A\sim (\gcd(a_{11},  \ldots ,  a_{1m}),0,\ldots , 0)^t$.
\end{lemma}

\begin{proof} To prove the lemma, we use induction on $m\geq 2$. The initial case, $m=2$ is basically the classical Euclidean division algorithm:  
\begin{eqnarray*}
 a_{11}&=&q_1a_{12}+r_1, \;\; d_M(r_1)<d_M(a_{12}), \\
 a_{12}&=&q_2r_1+r_2, \;\;d_M(r_2)<d_M(r_1),  \\
r_1&=&q_3r_2+r_3, \;\; d_M(r_3)<d_M(r_2),   \\
&\cdots& \\
 r_i&=&q_{i+2}r_{i+1}+r_{i+2}, \;\; d_M(r_{i+2})<d_M(r_{i+1}),\\
 r_{i+1}&=&q_{i+3} r_{i+2}
\end{eqnarray*}
for some elements $q_1,\ldots , q_{i+3}\in R$ and $r_1,\ldots , r_{i+2}\in M$. It  follows that $a_{11}=\alpha r_{i+2}$,  $a_{12}=\beta r_{i+2}$ and $r_{i+2}=\g a_{11}+\d a_{12}$ for some elements $\alpha,\beta , \g , \d\in R$. By Proposition \ref{A14Jul26}, $r_{i+2}=\gcd(a_{11},  \ldots ,  a_{1m})$. 
Furthermore, $A\sim (r_{i+2},0)^t$ (use the equalities above), as required.

Suppose that $m>2$ and the lemma is true for all $m'<m$. By induction, $(a_{12}\, \ldots\, a_{1m})^t\sim  (\gcd(a_{12},  \ldots ,  a_{1m}),0,\ldots , 0)^t$. Therefore, 
    \begin{eqnarray*}
A &\sim &(a_{11},\gcd(a_{12},  \ldots ,  a_{1m}),0,\ldots , 0)^t\sim (\gcd(a_{11},  \gcd( a_{12},\ldots ,  a_{1m})),0,\ldots , 0)^t\\
&=&(\gcd(a_{11},  \ldots ,  a_{1m}),0,\ldots , 0)^t.
\end{eqnarray*}
\end{proof}
\begin{definition}[{\bf Reduced echelon form}]
Let $M=\bigoplus_{i=1}^nM_i$ be a direct sum of Euclidean $R$-modules $M_i$.
A matrix $A \in M_{m,n}(M_1, \ldots , M_n)$ is said to be in {\bf (left) reduced echelon form}
if the following conditions hold:
\begin{enumerate}
    \item Every nonzero row of $A$ has a leftmost nonzero entry (a \emph{pivot}).
    \item If the pivot of row $i$ is in column $j$, then all entries in column $j$
  below  the pivot are zero  and the degrees of entries above the pivot are strictly less that the degree of the pivot (in the Euclidean  module $M_j)$.
    %\item The pivot in each nonzero row is a \emph{minimal} element of its left ideal
         % with respect to the Euclidean function.
    \item The pivot of each row lies strictly to the right of the pivot of the row above it.
    %\item Every pivot is left-invertible (and may be scaled to $1$ if $R$ has $1$).
    
\end{enumerate}
The set of pivots of $A$ is denoted by $\mP (A)$. 
\end{definition}

\begin{theorem}[{\bf Existence of reduced echelon form}] \label{REForm}%\marginpar{REForm}

Let $M=\bigoplus_{i=1}^nM_i$ be a direct sum of Euclidean $R$-modules $M_i$. Then every nonzero  matrix
$A \in M_{m,n}(M_1, \ldots ,M_n)$ is left-row equivalent to a matrix in
left reduced echelon form.

% where each pivot is  the greatest common divisor of all elements of the matrix $A$ in the column of the pivot. 

\end{theorem}

\begin{proof} To prove the theorem, we use induction on $n$. The case $n=1$ is Lemma \ref{EuclDivAlg}.

%obvious since  by using the division algorithm we see that the $m\times 1$ matrix $A=(a_{11},\ldots , a_{m1})^t$, where `$t$' means the transposition, is equivalent to the $m\times 1$ matrix $(\gcd(a_{11},\ldots , a_{m1}),0,\ldots , 0)^t$.

Suppose that $n>1$ and the result is true for all $n'<n$. Choose the first nonzero column of the matrix $A$, say $(a_{1j_1},\ldots , a_{mj_1})^t$. Then by applying  elementary row operations we obtain a matrix, say $B$,  where the column  indexed by $j_1$ is equal to $(\gcd (a_{1j_1},\ldots , a_{mj_1}),0,\ldots , 0)^t$. Let $B'$ be the matrix  obtained from  $B$ by deleting the first row and the column  indexed by $j_1$. 
Inductively, $B'$ 
 has a reduced echelon form. After reinstating the deleted row and column, we adjust the entries in the first row and above the pivots to ensure their degrees are strictly smaller than the degrees of the respective pivots (by using the division algorithm).
\end{proof}

{\bf The strong reduced echelon form.}  
Let $M=\bigoplus_{i=1}^nM_i$ be a direct sum of Euclidean $R$-modules $M_i$. The $R$-module $M$ admits a strictly descending of chain of $R$-submodules 
$$
M=M_{\geq 1}\supset \cdots  \supset M_{\geq i}\supset \cdots  \supset M_{\geq n}\supset M_{\geq n+1}:=\{ 0\}\;\; {\rm where}\;\;   M_{\geq i}:=M_i\oplus\cdots \oplus M_n.
$$
The descending chain defines the $R$-module filtration $\{ M_{\geq i}\}$ on $M$. The associated graded $R$-module ${\rm gr}(M):=\bigoplus_{i=1}^n {\rm gr}(M)_i$, where ${\rm gr}(M)_{\geq i}:=M_i/M_{\geq i+1}\simeq M_i$, is isomorphic to the $R$-module $M$. Let $N$ be an $R$-submodule of $M$. Then the $R$-modules $N$ and $\bM :=M/N$ admit the induced filtrations $ \{ N_{\geq i}:=N\cap M_{\geq i} \}$ and   $ \{ \bM_i:= (N +M_{\geq i})/N \}$, respectively. Then the associated graded $R$-module ${\rm gr}(N):=\bigoplus_{i=1}^n{\rm gr}(N)_i$ is a submodule of the $R$-module $M\simeq {\rm gr}(M)$ where ${\rm gr}(N)_i:=N_{\geq i}/N_{\geq i+1}$. For each natural number $n\geq 1$, let  $[n]:=\{ 1, \ldots , n\}$.

\begin{definition}
The set $\mP_s (N):=\mP_s (N,M):=\{ i\in [n]\, | \, {\rm gr}(N)_i\neq 0 \}$ is called the {\bf strong pivot set} and the elements of the set $\mP_s (N)$ are called the set of {\bf strong pivots} of $N$.
\end{definition}

The associated graded $R$-module ${\rm gr}(\bM):=\bigoplus_{i=1}^n{\rm gr}(\bM)_i$  
is a direct sum of the $R$-modules 
\begin{eqnarray*}
{\rm gr}(\bM)_i &:=& \bM_{\geq i}/\bM_{\geq i+1}=(M_{\geq i}+N)/(M_{\geq i+1}+N)\simeq M_{\geq i}/(M_{\geq i+1}+M_{\geq i}\cap N)\\
&\simeq & \Big( M_{\geq i}/M_{\geq i+1}\Big)/\Big((M_{\geq i+1}+M_{\geq i}\cap N)/M_{\geq i+1}\Big)\simeq M_i/\Big( M_{\geq i}\cap N/M_{\geq i+1}\cap N\Big)\\
&\simeq& M_i/{\rm gr}(N)_i.
\end{eqnarray*}
As a result, there is the short exact sequence of graded $R$-modules:
%\marginpar{ses-grN}
\begin{equation}\label{ses-grN}
0\ra {\rm gr}(N)\ra {\rm gr}(M)\simeq M\ra {\rm gr}(\bM)\ra 0.
\end{equation}
By Theorem \ref{14Jun26}.(1), each  Euclidean $R$-module $M_i$ is a Noetherian $R$-module. Thus,  the $R$-module $M=\bigoplus_{i=1}^nM_i$  is also a Noetherian module. In particular, the  $R$-submodule $N=\sum_{i=1}^ma_i$ of $M$ is a finitely generated $R$-module where the elements $a_i=(a_{i1},\ldots , a_{in})\in M$ are generators of $N$ and $a_{ij}\in M_j$. As above, the $R$-module $N$ is determined by the $m\times n$ matrix $A=(a_{ij})\in M_{m,n}(M_1, \ldots , M_n)$. 
The number  $m$ is not fixed. We can always add a zero element of $N$ to the set of generators of $N$ to increase the number $m$. 
 So, if necessary, we may assume that our matrix $A$ (that depends on the generating set) contains zero row. To accommodate this situation, we consider the set $M_{fin,n}(M_1, \ldots, M_n)$ of all $\infty\times n$ matrices with only finitely many nonzero rows, i.e. 
 $$
 M_{fin,n}(M_1, \ldots, M_n)=\bigcup_{m\geq 1}M_{m,n}(M_1, \ldots, M_n).
 $$
 Then the operations (ER1) and (ER2) act on the set $ M_{fin,n}(M_1, \ldots, M_n)$.  
 In addition to the elementary row operations (ER1) and (ER2), on the rows $a_1, \ldots , a_m$ of the matrix $A$, we  perform a new  elementary row operation -- adding to the  zero row a linear combination of the  other rows:
\begin{eqnarray*}
{\rm (ER0)}\;\;&\;\;\;\;\; & 0\mapsto 0+\sum_{i=1}^m r_ia_i\;\; {\rm where}\;\; r_i\in R.
\end{eqnarray*}
On the level of generators, it means that we replace the  generating set $a_1, \ldots , a_m$ of the $R$-module $N$ by the generating set  $a_1, \ldots , a_m,\sum_{i=1}^mr_ia_i$.  
In the classical situation when the ring $R$ is a left Euclidean domain and all  $M_i=R$, this operation is redundant because nonzero elements of the ring $R$ act injectively on the modules $M_i=R$. In the general situation, this is not the case and this is the main reason that we have to add the operation (ER0) in order to obtain the strong reduced echelon form which is a more refined version of the reduced echelon form and it contains more information about the module as we will see later. The element $\sum_{i=1}^mr_ia_i$ can be deleted from the list of generators $a_1, \ldots , a_m,\sum_{i=1}^mr_ia_i$ by using the operation (ER1). 
 
Two matrices $A, A'\in  M_{m,n}(M_1, \ldots , M_n)$ and $A'\in  M_{m',n}(M_1, \ldots , M_n)$ are called {\bf strongly left-row equivalent}, $A\sim_s A'$, if one of them can be obtained from the other by finitely many elementary row operations (ER0), (ER1) and (ER2) (`$s$' stands for `strongly'). The relation $\sim_s$ is an equivalence relation.   If $A\sim_s A'$ then $N(A)=N(A')$ where $N(A)=N(a_1, \ldots , a_m)=\sum_{i=1}^mRa_i$ and $N(A')=N(a_1', \ldots , a_{m'}')=\sum_{i=1}^{m'}Ra_i'$.

\begin{definition}[{\bf Strong reduced echelon form}]
Let $M=\bigoplus_{i=1}^nM_i$ be a direct sum of Euclidean $R$-modules $M_i$, $0\neq A =(a_1, \ldots , a_s)^t\in M_{s,n}(M_1, \ldots , M_n)$, where $a_1, \ldots , a_s$ are the rows of the matrix $A$, $N(A)=\sum_{i=1}^sRa_i$  and $\mP_s(N(A))=\{ j_1, \ldots , j_s\}$ where $j_1< \cdots < j_s$. The  matrix $A$ is said to be in {\bf strong (left) reduced echelon form} if 
\begin{enumerate}

\item  for each $\nu = 1, \ldots, s$, $N_{\geq  j_\nu}=\sum_{\mu =\nu}^sRa_\mu$ and the element $a_{\nu j_\nu}\in M_{j_\nu}$ has the least degree in the set of nonzero elements of the $R$-module  ${\rm gr}(N)_{j_\nu}=Ra_{\nu j_\nu}$, and 

\item 
  the degrees of the  elements $a_{1j_\nu}, \ldots , a_{\nu-1, j_\nu}\in M_{j_\nu}$ are  strictly less that the degree of the element $a_{\nu j_\nu}\in M_{j_\nu}$.

\end{enumerate}
\end{definition}
It follows from the definition that  if the matrix $A$ is in strong reduced echelon form then it is also in reduced echelon form. 

Lemma \ref{a26Jul26} shows that if the matrix $A$ is in strong reduced echelon form then it contains a lot of information about the $R$-submodule $N(A)$ of $M$.

\begin{lemma}\label{a26Jul26}%\marginpar{a26Jul26}
Let $M=\bigoplus_{i=1}^nM_i$ be a direct sum of Euclidean $R$-modules $M_i$, $0\neq A =(a_1, \ldots , a_s)^t\in M_{s,n}(M_1, \ldots , M_n)$ be a matrix which is  in  strong  reduced echelon form   with $\mP_s(N(A))=\{ j_1, \ldots , j_s\}$ where $j_1< \cdots < j_s$. Then the matrices $0\neq  (a_\nu ,a_{\nu+1}, \ldots , a_s)^t\in M_{s-\nu+1,n}(M_{j_\nu}, \ldots , M_n)$,  where $\nu = 1, \ldots, s$, are in  strong  reduced echelon form.
\end{lemma}

\begin{proof} The lemma follows from the definition of the strong reduced echelon form. 
\end{proof}

 By Lemma \ref{a26Jul26}, the strong  reduced echelon form of $A$ encodes not only the generators of the $R$-module $N(A)$ but also generators of all the intersections $N\cap M_{\geq i}$ where $i=1, \ldots , n$. We will see that  each reduced echelon form of the matrix $A$ can be transformed (using (ER0)--(ER2)) to  a strong  reduced echelon form of it (Theorem \ref{STR-REForm}).
 In the case of a left Euclidean domain where  all $M_i=R^n$, the two forms coincide, and in this situation the main purpose of the reduced echelon form is that it also encodes the generators of the intersections described above.

%\begin{theorem}[{\bf Existence of strong reduced echelon form}] \label{STR-REForm}%\marginpar{STR-REForm}

%Let $M=\bigoplus_{i=1}^nM_i$ be a direct sum of Euclidean $R$-modules $M_i$, $0\neq A \in M_{m,n}(M_1, \ldots ,M_n)$ and $\mP_s(N(A))=\{ j_1, \ldots , j_s\}$  where $j_1< \cdots < j_s$. Then the  matrix
%$A$ is strongly left-row equivalent to a matrix in strong 
%   reduced echelon form.
%\end{theorem}

\begin{proof}[{\bf Proof of Theorem \ref{STR-REForm}}] The  nonzero  matrix
$A \in M_{m,n}(M_1, \ldots ,M_n)$ yields the $R$-submodule $N:=N(A)=\sum_{i=1}^mRa_i$ of $M$ where $a_i$ are the rows of $A$. Recall that  $\mP_s(N)=\{ j_1, \ldots , j_s\}$ where $j_1< \cdots < j_s$. By the definition of the  set $\mP_s(N)$, ${\rm gr}(N)=\bigoplus_{\nu =1}^s {\rm gr}(N)_{j_\nu}$ for some $R$-submodules $ {\rm gr}(N)_{j_\nu}$ of $M_{j_\nu}$. 
 Since the $R$-module $M_{j_1}$ is Euclidean and  the  $R$-module $ {\rm gr}(N)_{j_1}$ is an   $R$-submodule of $M_{j_1}$,  $ {\rm gr}(N)_{j_1}=Rb_{1 j_1}$ for some element $b_{1 j_1}\in \CG ({\rm gr}(N)_{j_1})\subseteq  M_{j_1}$, by  Proposition \ref{A28May26}.(1) (using (ER0)-(ER2)). 
The elements $b_{1 j_1}\in M_{j_1}$ can be chosen to have the least degree in the $R$-submodule $ {\rm gr}(N)_{j_1}$ (using (ER1) and (ER2)).  As a result (using the above operations (ER0)-(ER2)), for the  element  $b_{1 j_1}$, we obtain  an element 
$$
b_1=(0,\ldots , 0,b_{1 j_1}, \ldots , b_{1n})\in M_{\geq j_1},
$$
 i.e. $b_1 \equiv b_{1 j_1} \mod  M_{\geq j_1+1}$.  Let $A':= (b_1, a_2', \ldots , a_{m'}')\in M_{m',n}(M_1, \ldots, M_n)$ be the matrix that is obtained from the matrix $A$ by the above  operations (ER0)--(ER2). Using the fact that ${\rm gr}(N)_{j_1}=Rb_{1j_1}$ and the operations (ER1), we may assume that $a_2', \ldots , a_{m'}'\in N_{\geq j_1+1}$ or, equivalently,  $a_2', \ldots , a_{m'}'\in N_{\geq j_2}$ (since $N_{\geq j_1+1}=\cdots =N_{\geq j_2}$). 
 In particular, 
 $$
 A\sim_sA'\;\; {\rm  and}\;\;N(A)=N(A').
 $$  
 Since $N(A)=N(A')$, we have the equality $\mP_s(N(A'), M) =\mP_s(N(A), M)=\{ j_1, \ldots , j_s\}$. By the choice of the index $j_2$,
$$ 
\ga_{1j_1}:=\ann_R(b_{1j_1}):=\{r\in R\, | \, rb_{1j_1}=0 \}\subseteq \ann_R(b_{1j})\;\; \text{for all}\;\; j_1\leq  j<j_2.
$$
Therefore, 
$$
N_{\geq j_1+1}=\cdots =N_{\geq j_2-1}=N_{\geq j_2}=\ga_{1j_1}b_1+N\Big({\rm Rows}(A')\backslash \{ b_1\}\Big)\subseteq M_{\geq j_2}:=\bigoplus_{i=j_2}^nM_i.
$$
By Theorem \ref{14Jun26}.(1), the $R$-module $M_{\geq j_2}$ is a Noetherian module. Hence, its submodule  $\ga_{1j_1}b_1$ is a finitely generated $R$-module. Let $r_1b_1, \ldots , r_tb_1$ be its  set of generators where $r_1, \ldots , r_t\in R$. Then 
$$N_{\geq j_2}=\sum_{i=1}^tr_ib_1+N\Big({\rm Rows}(A')\backslash \{ b_1\}\Big)=N(A'')\subseteq M_{\geq j_2}
$$
 where $A''$ is the matrix with rows  ${\rm Rows}(A')\backslash \{ b_1\}$ and $r_1b_1,\ldots , r_tb_1$. Let $B$ be the matrix with rows   ${\rm Rows}(A')$ and $r_1b_1,\ldots , r_tb_1$, i.e. 
\[
B=\begin{pmatrix}
b_1\\[4pt]
A''
\end{pmatrix}.
\]
 Then $A'\sim_sB$. Since $N(A'')=N_{\geq j_2}\subseteq M_{\geq j_2}$, by induction,  $A''\sim_s C$ for some matrix $C=(c_{ij})\in M_{l, n-j_2+1}(M_{j_2},\ldots , M_n)$ which is in strong reduced echelon form. Therefore,
 
\[
B=\begin{pmatrix}
b_1\\[4pt]
A''
\end{pmatrix}\sim_s B':=\begin{pmatrix}
b_1\\[4pt]
C
\end{pmatrix}.
\]
Notice that $\mP_s(C)=\mP_s(A'')=\mP_s(N_{\geq j_2})=\{ j_2, \ldots , j_s\}= \mP_s(N(A))\backslash \{ j_1\}$. Now, using the operation (ER1), the degree of the element $b_{1j_\nu}\in M_{j_1}$, where $\nu =2, \ldots , s$, can be made smaller than the degree of the element $c_{\nu j_\nu}$ (where $C=(c_{ij})$). As the result, we obtain a matrix, say  $B''$,  which is in strong reduced echelon form and $B''\sim_sB'\sim_sA$, and so $A\sim_s B''$. 
\end{proof}

\begin{corollary} \label{aSTR-REForm}%\marginpar{aSTR-REForm}

Let $M=\bigoplus_{i=1}^nM_i$ be a direct sum of Euclidean $R$-modules $M_i$, $0\neq A \in M_{m,n}(M_1, \ldots ,M_n)$ and $\mP_s(N(A))=\{ j_1, \ldots , j_s\}$  where $j_1< \cdots < j_s$. Then:

\begin{enumerate}

\item  If $A\sim_s B$  and $A\sim_s B'$, where $B$ and $B'$ are matrices in strong    reduced echelon form, then   for all  $\nu =1, \ldots , s$,  
$
Rb_{\nu j_\nu}=R  
b_{\nu j_\nu}'={\rm gr}(N(A))_{j_\nu}\subseteq M_{j_\nu}$ and the elements $b_{\nu j_\nu}$ and   $b_{\nu j_\nu}'$ are elements of least degree in the Euclidean $R$-submodule ${\rm gr}(N(A))_{j_\nu}$ of $M_{j_\nu}$ where $b_\nu =(0, \ldots , 0, b_{\nu j_\nu}, \ldots , b_{\nu n})$ and $b_\nu' =(0, \ldots , 0, b_{\nu j_\nu}', \ldots , b_{\nu n}')$ are the rows of the matrices $B$ and $B'$, respectively.

\item  If $A\sim_s B=(b_1, \ldots , b_s)^t$, where $B$ is the  matrix in strong    reduced echelon form,  then for each $i=1, \ldots , n$, the matrix $(b_{\nu_i}, \ldots , b_s)^t$ is in   strong    reduced echelon form and $N(A)_{\geq i}=\bigoplus_{\mu =\nu_i}^sRb_\mu$  where $\nu_i=\min\{ \nu \in [s]\, | \, i\leq j_\nu\}$ where $[s]=\{1,\ldots , s\}$.
\end{enumerate}
\end{corollary} 

\begin{proof} 1. See the proof of Theorem \ref{STR-REForm}.

2. Clearly, for all $i=1, \ldots , n$, $N(A)_{\geq i}=N(A)_{\geq j_{\nu_i}}=\bigoplus_{\mu =\nu_i}^sRb_\mu$, where $\nu_i=\min\{ \nu \in [s]\, | \, i\leq j_\nu\}$, and  the statement follows from the definition of the strong reduced echelon form.
\end{proof}

Theorem \ref{26Jul26} is a criterion for two matrices in  strong reduced echelon form to be strongly equivalent. 

\begin{theorem}[{\bf The equivalence class of strong reduced echelon forms of $A$ depends only on the module $N(A)$}] \label{26Jul26}%\marginpar{26Jul26}
  Suppose that matrices $B \in M_{s,n}(M_1, \ldots ,M_n)$ and $B' \in M_{s',n}(M_1, \ldots ,M_n)$ are in strong reduced echelon form. Then  $B\sim_s B'$ iff $N(B) = N(B')$. 
\end{theorem}

\begin{proof} $(\Rightarrow)$ Clearly, if $B\sim_s B'$ then  $N(B) = N(B')$.

$(\Leftarrow)$ Suppose that $N:=N(B) = N(B')$ and $\mP_s(N)=\{ j_1, \ldots , j_s\}$ where $j_1< \ldots < j_s$.  Then $s=s'$, $B=(b_1, \ldots , b_s)^t$ and $B'=(b_1', \ldots , b_s')^t$  where $b_1, \ldots , b_s$ and $b_1', \ldots , b_s'$ are the rows of the matrices $B$ and $B'$, respectively. We have to show that $B\sim_s B'$. 
To prove the result, we use induction on $s$. 

Suppose that $s=1$. Then $B=(b_1)=(0\ldots  0 \, b_{1j_1}\, \ldots \, b_{1n})$ and $B'=(b_1')=(0\ldots  0 \, b_{1j_1}'\, \ldots \, b_{1n}')$.  Notice that  ${\rm gr}(N)_{j_1}=R b_{1j_1}=R b_{1j_1}'$, and so   $b_1=\alpha b_1'$ and $b_1'=\beta b_1$ for some elements $\alpha, \beta \in R$. Notice that $b_1-\alpha b_1'\in {\rm gr}(N)_{\geq j_1+1}=\{ 0\}$, i.e. $b_1-\alpha b_1'=0$.  Now, 
$$B= (b_1)\sim_s
\begin{pmatrix}
b_1\\[4pt]
\beta b_1
\end{pmatrix}
=\begin{pmatrix}
b_1\\[4pt]
b_1'
\end{pmatrix}
\sim_s 
\begin{pmatrix}
b_1-\alpha b_1'\\[4pt]
b_1'
\end{pmatrix}
\sim_s 
\begin{pmatrix}
0\\[4pt]
b_1'
\end{pmatrix}
\sim_s 
\begin{pmatrix}
b_1'\\[4pt]
0
\end{pmatrix}
\sim_s(b_1')=B'.
$$
Suppose that $s\geq 2$. 
Let $B_1:=(b_2, \ldots , b_s)^t$ and $B_1'=(b_2', \ldots , b_s')^t$  where $b_2, \ldots , b_s$ and $b_2', \ldots , b_s'$ are the rows of the matrices $B_1$ and $B_1'$, respectively.  Since the matrices $B$ and $B'$ are in strong reduced echelon form,   the matrices $B_1$ and $B_1'$ are also in strong reduced echelon form. 
Since $N(B_1)=\sum_{\nu=2}^sRb_\nu=N_{\geq j_1}=\sum_{\nu=2}^sRb_\nu'=N(B_1')$ and $B_1,B_1'\in M_{s-1, n-j_1+1}(M_{j_1}\oplus\cdots \oplus M_n)$, we have that $B_1\sim_sB_1'$, by induction on $s$. Therefore, 
$$
B=(b_1,b_2\ldots , b_s)\sim_s A:=(b_1,b_2'\ldots , b_s').
$$
Thus, $N(B)=N(A)$. 
Since $\sum_{\nu=1}^sRb_\nu'=N(B')=N(B)=N(A)=Rb_1+\sum_{\nu=2}^sRb_\nu'$, we have the equalities  $b_1=\alpha b_1'+\sum_{i =2}^s\alpha_ib_i'$ and $b_1'=\beta b_1+\sum_{i =2}^s\beta_ib_i'$ for some elements $\alpha, \alpha_i, \beta, \beta_i\in R$. Now,
\begin{eqnarray*}
 A &=& (b_1,b_2'\ldots , b_s')^t\sim_s(b_1,b_2'\ldots , b_s', \beta b_1)^t\sim_s\bigg(b_1,b_2'\ldots , b_s',  b_1'-\sum_{i =2}^s\beta_ib_i'\bigg)^t\\
 &\sim_s & (b_1,b_2'\ldots , b_s',  b_1')^t\sim_s (b_1-\alpha b_1',b_2'\ldots , b_s',  b_1')^t \sim_s \bigg(\sum_{i =2}^s\alpha_i b_1',b_2'\ldots , b_s',  b_1'\bigg)^t\\
 &\sim_s & (0,b_2'\ldots , b_s',  b_1')^t\sim_s (b_1',b_2'\ldots , b_s')^t=B'.
\end{eqnarray*}
 Thus, $B\sim_s B'$.
\end{proof}

Theorem \ref{b26Jul26} describes the equivalence class of the matrix $A$.

\begin{theorem}[{\bf  The equivalence class of $A$ depends only on the module $N(A)$}] \label{b26Jul26}%\marginpar{b26Jul26}
 Suppose that  $A\in M_{m,n}(M_1, \ldots ,M_n)$ and  $A'\in M_{m',n}(M_1, \ldots ,M_n)$. Then $A\sim_s A'$ iff  $N(A) = N(A')$.
\end{theorem}

\begin{proof} By Theorem \ref{STR-REForm}, $A\sim_sB$ and $A'\sim_sB'$ for some matrices $B$ and $B'$ that are in strong reduced echelon form. Then $A\sim_sA'$ iff $B\sim_sB'$ iff $N(B)=N(B')$ (Theorem \ref{26Jul26}) iff $N(A)=N(A')$
 (since $N(A)=N(B)$ and $N(A')=N(B')$).
\end{proof}
Next, in the case where  the  matrix $A$  consists of a single row,  we consider in detail its  strong reduced echelon form. We will see how the presence of zero divisors complicates the situation. \\ 

{\bf Strong reduced echelon form of the matrix $(a)=(a_{11}\,  a_{12}\, \ldots  \, a_{1n})\in M_{1,n}(M_1,\ldots , M_n)$.} Let us describe a strong reduced echelon form of the  matrix  
$$A=(a)=(a_{11}\,  a_{12}\, \ldots  \, a_{1n})\in M_{1,n}(M_1,\ldots , M_n).
$$
 Let $N:=N(A)=Ra$. Suppose that $\mP_s(N)=\{ j_1, \ldots , j_s\}$ where $j_1< \ldots < j_s$.  For each $\nu =1, \ldots , s$, let  $\ga_\nu :=\ann_R(a_{1j_\nu})$ and 
$$
\gb_\nu:=      
\begin{cases}
R& \text{if }\nu =1,\\
\ga_1\cap\cdots \cap \ga_{\nu-1}& \text{if }\nu=2,\ldots , s.\\
\end{cases}
$$
It  follows from the equality $\mP_s(N)=\{ j_1, \ldots , j_s\}$ that $\ga_\nu\subseteq \ann_R(a_{1i})$ for $i=j_\nu,\ldots, j_{\nu+1}-1$. Thus $\gb_\nu=\bigcap_{i=1}^{j_\nu-1}\ann_R(a_{1i})$ for $\nu=2,\ldots , s$. Therefore, 
$N_{\geq j_\nu}=\gb_\nu a$ for  $\nu =1, \ldots , s$. 

By  Proposition \ref{A28May26}, $p_{j_\nu}(N_{\geq j_\nu})=Rb_{\nu j_\nu}$ for some element $b_{\nu j_\nu}\in \CG(p_{j_\nu}(N_{\geq j_\nu}))$ where $p_{j_\nu}$ is the projection $M\ra M_{j_\nu}$. Therefore there  is an element $\alpha_\nu \in \gb_\nu$ such that $b_{\nu j_\nu}=\alpha_\nu a_{1j_\nu}$. Let $b_\nu := \alpha_\nu a$. Then $p_{j_\nu} (b_\nu)=\alpha_\nu a_{1j_\nu}=b_{\nu j_\nu}$. Thus $N=\sum_{\nu =1}^s Rb_\nu$ and $N_{\geq \nu}=\sum_{\mu =1}^\nu Rb_\mu$ for $\nu=1, \ldots , s$. Therefore,   
$$
A\sim_s(b_1, \ldots b_s)^t.
$$
Finally, $(b_1, \ldots b_s)^t\sim_s (c_1, \ldots c_s)^t$ where the matrix $(c_1, \ldots c_s)^t$ is in strong reduced echelon form which is obtained from the matrix $(b_1, \ldots b_s)^t$ in the obvious way. In more detail,  
 using the operation (ER1), for each $\nu =2, \ldots , s$, the degrees of the elements
 above the element $b_{\nu j_\nu}$ can be made smaller than the degree of  $b_{\nu j_\nu}$.
 
Proposition \ref{a30Jul26} provides criteria ensuring  that  the $R$-modules $N(A)$ and $M/N(A)$ have finite length and it gives explicit expressions for these lengths.

\begin{proposition} \label{a30Jul26}%\marginpar{a30Jul26}

 Let $M=\bigoplus_{i=1}^nM_i$ be a direct sum of l-finite Euclidean $R$-modules $M_i$, $0\neq A \in M_{m,n}(M_1, \ldots ,M_n)$, $\mP_s(N(A))=\{ j_1, \ldots , j_s\}$  where $j_1< \cdots < j_s$ and $A\sim_sB$ where the matrix $B=(b_1, \ldots , b_s)^t$ is in strong reduced echelon form. Then:
 \begin{enumerate}

\item  $l_R(N(A))<\infty$ iff $l_R(M_j)<\infty$ for all  $j\in \mP_s(N(A))$. If $l_R(N(A))<\infty$ then $l_R(N(A))=\sum_{\nu=1}^s l_R(Rb_{\nu j_\nu})$ where 
where $b_\nu =(0, \ldots , 0, b_{\nu j_\nu}, \ldots , b_{\nu n})$, $\nu =1,\ldots  , s$,  are the rows of the matrices $B$.

\item $l_R(M/N(A))<\infty$ iff $l_R(M_j)<\infty$ for all  $j\in [n]\backslash \mP_s(N(A))$. If $l_R(M/N(A))<\infty$ then $l_R(M/N(A))=\sum_{\nu=1}^s l_R(M_{j_\nu}/Rb_{\nu j_\nu})+\sum_{j\in [n]\backslash \mP_s(N(A)} l_R(M_j)$.

\end{enumerate}
 
\end{proposition} 

\begin{proof} 1. Since $A\sim_s B$, $N:=N(A)=N(B)$. For each $\nu =1,\ldots , s$, $N_{\geq j_\nu}=\sum_{\mu=\nu}^sRb_\mu$ and 
\begin{eqnarray*}
{\rm gr}(N)_{\geq j_\nu} &=& N_{\geq j_\nu}/N_{\geq j_\nu +1}=N_{\geq j_\nu}/N_{\geq j_{\nu +1}}=\bigg(\sum_{\mu=\nu}^sRb_\mu \bigg)/\bigg(\sum_{\mu=\nu+1}^sRb_\mu \bigg)\\
 &=& Rb_\nu/Rb_\nu\cap\bigg(\sum_{\mu=\nu+1}^sRb_\mu \bigg)=Rb_\nu/Rb_\nu\cap N_{\geq j_\nu +1}\simeq R/\{r\in R\, | \, rb_\nu\in N_{\geq j_\nu +1} \}\\
 &=& R/\{r\in R\, | \, rb_\nu\in M_{\geq j_\nu +1} \}= R/\{r\in R\, | \, rb_{\nu j_\nu}=0 \}= R/\ann_R(b_{\nu j_\nu})\simeq Rb_{\nu j_\nu}. 
\end{eqnarray*}
Therefore, 
$$
l_R(N(A))=l_R(N(B))=l_R({\rm gr}(N(B)))=
\sum_{\nu=1}^s l_R({\rm gr}(N(B))_{j_\nu})=\sum_{\nu=1}^s l_R(Rb_{\nu j_\nu}),
$$
and statement 1  follows from the fact that every  nonzero submodule of a Euclidean module of infinite length has also infinite length (Lemma \ref{b14Jun26}(4)).

2. Since $A\sim_sB$, $N(A)=N(B)$. By the graded short exact sequence (\ref{ses-grN}), 
\begin{eqnarray*}
l_R(M/N(A))&=&l_R(M/N(B))=l_R\Big({\rm gr}(M/N(B))\Big)=\sum_{\nu=1}^s l_R(M_{j_\nu}/{\rm gr}(N(B))_{j_\nu})\\
&+&\sum_{j\in [n]\backslash \mP_s(N(A))} l_R(M_j)=\sum_{\nu=1}^s l_R(M_{j_\nu}/Rb_{\nu j_\nu})+\sum_{j\in [n]\backslash \mP_s(N(A))} l_R(M_j),   
\end{eqnarray*}
and statement 2  follows  from the fact that every  proper factor module of a Euclidean module has   finite length (Theorem  \ref{14Jun26}(2)). 
\end{proof}
The {\bf uniform dimension} of an $R$-module $M$, denoted $\udim (M)$, is  the supremum of the integers $n$ such that there is a direct sum $\oplus_{i=1}^n U_i\subseteq  M$ of uniform submodules $U_i$ of $M$. If no such finite $n$ exists, we say that the module $M$  has {\em infinite} uniform dimension. For the direct sum $M=\bigoplus_{i=1}^nM_i$ of Euclidean $R$-modules $M_i$ and a subset $I$ of $[n]$, let $M(I):=\bigoplus_{i\in I}M_i$ and $CI:=[n]\backslash I$ be the {\em complement} of the set $I$ in the set $[n]$.  

Proposition \ref{REForm-N1} presents a sufficient condition  for a reduced echelon form of the matrix $A$ to be a strong reduced echelon form and it also computes the  uniform dimensions of the $R$-module $N(A)$ and $N(A)+M(C\mP_s(N(A))$.

\begin{proposition} \label{REForm-N1}%\marginpar{REForm-N1}

Let $M=\bigoplus_{i=1}^nM_i$ be a direct sum of Euclidean $R$-modules $M_i$,
$A =(a_{ij})\in M_{m,n}(M_1, \ldots ,M_n)$,  $N(A):=\sum_{i=1}^mR a_i$ be the  $R$-submodule of $M$ which is generated by the rows $a_i:=(a_{i1}, \ldots , a_{in})$ of the matrix $A$ and $B=(b_{ij})\in M_{m,n}(M_1, \ldots ,M_n)$ be a reduced echelon form of the matrix $A$ whose pivots lie in the columns with indices $j_1<\ldots <j_s$. Suppose that $b_1, \ldots , b_s$ are the rows of the matrix $B$ such that $\ann_R(b_i)= \ann_R(b_{ij_i})$ for all $i=1, \ldots s$. Then:

\begin{enumerate}

\item The matrix $B$ is in strong reduced echelon form and $\mP_s(N(A))=J:=\{ j_1, \ldots , j_s\}$.

\item  $N(A)=\sum_{i=1}^sRb_i=\bigoplus_{i=1}^s Rb_i$ and      $N(A)\cap M(CJ)=0$, i.e. $N(A)+M(CJ)=N(A)\oplus M(CJ)=\bigoplus_{i=1}^s Rb_i \oplus  M(CJ)\subseteq M$.

\item Suppose, in addition,  that the $R$-modules $M_i$ are also uniform (this is the case if all of them are l-finite and of  infinite length, Lemma \ref{b14Jun26}.(1)). Then the $R$-module  $N(A)\oplus M(CJ)$ is an essential submodule of $M$ and $\udim (N(A))=|J|=s$. 

\end{enumerate}
\end{proposition}

\begin{proof} 1. Let $N=N(A)$. Since $A\sim B$, $N=N(B)=\sum_{\nu=1}^sRb_\nu$.  
 To prove  statement 1,  it remains to show that $N_{\geq \nu}=L_\nu :=\sum_{\mu =\nu}^sRb_\mu$ for all $\nu =1,\ldots , s$. Clearly, $N_{\geq \nu}\supseteq L_\nu$ for all $\nu =1,\ldots , s$.  Suppose that $N_{\geq \nu}\neq  L_\nu$ 
 for some $\nu$. We seek a contradiction. Then there is an element $u=u_\mu b_\mu+\cdots +u_sb_s\in N_{\geq j_\nu}$ where $u_\mu,\ldots , u_s\in R$, $\mu<\nu$ and $u_\mu b_\mu\neq 0$. Since 
$\mu<\nu$  and  $u\in N_{\geq j_\nu}$, we must have $u_\mu \in \ann_R(b_{\mu j_\mu})$. By the assumption  $\ann_R(b_mu)= \ann_R(b_{\mu j_\mu})$.   Therefore, $u_\mu b_\mu =0$, a contradiction.

2. The matrices $A$ and $B$ are left-row equivalent. Therefore,  
$
N(A)=\sum_{i=1}^mRa_i=\sum_{i=1}^s Rb_i.
$ 
Now, the equalities  $\ann_R(b_i)= \ann_R(b_{ij_i})$ for all $i=1, \ldots s$ imply  the equalities
$$
\sum_{i=1}^sRb_i=\bigoplus_{i=1}^s Rb_i \;\; {\rm and}\;\;N(A)\cap M(CJ)=0.
 $$
In more detail, suppose that $x:=\sum_{i=1}^sr_ib_i=0$ for some elements $r_i\in R$ such that not all summands are equal to zero. By statement 1,   $\mP_s(N(A))=\{ j_1, \ldots , j_s\}$.  So, let $r_\nu b_\nu$ be the first nonzero summand, i.e.  $x=\sum_{i=\nu}^sr_ib_i=0$, $r_\nu \not \in \ann_R(b_{\nu)}$ and $\pi_\nu :M\ra M_\nu$ be the projection onto the direct summand $M_\nu$ of $M$. Then,   $0=\pi_\nu (x)=r_\nu b_{\nu, j_\nu}$, and so $r_\nu \in \ann_R(b_{\nu, j_\nu)}=\ann_R(b_{\nu})$, a contradiction.  Therefore,   $\sum_{i=1}^sRb_i=\bigoplus_{i=1}^s Rb_i$.

Suppose that $0\neq y:=\sum_{i=1}^sy_ib_i\in N(A)\cap M(CJ)$ for some elements $y_i\in R$ such that not all summands are equal to zero. Let $r_\mu b_\mu$ be the first nonzero summand, i.e.  $y=\sum_{i=\mu}^sy_ib_i$ and $r_\mu \not \in \ann_R(b_{\mu)}$, and $\pi_\mu :M\ra M_\mu$ be the projection onto the direct summand $M_\mu$ of $M$. Then,   $0=\pi_\mu (y)=r_\mu b_{\mu, j_\nu}$ (since $y\in M(CJ)$ and $\mu \not\in CJ$), and so $r_\mu \in \ann_R(b_{\mu, j_\mu)}=\ann_R(b_{\mu)}$. Thus, $r_\mu b_\mu=0$, a contradiction.  Therefore,   $N(A)\cap M(CJ)=0$.
 
Clearly, $N(A)\oplus M(CJ)=\bigoplus_{i=1}^s Rb_i \oplus  M(CJ)\subseteq M$.

3.  By the assumption, the $R$-modules $M_i$ are  uniform and $M=\bigoplus_{i=1}^sM_i$. Therefore, $\udim (M)=n$.  Since the $R$-module  $N(A)\oplus M(CJ)= \bigoplus_{\nu =1}^s Rb_{\nu} \oplus  M(CJ)$ contains exactly $n=\udim (M)$ nonzero summands, its uniform dimension is also $n$. Thus, it  is an essential submodule of $M$ and every summands of it is a uniform submodule of $M$. Hence,  $\udim (N(A))=\udim (M)-\udim(M(CJ))=n-|CJ|=s$. 
\end{proof}

\begin{corollary} \label{REForm-N2}%\marginpar{REForm-N2}

Let $M=\bigoplus_{i=1}^nM_i$ be a direct sum of l-finite Euclidean $R$-modules $M_i$ of infinite length,
$A =(a_{ij})\in M_{m,n}(M_1, \ldots ,M_n)$,  $N(A):=\sum_{i=1}^mR a_i$ be the  $R$-submodule of $M$ which is generated by the rows $a_i:=(a_{i1}, \ldots , a_{in})$ of the matrix $A$ and $B=(b_{ij})\in M_{m,n}(M_1, \ldots ,M_n)$ be a reduced echelon form of the matrix $A$ whose pivots lie in the columns with indices $j_1<\ldots <j_s$. Suppose that $b_1, \ldots , b_s$ are the rows of the matrix $B$ such that $\ann_R(b_i)= \ann_R(b_{ij_i})$ for all $i=1, \ldots s$. Then:

\begin{enumerate}

\item  $\udim (N(A))=s$. 

\item $l_R(M/N(A))<\infty$ iff  $s=n$.

\item If $l_R(M/N(A))<\infty$ then $l_R(M/N(A))=\sum_{\nu =1}^nl_R(M_{j_\nu}/Rb_{\nu j_\nu})$.

\end{enumerate}
\end{corollary}

\begin{proof} Since every Euclidean module of  infinite length is a uniform module (Lemma \ref{b14Jun26}.(1),  the corollary follows from Proposition \ref{REForm-N1}.
\end{proof}

{\bf Left-row-column equivalence $\approx $.} 
Let $E$ be a Euclidean $R$-module  and   $0\neq A=(a_{ij}) \in M_{m,n}(E, \ldots ,E)$.
 On the columns $c_1, \ldots , c_n$ of the matrix $A$, we can do two types of operations -- {\em elementary column  operations}:
\begin{eqnarray*}
 &{\rm (EC1)}& c_i\mapsto c_i+\sum_{j\neq i}r_jc_j\;\; {\rm where}\;\; r_j\in R,  \\
 &{\rm (EC2)}& c_i\mapsto c_j, \;\; c_j\mapsto c_i.
\end{eqnarray*}
For a fixed column, (EC1) adds a linear combination of the other columns to it, and (EC2) swaps two columns. 
 These operations are reversible.

Two matrices $A, A'\in  M_{m,n}(E,\ldots , E)$ are called {\bf left-column equivalent}, $A\sim_c A'$, if one of them can be obtained from the other by finitely many elementary column operations (EC1) and (EC2). The relation $\sim_c$ is an equivalence relation.  Two matrices $A, A'\in   M_{m,n}(E,\ldots , E)$ are called {\bf left-row-column equivalent}, $A\approx A'$, if one of them can be obtained from the other by finitely many elementary row and column operations: (ER1), (ER2), (EC1) and (EC2). The relation $\approx$ is an equivalence relation. For a matrix $A=(a_{ij})\in   M_{m,n}(E,\ldots , E)$, let $\gcd (A):=\gcd \{ a_{ij}\, | \, i=1,\ldots, m; j=1,\ldots , n\}$. Then $\CN_R(A):=\sum_{i=1}^m\sum_{j=1}^nRa_{ij}=R\gcd (A)$. 
Clearly, if $A\approx A'$ then $R\gcd (A)=R\gcd (A')$.

%\begin{theorem}\label{REForm-N3}%\marginpar{REForm-N3}

%Let $(E,d_E)$ be a  Euclidean $R$-module  and   $0\neq A \in M_{m,n}(E, \ldots ,E)$ where $(m,n)\neq (1,1)$. Then $A\approx B$ where
%\[
%B
%=
%\begin{pmatrix}
%b_{11} & 0      & \cdots & 0      & 0      & \cdots & 0 \\
%0      & b_{22} & \cdots & 0      & 0      & \cdots & 0 \\
%\vdots & \vdots & \ddots & \vdots & \vdots &        & \vdots \\
%0      & 0      & \cdots & b_{ss} & 0      & \cdots & 0 \\[4pt]
%\hline
%0      & 0      & \cdots & 0      & 0      & \cdots & 0 \\
%\vdots & \vdots &        & \vdots & \vdots %&        & \vdots \\
%0      & 0      & \cdots & 0      & 0      & \cdots & 0
%\end{pmatrix}_{m\times n}
%\]
%where  $Rb_{11}\supset Rb_{22}\supset \cdots \supset Rb_{ss}\neq \{0\}$, $d_E(b_{11})<d_E(b_{22})< \cdots < d_E(b_{ss})$ and $b_{11}=\gcd (A)$. 
%\end{theorem}

\begin{proof}[{\bf Proof of Theorem \ref{REForm-3}}] To prove the theorem, we use induction on $n=1$. Suppose that $n=1$, i.e. $A=(a_{11}\, \ldots \, a_{m1})^t$ where $m\geq 2$ (since $(m,n)\neq (1,1)$). Then
$A\sim (\gcd (a_{11}, \ldots , a_{m1}) \,0 \,\ldots \, 0)^t$, by Lemma \ref{EuclDivAlg}. 

Suppose that $n>1$  and the statement is true for all $n'<n$. The case $m=1$ can be done in the same way as the case $n=1$ above. So, we assume that $m\geq 2$. Then repeating the same argument successively to the first row and the first column, we obtain the chain 
\[
A\approx 
A_1=\begin{pmatrix}
b_1 & * \\
0 &*
\end{pmatrix}\approx 
A_2=\begin{pmatrix}
b_2 & 0 \\
* &*
\end{pmatrix}\approx 
A_3=\begin{pmatrix}
b_3 & * \\
0 &*
\end{pmatrix}\approx 
A_4=\begin{pmatrix}
b_4 & 0 \\
* &*
\end{pmatrix}\approx \cdots
\]
where $Rb_1\subset Rb_2\subset Rb_3\subset \cdots$, $d_E(b_1)> d_E(b_2)> d_E(b_3)> \cdots$, the element $b_1$ is a greatest common divisor of the first column of the matrix $A$, the element $b_2$ is a greatest common divisor of the first row of the matrix $A_1$, the element $b_3$ is a greatest common divisor of the first column of the matrix $A_2$, the element $b_4$ is a greatest common divisor of the first row of the matrix $A_3$, etc. Since the $R$-module $E$ is a Euclidean module, it is a Noetherian module, and so the ascending chain of submodules of $E$, $Rb_1\subset Rb_2\subset Rb_3\subset \cdots$, stabilizers, say, on $p$'th step,
$$
A\approx \begin{pmatrix}
b_p & 0 \\
0 & C
\end{pmatrix}.
$$
Notice that $Rb_p\supseteq \CN_R(C)$. 
By induction, $C\approx B'$ where 
\[
B'
=
\begin{pmatrix}
b_{22}' & 0      & \cdots & 0      & 0      & \cdots & 0 \\
0      & b_{33}' & \cdots & 0      & 0      & \cdots & 0 \\
\vdots & \vdots & \ddots & \vdots & \vdots &        & \vdots \\
0      & 0      & \cdots & b_{tt}' & 0      & \cdots & 0 \\[4pt]
%\hline
0      & 0      & \cdots & 0      & 0      & \cdots & 0 \\
\vdots & \vdots &        & \vdots & \vdots &        & \vdots \\
0      & 0      & \cdots & 0      & 0      & \cdots & 0
\end{pmatrix}_{m-1\times n-1},
\]
$Rb_{22}'\supset Rb_{33}'\supset \cdots \supset Rb_{tt}'\neq \{0\}$,  $d_E(b_{22}')<d_E(b_{33}') <\cdots < d_E(b_{tt}')$, $b_{22}'=\gcd (C)$ and $Rb_{22}'=R\gcd (C)$. Notice that if $A\approx A'$ then $R\gcd (A)=R\gcd (A')$. Thus, $R\gcd (A) = R\gcd (b_p,\gcd (C))=R\gcd (b_p,\gcd (B'))=R\gcd (b_p,b_{22}')$. Let $b_{11}=\gcd (b_p,b_{22}')$. Then $Rb_{11}=R\gcd (A)$ and 
 $$
 \begin{pmatrix}
b_p & 0 \\
0 & b_{22}'
\end{pmatrix}\approx
\begin{pmatrix}
b_p & b_{11} \\
0 & b_{22}'
\end{pmatrix}\approx 
 \begin{pmatrix}
b_{11} & b_p \\
b_{22}' & 0
\end{pmatrix}\approx 
\begin{pmatrix}
b_{11} &0\\
b_{22}' & rb_{22}'
\end{pmatrix}\approx 
\begin{pmatrix}
b_{11} &0\\
0 & rb_{22}'
\end{pmatrix}
 $$ 
 for some element $r\in R$.
Therefore, 
$$
A\approx 
\begin{pmatrix}
b_{11} & 0      & \cdots & 0      & 0      & \cdots & 0 \\
0      & rb_{22}' & \cdots & 0      & 0      & \cdots & 0 \\
\vdots & \vdots & \ddots & \vdots & \vdots &        & \vdots \\
0      & 0      & \cdots & b_{tt}' & 0      & \cdots & 0 \\[4pt]
%\hline
0      & 0      & \cdots & 0      & 0      & \cdots & 0 \\
\vdots & \vdots &        & \vdots & \vdots &        & \vdots \\
0      & 0      & \cdots & 0      & 0      & \cdots & 0
\end{pmatrix}
\approx
B
=
\begin{pmatrix}
b_{11} & 0      & \cdots & 0      & 0      & \cdots & 0 \\
0      & b_{22} & \cdots & 0      & 0      & \cdots & 0 \\
\vdots & \vdots & \ddots & \vdots & \vdots &        & \vdots \\
0      & 0      & \cdots & b_{ss} & 0      & \cdots & 0 \\[4pt]
%\hline
0      & 0      & \cdots & 0      & 0      & \cdots & 0 \\
\vdots & \vdots &        & \vdots & \vdots &        & \vdots \\
0      & 0      & \cdots & 0      & 0      & \cdots & 0
\end{pmatrix}
$$
where the second $\approx$ is obtained by applying induction to the obvious submatrix.
The inequalities $d_E(b_{11})<d_E(b_{22})< \cdots < d_E(b_{ss})$ and the inclusions  $Rb_{11}\supset Rb_{22}\supset \cdots \supset Rb_{ss}\neq \{0\}$ are obvious (the first one follows from the fact that $b_{11}=\gcd (A)$). 
\end{proof}

{\bf Euclidean bimodules.} An $R$-bimodule $E$ is called a {\bf cyclic $R$-bimodule} if $E=Re=eR$ for some element $e\in E$ which is called a {\bf cyclic generator} for $E$. An element $c\in E$ of an $R$-module $E$ is called a {\bf cyclic element} of $E$ if $Rc=cR$. Let $\CC (E)$ be the set of cyclic elements of $E$. 
\begin{definition}
An $R$-bimodule $(E,d_E)$ is called a {\bf Euclidean $R$-bimodule} if the left and right $R$-modules, $(_RE,d_E)$ and $(E_R,d_E)$, are Euclidean. 
\end{definition}

\begin{lemma}\label{a9Aug26}%\marginpar{a9Aug26}
Let $(E,d_E)$ be a Euclidean $R$-bimodule. Then:  

\begin{enumerate}

\item Every sub-bimodule of the Euclidean $R$-bimodule $E$ is a cyclic $R$-bimodule.

\item If $N$ is a nonzero sub-bimodule of $E$ then the set of cyclic generators of the $R$-bimodule $N$ contains  the non-empty  set $\CG (N)$.
%$:=\{ n\in N\, | \, d_M(n)=l \}$ where $l:=\min \{ d_M(n')\, | \, 0\neq n'\in N \}$. 

\end{enumerate}
\end{lemma}

\begin{proof} 2. By the definition, $\CG (N)\neq \emptyset$. Fix an element $e\in \CG (N)$. Then $N=Re$ and $N=eR$, by Proposition \ref{A28May26}, and so $e$ is a cyclic generator for the $R$-bimodule $N$.

1. Statement 1 follows from statement 2.
\end{proof}  

\begin{definition}
Let $E$ be a Euclidean $R$-bimodule. For  elements $a_1, \ldots , a_n\in E$, a cyclic  generator 
$\gcd_b (a_1, \ldots , a_n)$ of the cyclic  sub-bimodule $N_b(a_1, \ldots , a_n):=\sum_{i=1}^n Ra_iR$ of $E$   is called the {\bf bimodule greatest common divisor} of the elements  $a_1, \ldots , a_n$ where $Ra_iR$ is a sub-bimodule of $E$ generated by the element $a_i$. Similarly, a cyclic generator 
$\lcm_b (a_1, \ldots , a_n)$ of the cyclic  sub-bimodule $I_b(a_1, \ldots , a_n):=\bigcap_{i=1}^n Ra_iR$ of $E$   is called the {\bf bimodule least common multiple} of the elements  $a_1, \ldots , a_n$.
\end{definition}

In general,  the elements $\gcd_b (a_1, \ldots , a_n)$ and $\lcm_b (a_1, \ldots , a_n)$
are not unique, see Corollary \ref{b16Jun26}.(2). More precisely, by Corollary \ref{b16Jun26}.(2), if cyclic elements $g, g'\in E$  are bimodule greatest common divisors of the  set of elements $a_1, \ldots , a_n\in E$ then there exist elements $u,v\in R$ such that 
$g= ug'$ and $g'= vg$  or, equivalently, there exist elements $x,y\in R$ such that $ g= g'x$ and $g'= g y$.

Similarly, if cyclic elements $l, l'\in E$  are bimodule little  common multiples  of the  set of elements $a_1, \ldots , a_n\in E$ then there exist elements $u',v'\in R$ such that $l= u'l'$ and $l'= v'l$  or, equivalently, there exist elements
$x',y'\in R$ such that 
$l= l'x'$ and $l'= l y'$.

\begin{proposition}\label{bi-A14Jul26}%\marginpar{bi-A14Jul26}
Let $E$ be a Euclidean $R$-bimodule 
%of infinite length 
 and  $a_1, \ldots , a_n\in E$. Then for a cyclic element $g\in E$ the following statements are equivalent:
 
\begin{enumerate}

\item $g=\gcd_b (a_1, \ldots , a_n)$.

\item For each $i=1, \ldots , n$, $a_i=r_ig$ for some  $r_i\in R$ (equivalently, $a_i=gs_i$ for some  $s_i\in R$)  and if a cyclic element $g'\in E$ satisfies the condition above (i.e. $a_i=r_i'g'$ for some elements $r_i'\in R$)  then $g=\alpha g'$ for some element $\alpha \in R$ (equivalently, $g=q'\alpha'$ for some element $\alpha' \in R$.  

\item For each $i=1, \ldots , n$, $a_i=r_ig$ for some  $r_i\in R$ (equivalently, $a_i=gr_i'$ for some  $r_i'\in R$ ) and $g=\sum_{i=1}^ns_ia_it_i$ for  some elements $s_i,t_i\in R$. 

\end{enumerate}

In particular,  $\gcd_b (a_1, \ldots , a_n)= \sum_{i=1}^ns_ia_it_i$ for  some elements $s_i,t_i\in R$.
\end{proposition}

\begin{proof} $(1\Leftrightarrow 2)$ $g=\gcd_b (a_1, \ldots , a_n)$ iff $Rg=gR=\sum_{i=1}^nRa_iR$ iff $a_i\in Rg$ for $i=1, \ldots , n$ and $Rg= \sum_{i=1}^nRa_iR$ 
iff for each $i=1, \ldots , n$, $a_i=r_ig$ for some  $r_i\in R$ and if an element $g'\in E$ satisfies the condition above (i.e. for each $i=1, \ldots , n$, $a_i=r_i'g'$ for some  $r_i'\in R$ )    then $g=\alpha g'$ for some element $\alpha \in R$ (since $g\in \sum_{i=1}^nR\oa_iR=\sum_{i=1}^nRr_i'g'R\subseteq Rg'$).

$(1\Leftrightarrow 3)$ $g=\gcd_b (a_1, \ldots , a_n)$ iff $Rg=\sum_{i=1}^nRa_iR$ iff $a_i\in Rg$ for $i=1, \ldots , n$ and $Rg\subseteq \sum_{i=1}^nRa_iR$ iff  for each $i=1, \ldots , n$, $a_i=r_ig$ for some  $r_i\in R$ and $g=\sum_{i=1}^ns_ia_it_i$ for  some elements $s_i,t_i\in R$.
\end{proof}

\begin{proposition}\label{bi-B14Jul26}%\marginpar{bi-B14Jul26}
Let $E$ be a Euclidean $R$-module   
 and  $a_1, \ldots , a_n\in E$. Then for an element $l\in E$ the following statements are equivalent:
 
\begin{enumerate}

\item $l=\lcm_b (a_1, \ldots , a_n)$.

\item For each $i=1, \ldots , n$, $l\in Ra_iR $ and if a cyclic element $l'\in E$ satisfies the condition above (i.e. for each $i=1, \ldots , n$, $l'\in Ra_iR $)  then $l'=\alpha l$ for some element $\alpha \in R$.

\end{enumerate}
\end{proposition}

\begin{proof} $(1\Leftrightarrow 2)$ $l=\lcm_b(a_1, \ldots , a_n)$ iff $Rl=lR=\bigcap_{i=1}^nRa_iR$ iff $l\in Ra_iR$ for $i=1, \ldots , n$ and $Rl= \bigcap_{i=1}^nRa_iR$ 
iff  a cyclic element $l'\in E$ satisfies the condition above  (i.e. for each $i=1, \ldots , n$, $l'\in Ra_iR $)  then $l'=\alpha l$ for some element $\alpha \in R$ (since $Rl=\bigcap_{i=1}^nR\oa_iR\ni l'$).
\end{proof}
Clearly, 
$$\gcd_b (a_1, \ldots , a_n)=\gcd_b\Big( \gcd_b (a_1, \ldots ,a_{n-1}),  a_n\Big)\;\; {\rm and}\;\; \lcm_b (a_1, \ldots , a_n)=\lcm_b\Big( \lcm_b (a_1, \ldots ,a_{n-1}),  a_n\Big).
$$

Suppose that  $(E,d_E)$ is  a Euclidean $R$-bimodule and $\emptyset \neq S\subseteq E$. Let $N_b (S):=\sum_{s\in S} RsR$. Any cyclic generator  of the $R$-bimodule $N_b(S)$ is denoted by $\gcd_b(S)$ and called a {\bf bimodule greatest  common divisor} of the elements of the set $E$. The set $\CG (N_b (S))$ consists of bimodule greatest  common divisors  of the $R$-bimoodule $N_b (S)$ (Lemma \ref{a9Aug26}.(2)). A {\bf bi-module generator} of $E$ is an element $a\in E$ such that $E=RaR$.  A cyclic generator $a$  of $E$ is a  bimodule such that  $E=Ea=aE$.

Let $(E,d_E)$ be a Euclidean $R$-bimodule  and   $0\neq A=(a_{ij}) \in M_{m,n}(E, \ldots ,E)$.
 On the columns $c_1, \ldots , c_n$ of the matrix $A$, we can do two types of operations -- {\em elementary column  operations}:
\begin{eqnarray*}
 &{\rm (EC1)}& c_i\mapsto c_i+\sum_{j\neq i}c_jr_j\;\; {\rm where}\;\; r_j\in R,  \\
 &{\rm (EC2)}& c_i\mapsto c_j, \;\; c_j\mapsto c_i.
\end{eqnarray*}
Notice that the elements $r_j\in R$ are written on the {\em right} not on the {\em left} as in the previous definition of (EC1).  For a fixed column, (EC1) adds a linear combination of the other columns to it, and (EC2) swaps two columns. 
 These operations are reversible.

Two matrices $A, A'\in  M_{m,n}(E,\ldots , E)$ are called {\bf right-column equivalent}, $A\sim_{rc} A'$, if one of them can be obtained from the other by finitely many elementary column operations (EC1) and (EC2). The relation $\sim_{rc}$ is an equivalence relation.  Two matrices $A, A'\in   M_{m,n}(E,\ldots , E)$ are called {\bf row-column equivalent}, $A\approx_b A'$ (where $b$ stands for `bimodule'), if one of them can be obtained from the other by finitely many elementary row and column operations: (ER1), (ER2), (EC1) and (EC2). The relation $\approx_b$ is an equivalence relation. For a matrix $A=(a_{ij})\in   M_{m,n}(E,\ldots , E)$, let $\gcd_b (A):=\gcd_b \{ a_{ij}\, | \, i=1,\ldots, m; j=1,\ldots , n\}$. Then $N_b(A):=\sum_{i=1}^m\sum_{j=1}^nRa_{ij}R=R\gcd_b (A)=\gcd_b (A)R$. 
Clearly, if $A\approx_b A'$ then $N_b(A)=N_b(A')$ and $R\gcd_b (A)=R\gcd_b (A')$. 

%\begin{theorem}\label{bi-REForm-N3}%\marginpar{bi-REForm-N3}

%Let $(E,d_E)$ be a  Euclidean $R$-bimodule  and   $A \in M_{m,n}(E, \ldots ,E)$ where $m,n\geq 2$. Then $A\approx_b B$ where
%\[
%B
%=
%\begin{pmatrix}
%b_{11} & 0      & \cdots & 0      & 0      & \cdots & 0 \\
%0      & b_{22} & \cdots & 0      & 0      & \cdots & 0 \\
%\vdots & \vdots & \ddots & \vdots & \vdots &        & \vdots \\
%0      & 0      & \cdots & b_{ss} & 0      & \cdots & 0 \\[4pt]
%\hline
%0      & 0      & \cdots & 0      & 0      %& \cdots & 0 \\
%\vdots & \vdots &        & \vdots & \vdots &        & \vdots \\
%0      & 0      & \cdots & 0      & 0      & \cdots & 0
%\end{pmatrix}_{m\times n}
%\]
%where  $N_b(A)=Rb_{11}R\supseteq Rb_{22}R\supseteq \cdots \supseteq Rb_{ss}R\neq \{0\}$.

%*** ***
% $Rb_{11}\supset Rb_{22}\supset \cdots \supset Rb_{ss}\neq \{0\}$, $b_{11}R\supset b_{22}R\supset \cdots \supset b_{ss}R\neq \{0\}$, $d_E(b_{11})<d_E(b_{22})< \cdots < d_E(b_{ss})$ and $b_{11}=\gcd (A)=\gcd_r (A)$. In particular, $Rb_{11}R\supseteq Rb_{22}R\supseteq \cdots \supseteq Rb_{ss}R\neq \{0\}$, $Rb_{11}=\sum_{i=1}^m\sum_{j=1}^nRa_{ij}$ and $b_{11}R=\sum_{i=1}^m\sum_{j=1}^na_{ij}R$.
 
% **
%\end{theorem}

\begin{proof}[{\bf Proof of Theorem \ref{bi-REForm-N3}}] To prove the theorem, we use induction on $n\geq 2$. Suppose that $n=2$.
 Then 
 \[
A\approx_b 
A_1=\begin{pmatrix}
b_1 & * \\
0 &*
\end{pmatrix}\approx_b 
A_2=\begin{pmatrix}
b_2 & 0 \\
* &*
\end{pmatrix}\approx_b 
A_3=\begin{pmatrix}
b_3 & * \\
0 &*
\end{pmatrix}\approx_b 
A_4=\begin{pmatrix}
b_4 & 0 \\
* &*
\end{pmatrix}\approx_b \cdots
\]
where  $d_E(b_1)> d_E(b_2)> d_E(b_3)> \cdots$, the element $b_1$ is a {\em left} greatest common divisor of the first column of the matrix $A$, the element $b_2$ is a {\em right}  greatest common divisor of the first row of the matrix $A_1$, the element $b_3$ is a {\em left}  greatest common divisor of the first column of the matrix $A_2$, the element $b_4$ is a {\em right} greatest common divisor of the first row of the matrix $A_3$, etc. Since the $R$-modules ${}_RE$ and $E_R$ are Euclidean modules, the descending chain $d_E(b_1)> d_E(b_2)> d_E(b_3)> \cdots$ stabilizers, say, on $p$'th step,
$$
A\approx_b \begin{pmatrix}
b_p & 0 \\
0 & \alpha
\end{pmatrix}.
$$
Then $Rb_p\supseteq R\alpha$ and $b_pR\supseteq \alpha R$. In particular, $Rb_pR\supseteq R\alpha R$, and so $N_b(A)=Rb_pR$.

Suppose that $n>1$  and the statement is true for all $n'<n$.  Then 
\[
A\approx_b 
A_1'=\begin{pmatrix}
b_1' & * \\
0 &*
\end{pmatrix}\approx_b 
A_2'=\begin{pmatrix}
b_2' & 0 \\
* &*
\end{pmatrix}\approx_b 
A_3'=\begin{pmatrix}
b_3' & * \\
0 &*
\end{pmatrix}\approx_b 
A_4'=\begin{pmatrix}
b_4' & 0 \\
* &*
\end{pmatrix}\approx_b \cdots
\]
where   
$d_E(b_1')> d_E(b_2')> d_E(b_3')> \cdots$,
 the element $b_1'$ is a greatest common divisor of the first column of the matrix $A$, the element $b_2'$ is a greatest common divisor of the first row of the matrix $A_1$, the element $b_3'$ is a greatest common divisor of the first column of the matrix $A_2$, the element $b_4'$ is a greatest common divisor of the first row of the matrix $A_3$, etc. Since the $R$-modules ${}_RE$ and $E_R$ are Euclidean modules, the descending chain $d_E(b_1')> d_E(b_2')> d_E(b_3')> \cdots$ stabilizers, say, on $q$'th step,
$$
A\approx_b \begin{pmatrix}
b_{11} & 0 \\
0 & C
\end{pmatrix}.
$$
Notice that $Rb_{11}\supseteq Rc_{ij}$ and $b_{11}R\supseteq c_{ij}R$ for all $i$ and $j$ where $C=(c_{ij})$. In particular, the $R$-bimodule $Rb_{11}R$ contains the $R$-bimodule $N_b(C)=\sum_{i,j}Rc_{ij}R$, and so $N_b(A)=Rb_{11}R$. 

By induction, $C\approx_b B'$ where 
\[
B'
=
\begin{pmatrix}
b_{22} & 0      & \cdots & 0      & 0      & \cdots & 0 \\
0      & b_{33} & \cdots & 0      & 0      & \cdots & 0 \\
\vdots & \vdots & \ddots & \vdots & \vdots &        & \vdots \\
0      & 0      & \cdots & b_{tt} & 0      & \cdots & 0 \\[4pt]
%\hline
0      & 0      & \cdots & 0      & 0      & \cdots & 0 \\
\vdots & \vdots &        & \vdots & \vdots &        & \vdots \\
0      & 0      & \cdots & 0      & 0      & \cdots & 0
\end{pmatrix}_{m-1\times n-1},
\]
where $Rb_{22}R\supseteq Rb_{33}R\supseteq \cdots \supseteq Rb_{tt}R\neq \{0\}$. Since  $b_{22}\in N_b(C)$ and $Rb_{11}R\supseteq N_b(C)$, we obtain the descending chain of $R$-bimodules    $Rb_{11}R\supseteq Rb_{22}R\supseteq \cdots \supseteq Rb_{tt}R\neq \{0\}$. Clearly, 
$$
A\approx_b B:=\begin{pmatrix}
b_{11} & 0 \\
0 & B'
\end{pmatrix}.
$$
\end{proof}

{\bf Ore and denominator sets, localizations and sufficient conditions for a ring to have a semisimple Artinian left quotient  ring.} Let $R$ be a ring. A subset $S$ of $R$ is called a {\em multiplicative set} if  $SS\subseteq S$, $1\in S$ and $0\not\in S$. A 
multiplicative subset $S$ of $R$   is called  a {\em
left Ore set} if it satisfies the {\em left Ore condition}: for
each $r\in R$ and
 $s\in S$, $$ Sr\bigcap Rs\neq \emptyset .$$
Let $\Ore_l(R)$ be the set of all left Ore sets of $R$.
  For  $S\in \Ore_l(R)$, $\ass_l (S) :=\{ r\in
R\, | \, sr=0 \;\; {\rm for\;  some}\;\; s\in S\}$  is an ideal of
the ring $R$.

A left Ore set $S$ is called a {\em left denominator set} of the
ring $R$ if $rs=0$ for some elements $ r\in R$ and $s\in S$ implies
$tr=0$ for some element $t\in S$, i.e., $r\in \ass_l (S)$. Let
$\Den_l(R)$ (resp., $\Den_l(R, \ga )$) be the set of all left denominator sets of $R$ (resp., such that $\ass_l(S)=\ga$). For
$S\in \Den_l(R)$, let $$S^{-1}R=\{ s^{-1}r\, | \, s\in S, r\in R\}$$
be the {\em left localization} of the ring $R$ at $S$ (the {\em
left quotient ring} of $R$ at $S$). By definition, in Ore's method of localization one can localize {\em precisely} at the left denominator sets.
 In a similar way, right Ore and right denominator sets are defined. 
Let $\Ore_r(R)$ and 
$\Den_r(R)$ be the set of all right  Ore and  right   denominator sets of $R$, respectively.  For $S\in \Ore_r(R)$, the set  $\ass_r(S):=\{ r\in R\, | \, rs=0$ for some $s\in S\}$ is an ideal of $R$. For
$S\in \Den_r(R)$,  $$RS^{-1}=\{ rs^{-1}\, | \, s\in S, r\in R\}$$   is the {\em right localization} of the ring $R$ at $S$. If the set $\CC_R$ of {\em regular elements} (i.e. non-zero-divisors) of $R$ is a left Ore set (hence,  a left denominator set), the the ring $Q_l(\CC_R^{-1}R)$ is called the {\bf left quotient ring} of $R$.

%For a ring $R$ and its ideal $\ga$, let $\Den_l(R, \ga)$ be the set of left denominator sets $S$ of $R$ such that $\ass_l(S):=\{ r\in R\mid  sr=0$ for some $s=s(r)\in S\}=\ga$. 

\begin{proposition}\label{A29Jun26}%\marginpar{A29Jun26}

Suppose that $S\in \Der_l(R, \ga)$, $T$ is multiplicative subset of $R$ that contains $S$ and  the image of $T$ in the ring $S^{-1}R$ consists of units. Then $T\in \Der_l(R, \ga)$ and $S^{-1}R\simeq T^{-1}R$.

\end{proposition}

\begin{proof} (i) {\em The set $T$ is a left Ore set of $R$}: We have to show that for given elements $t\in T$ and $r\in R$, there exist elements $t'\in T$ and $r'\in R$ such that $t'r=r't$. By the assumption, the element $t\in T$ is a unit in the ring $S^{-1}R$. Therefore, $rt^{-1}=s^{-1}a$ for  some elements $s\in S$ and $a\in R$, and so $s_1sr=s_1at\in R$ for some element $s_1\in S$. Now, it suffice to take $t'=s_1s\in S\subseteq T$ and $r'=s_1a\in R$.

(ii) $\ass_l(T)=\ga$: The inclusion $T\supseteq S$ implies the inclusion $\ga':= \ass_l(T)\supseteq \ass_l(S)=\ga$. Suppose that $a\in \ga'$. Then $ta=0$ for some element $t\in T$, and  the equality  yields  the equality $\frac{t}{1} \frac{a}{1} =0\in S^{-1}R$. By the assumption, the element  $\frac{t}{1}$ is   a unit in the ring $S^{-1}R$, and so $\frac{a}{1} =0$, i.e. $a\in \ga$. Therefore, $\ga'\subseteq \ga$, and the statement (ii) follows.

(iii) $T\in \Den_l(R, \ga)$: In view of the statements (i) and (ii), in order to show that 
$T\in \Den_l(R, \ga)$ it remains to prove  that $tr=0$ for   some elements $t\in T$ and $r\in R$ implies  $r\in \ga$. The equality $tr=0$ yields the equality $\frac{t}{1} \frac{r}{1} =0\in S^{-1}R$. The element  $\frac{t}{1}$ is   a unit in the ring $S^{-1}R$, and so $\frac{r}{1} =0$, i.e. $r\in \ga$, as required.

(iv) $S^{-1}R\simeq T^{-1}R$: Since $S\subseteq T$ and $S,T\in \Den_l(R,\ga)$, the map 
$$ \phi : S^{-1}R\ra T^{-1}R, \;\; s^{-1}r\mapsto s^{-1}r$$
is a ring monomorphism. Since every element $t$ is a unit in the ring $ S^{-1}R$ the map $\phi$ is an isomorphism.  
\end{proof}

A ring $R$  is called a {\bf semiprime ring} if it has no nonzero nilpotent ideals. The (left)  {\bf uniform dimension} of 
$R$   is the supremum of the number of nonzero left ideals in a direct sum of uniform left ideals. The ring $R$ is called a {\bf left  Goldie ring} if it  satisfies the ascending chain condition on left  annihilators and has finite left uniform dimension. Goldie's Theorem \cite{Goldie-PLMS-1960}  is a criterion for a ring to have a semisimple  left quotient ring (earlier,  characterizations were given, by Goldie \cite{Goldie-PLMS-1958} and Lesieur and Croisot \cite{Lesieur-Croisot-1959}, of left orders in a simple Artinian ring).

\begin{theorem}[\textbf{Goldie's Theorem, \cite{Goldie-PLMS-1960}}]\label{GoldieThm}%%\marginpar{GoldieThm}
%{\bf (Goldie's Theorem, \cite{Goldie-PLMS-1960})}
 A ring has a semisimple left quotient ring iff it is a semiprime  left Goldie ring.
 % that satisfies the ascending chain condition on left annihilators and does not contain infinite direct sums of nonzero left ideals.
\end{theorem}

Every left Noetherian ring is a left Goldie ring, but the converse is not true. Goldie's conditions are weaker and more general than the Noetherian condition. In \cite{Crit-S-Simp-lQuot}, four new criteria are given  for a ring to have a semisimple left quotient ring, they are based on completely new ideas.

Corollary \ref{a29Jun26} gives a sufficient condition for a ring to be a semiprime left Goldie ring. Corollary \ref{a29Jun26} is used in the proof of Proposition \ref{a29Jun26}.

\begin{corollary}\label{a29Jun26}%\marginpar{a29Jun26}

Suppose that $S\in \Der_l(R,0)$ and $S^{-1}R$ is a semisimple Artinian ring. Then $\CC_R\in \Den_l(R, 0)$,  $Q_l(R)\simeq S^{-1}R$ and the ring $R$ is a semiprime left Goldie ring.

\end{corollary}

\begin{proof} Since $S\in \Der_l(R,0)$,   $R\subseteq S^{-1}R$. The corollary follows from Proposition \ref{A29Jun26} where $T=\CC_R$. In more detail, the set $\CC_R$ is a multiplicative set with $\ass_l(\CC_R)=0=\ass_l(S)$ and $S\subseteq \CC_R$ (since  $S\in \Der_l(R,0)$). In order to apply Proposition \ref{A29Jun26}, it remains to show that the set $\CC_R$ consists of units of the localization ring $S^{-1}R$. Clearly, the ring $R$ is an essential left $R$-submodule of the ring $S^{-1}R$. Therefore, for each element $c\in \CC_R$, the $R$-homomorphism 
$$ \cdot c: S^{-1}R\ra S^{-1}R, \;\; a\mapsto ac
$$
is a monomorphism. Every monomorphism of an Artinian module is an isomorphism. Therefore, the $R$-monomorphism $\cdot c$ is a bijection, and so $bc=1$ for some element $b\in S^{-1}R$.  This implies the map 
$$ c\cdot : S^{-1}R\ra S^{-1}R, \;\; a\mapsto ca
$$
is has zero kernel (if $cu=0$ for some element $u\in S^{-1}R$ then $0=b0=bcu=1\cdot u = u$). The map $c\cdot$ is a monomorphism of  the right Artinian $S^{-1}R$-module $S^{-1}R$. Therefore, the map $c\cdot$ is an isomorphism. This implies that $cd=1$ for some element $d\in S^{-1}R$. Hence, the element $c$ has right and left inverses in $S^{-1}R$ , i.e. it is an invertible element of the ring $S^{-1}R$, as required. 
\end{proof}

%{\bf Endomorphism ring $\End_R(\bigoplus_{i=1}^nM_i)$ of a direct sum of Euclidean modules.}

{\bf Endomorphism rings and localizations.}
 Let $R$ be a ring, $I$ be a left ideal of $R$, $\ga :=\ann_R(R/I)\subseteq I$, $\bR := R/\ga=\bR/ \bI$ where $\bI := I/\ga$. In particular, the $\bR$-module $R/I$ is a faithful module. Suppose that $S\in \Den_l(\bR)$.  Then  $S^{-1}(\bR/ \bI)\simeq S^{-1}\bR/ S^{-1}\bI$ and 
$$
\End_{S^{-1}\bR}(S^{-1}(\bR/ \bI))=\mI_{S^{-1}\bR}(S^{-1}\bI):=\Big\{ s^{-1}r\in S^{-1}\bR \mid S^{-1}\bI s^{-1}r\subseteq S^{-1}\bI \Big\}/S^{-1}\bI.
$$  
We write endomorphisms on the side opposite to the scalars, namely on the right; however, when dealing with a single endomorphism, we place it on the left to simplify the notation.

\begin{proposition}\label{a28Jun26}%\marginpar{a28Jun26}
Let $R$ be a ring, $I$ be a left ideal of $R$, $\ga :=\lann_R(R/I)\subseteq I$,  $\bR := R/\ga$  and $\bI := I/\ga$. Notice that the $\bR$-modules $R/I$ and $ \bR/ \bI$ are  isomorphic. 
 Suppose that  $S\in \Den_l(\bR)$ and  the $\bR$-module $\bR/\bI$ is $S$-torsion-free and  $S^{-1}\bI s^{-1}= S^{-1}\bI$ ($\Leftrightarrow$ $S^{-1}\bI s= S^{-1}\bI$)  for all elements $s\in S$. Then:
 \begin{enumerate}

\item    $\End_{S^{-1}\bR}(S^{-1}(\bR/ \bI))\simeq S^{-1}\End_{\bR}(\bR/ \bI)$ and $\End_{\bR}(\bR/ \bI)\subseteq S^{-1}\End_{\bR}(\bR/ \bI)$. 

\item Suppose,  in addition, that the $S^{-1}\bR$-module $S^{-1}(\bR/ \bI)\simeq S^{-1}\bR/ S^{-1}\bI$ is a finitely generated  semisimple  module. Then:
\begin{enumerate}

\item  The endomorphism ring 
$\End_{S^{-1}\bR}(S^{-1}(\bR/ \bI))\simeq S^{-1}\End_{\bR}(\bR/ \bI)$ is a semisimple Artinian ring 
%(it is a direct product $\prod_{i=1}^sM_{n_i}(D_i)$ of matrix rings over division rings $D_i$)
  that contains the endomorphism ring $\End_{\bR}(\bR/ \bI)$.

\item  The  classical left quotient ring  $Q_l(\End_{\bR}(\bR/ \bI))$ of $\End_{\bR}(\bR/ \bI)$ is isomorphic to  the ring  $\End_{S^{-1}\bR}(S^{-1}(\bR/ \bI))$.

\item The ring  $\End_{\bR}(\bR/ \bI)$ is a semiprime left Goldie ring.

\end{enumerate}

\item Suppose,  in addition, that the $S^{-1}\bR$-module $S^{-1}(\bR/ \bI)\simeq S^{-1}\bR/ S^{-1}\bI$ is a simple module. Then  the endomorphism ring 
$\End_{S^{-1}\bR}(S^{-1}(\bR/ \bI))\simeq S^{-1}\End_{\bR}(\bR/ \bI)$ is a division ring that contains the endomorphism ring $\End_{\bR}(\bR/ \bI)$ which is a domain of left uniform dimension 1 (i.e. all nonzero left ideals of it are essential left $\End_{\bR}(\bR/ \bI)$-modules) and the left classical quotient ring of $\End_{\bR}(\bR/ \bI)$, $Q_l(\End_{\bR}(\bR/ \bI))$, is isomorphic to the division ring  $\End_{S^{-1}\bR}(S^{-1}(\bR/ \bI))$.

\end{enumerate}
\end{proposition}

\begin{proof} 1. By the assumption, the $\bR$-module $\bR/\bI$ is $S$-torsion-free. Therefore, the $\bR$-module $\bR/\bI$ is an essential $\bR$-submodule of $S^{-1}(\bR/ \bI)$ and 
$$
\{0\}=\tor_S(\bR /\bI)=(S^{-1}\bI\cap \bR)/\bI\;\; (\Leftrightarrow\;\; S^{-1}\bI\cap \bR = \bI) .
$$ 
Given an element $s^{-1}r\in S^{-1}\bR$, where $s\in S$ and $r\in \bR$, then   $$
s^{-1}r\in \mI_{S^{-1}\bR}(S^{-1}\bI)=\End_{S^{-1}\bR}\Big(S^{-1}(\bR/ \bI)\Big)
$$
 iff $S^{-1}\bI s^{-1}r\subseteq S^{-1}\bI$ 
 iff $S^{-1}\bI r\subseteq S^{-1}\bI$
 (since $S^{-1}\bI s^{-1}=S^{-1}\bI$) 
iff $\bI r\subseteq S^{-1}\bI\cap \bR=\bI$   %(the inclusion follows from the assumption that $\bI s^{-1}\subseteq S^{-1}\bI$  for all elements $s\in S$)
 iff $$r+\bI\in \mI_{\bR}(\bI)/\bI=\End_{\bR}(\bR/ \bI).$$ Therefore, $\End_{S^{-1}\bR}(S^{-1}(\bR/ \bI))\simeq S^{-1}\End_{\bR}(\bR/ \bI)$. In particular, for all elements $s\in S$, 
 $$\cdot s^{-1}\in \End_{S^{-1}\bR}(S^{-1}(\bR/ \bI)).
 $$
Since  the $\bR$-module $\bR/\bI$ is $S$-torsion-free, the ring  $\End_{\bR}(\bR/\bI)=\mI(\bI)/\bI\subseteq \bR/\bI$ is also  $S$-torsion-free. Therefore,    the ring $\End_{\bR}(\bR/ \bI)$ is a subring of $S^{-1}\End_{\bR}(\bR/ \bI)$.

 2.(a) Since the $S^{-1}\bR$-module $S^{-1}(\bR/ \bI)$ is a finitely generated  semisimple module, it is a direct sum $\bigoplus_{i=1}^sV_i^{n_i}$ of simple non-isomorphic $S^{-1}\bR$-modules $V_i$ with multiplicities $n_i\geq 1$. Therefore, 
 the endomorphism ring 
$$
\End_{S^{-1}\bR}(S^{-1}(\bR/ \bI))\simeq \End_{S^{-1}\bR}\bigg(\bigoplus_{i=1}^sV_i^{n_i}\bigg)\simeq \prod_{i=1}^sM_{n_i}(\End_{S^{-1}\bR}(V_i))\simeq \prod_{i=1}^sM_{n_i}(D_i)
$$
 is  a direct product of matrix rings over division rings $D_i=\End_{S^{-1}(\bR/ \bI)}(V_i)$ (which is a semisimple Artinian ring).   By statement 1, the ring $\End_{S^{-1}\bR}(S^{-1}(\bR/ \bI))$    contains the endomorphism ring $\End_{\bR}(\bR/ \bI)$. 
 
(b) The statement (b) follows from the statement (a) and Corollary \ref{a29Jun26}.
 
(c) The statement (c) follows from the statement (b) and Goldie's Theorem.  

 3. Statement 3 is a particular case of statement 2. 
\end{proof}

Suppose that $S$ is a multiplicative subset of $R$ and $I$ is a non-empty subset of $R$. We say that the set $E$ satisfies the {\bf left $I$-Ore condition} or is a  {\bf left  $I$-Ore set} of $R$ if for any elements $i\in I$ and $s\in S$ there exist elements $s_1,s_2\in S$ and $i_1\in I$ such that $s_2(s_1i-i_1s)=0$. If $I=R$ then every left $I$-Ore set of $R$ is a left Ore set of $R$.

\begin{lemma}\label{a1Aug26}%\marginpar{a1Aug26}
Let $R$ be a ring, $I$ be a left ideal of $R$ and  $S\in \Der_l(R, \ga)$. Then   the following statements are equivalent:
\begin{enumerate}

\item The multiplicative set $S$ is a left  $I$-Ore set.

\item For any elements $i\in I$ and $s\in S$ there exist elements $s_1\in S$ and $i_1\in I$ such that $s_1i\equiv  i_1s\mod \ga$.

\item For all elements $s\in S$, $S^{-1}Is=S^{-1}I$.

\item For all elements $s\in S$, $S^{-1}Is^{-1}=S^{-1}I$.

\end{enumerate}
\end{lemma}

\begin{proof} The equivalences $(1\Leftrightarrow 2)$ and $(3\Leftrightarrow 4)$ are obvious. 

$(2\Leftrightarrow 4)$ Statement 2 holds iff for any elements $i\in I$ and $s\in S$ there exist elements $s_1\in S$ and $i_1\in I$ such that $is^{-1}=s_1^{-1}i_1$ iff for all elements $s\in S$, $S^{-1}Is^{-1}=S^{-1}I$.
\end{proof}

{\bf Licence.} For the purpose of open access, the author has applied a Creative Commons Attribution (CC BY) licence to any Author Accepted Manuscript version arising from this submission.

{\bf Declaration of interests.} The author declares that they have no known competing financial interests or personal relationships that could have appeared to influence the work reported in this paper.

%{\bf Disclosure statement.} No potential conflict of interest was reported by the author.

{\bf Data availability statement.} Data sharing not applicable – no new data generated.

{\bf Funding.} This research received no specific grant from any funding agency in the public, commercial, or not-for-profit sectors.

\small{

\begin{minipage}[t]{0.5\textwidth}
V. V. Bavula 

School of Mathematical and Physical Sciences, 

University of Sheffield

Hicks Building

Sheffield S3 7RH

UK

email: v.bavula@sheffield.ac.uk 
\end{minipage}

\end{document}